\documentclass[10pt, leqno]{amsart}

\everydisplay\expandafter{\the\everydisplay\selectfont}
\allowdisplaybreaks

\usepackage{
amsmath,
amssymb,
tikz,
fancyhdr,
enumitem,
setspace,
mathdots,
multirow,
mathabx,
upgreek,
ytableau,
tikz-cd,
makecell,
lscape
}
\usepackage{multirow}

\usetikzlibrary{arrows,calc,matrix,shapes,positioning,decorations.pathreplacing}

\usepackage[colorlinks=true, pdfstartview=FitV,linkcolor=blue,citecolor=blue,urlcolor=blue]{hyperref}

\usepackage[vcentermath,enableskew]{youngtab}
\usepackage{mathrsfs}
\usepackage[center,small,sc]{caption}
\usepackage{mathtools}
\def\multiset#1#2{\ensuremath{\left(\kern-.3em\left(\genfrac{}{}{0pt}{}{#1}{#2}\right)\kern-.3em\right)}}

\usepackage{multicol}
\usepackage{soul}
\usepackage[all]{xy}

\usepackage[colorinlistoftodos]{todonotes}
\usepackage{makecell}

\theoremstyle{plain}
\newtheorem{theorem}{Theorem}[section]
\newtheorem{lemma}[theorem]{Lemma}
\newtheorem{proposition}[theorem]{Proposition}
\newtheorem{corollary}[theorem]{Corollary}
\newtheorem{definition-Proposition}[theorem]{Definition-Proposition}

\theoremstyle{definition}
\newtheorem{definition}[theorem]{Definition}
\newtheorem{example}[theorem]{Example}
\newtheorem{algorithm}[theorem]{Algorithm}

\newtheorem{convention}[theorem]{Convention}
\newtheorem{remark}[theorem]{Remark}
\newtheorem{definition-proposition}[theorem]{Definition-Proposition}
\numberwithin{equation}{section} \numberwithin{figure}{section}
\numberwithin{table}{section}

\newenvironment{red}{\relax\color{red}}{\relax}
\newenvironment{blue}{\relax\color{blue}}{\hspace*{.5ex}\relax}

\newenvironment{magenta}{\relax\color{magenta}}{\hspace*{.5ex}\relax}

\newcommand{\nc}{\newcommand}

\def\seteq{\mathbin{:=}}
\nc{\matr}[2]{\left( \begin{matrix} #1 \\ #2 \end{matrix} \right)}
\nc{\matl}[2]{\left\langle \begin{matrix} #1 \\ #2 \end{matrix} \right\rangle_q}
\nc{\floor}[1]{\left\lfloor #1 \right\rfloor}
\nc{\inn}[1]{\left\langle #1 \right\rangle}
\nc{\udots}{\text{\reflectbox{$\ddots$}}}
\nc{\ra}{\rightarrow}
\nc{\sqact}{\sqbullet}

\nc{\bema}{\begin{magenta}}
\nc{\ema}{\end{magenta}}
\nc{\ber}{\begin{red}}
\nc{\er}{\end{red}}
\nc{\beb}{\begin{blue}}
\nc{\eb}{\end{blue}}

\nc{\blo}{ {\color{blue} 0} }
\nc{\redo}{{\color{red} 0}}

\nc{\Sym}{\mathfrak{S}}
\nc{\cmA}{\mathsf{A}}
\nc{\Afin}{\mathsf{C}}
\nc{\ev}{\mathsf{ev}}
\nc{\mx}{\mathrm{max}}
\nc{\fks}{\mathfrak{s}}
\nc{\La}{\Lambda}
\def\g{\mathfrak g}
\def\h{\mathfrak h}
\def\Z{\mathbb Z}
\def\C{\mathbb C}
\def\Q{\mathbb Q}
\nc{\cma}{\mathrm{a}}
\nc{\congN}{{\mathsf{N}}}
\nc{\stv}{\mathrm{e}}
\nc{\adj}{\mathrm{adj}}
\nc{\beads}{\mathsf{B}}
\nc{\necks}{\mathsf{N}}
\nc{\wordset}{\mathcal{W}}
\nc{\sfm}{\mathsf{m}}
\nc{\triT}{T}
\nc{\odd}{\mathrm{odd}}
\nc{\even}{\mathrm{even}}
\nc{\ttriT}{\widetilde{T}}
\nc{\DR}{{\tt{DR}}}
\nc{\bM}{\mathbf{M}}
\nc{\bnu}{\boldsymbol{\nu}}
\nc{\bm}{\mathbf{m}}
\nc{\bw}{\mathbf{w}}
\nc{\bx}{\mathbf{x}}
\nc{\scrN}{\mathscr{N}}
\nc{\scrS}{\mathscr{S}}
\nc{\upc}{\xi}

\nc{\svv}{ \textbf{\textit{s}}}
\nc{\svs}{ \textbf{\textit{S}}}
\nc{\wtsvv}{\widetilde{\svv}}
\nc{\whsvv}{\widehat{\svv}}

\nc{\oLa}{\overline{\Lambda}}
\nc{\oP}{\overline{P}}
\nc{\oQ}{\overline{Q}}

\nc{\affsr}{\alpha}
\nc{\finsr}{\overline{\alpha}}
\nc{\finfw}{\varpi}
\nc{\clfw}{\mathscr{F}}

\nc{\redr}{{\mathrm{red}_\congN}}
\nc{\dwt}{P^+}
\nc{\ellcov}{\ell \mathcal{C}_{{\rm af}} \cap (\overline{\Lambda}+\overline{Q})}
\nc{\Pclp}{P_{\mathrm{cl}}^+}
\nc{\pclp}[1]{P_{\mathrm{cl},#1}^+}
\nc{\bwm}{\mathbf{w}(\mathbf{m};d;\bnu)}
\nc{\DRpclp}{\DR(\pclp{\ell})}
\nc{\invA}{\widetilde{\Afin}}
\nc{\evS}{\ev_{{}_\svs}}
\nc{\evSp}{\ev_{{}_{\svs'}}}
\nc{\evfks}[1]{\ev_{\fks^{(#1)}}}
\nc{\invevfks}{\ev^{-1}_{\fks}}
\nc{\tpclp}{\tau \cdot \pclp}
\nc{\whfksone}{\widehat{\fks}^{(1)}}
\nc{\whfkstwo}{\widehat{\fks}^{(2)}}

\nc{\siggen}{\sigma}
\nc{\upsig}{\upsigma}

\nc{\funcMtoW}{\widehat{\Uppsi}}
\nc{\funcPcltoZ}{\phi_{\clfw}}

\newcommand*\mycirc[1]{%
   \begin{tikzpicture}
     \node[draw,circle,inner sep=0pt] {#1};
   \end{tikzpicture}}

\newcommand*\mydoublecirc[1]{%
   \begin{tikzpicture}
    \node[draw,circle,inner sep=0pt] {#1};
	\node[draw,circle,inner sep=0.6pt] {#1};
   \end{tikzpicture}}

\nc{\yh}[1]{\todo[size=\tiny,color=blue!20]{#1 \\ \hfill --- Young-Hun}}
\nc{\YH}[1]{\todo[size=\tiny,inline,color=blue!20]{#1
\\ \hfill --- Young-Hun}}

\nc{\sejin}[1]{\todo[size=\tiny,color=red!20]{#1 \\ \hfill --- Se-jin}}
\nc{\Sejin}[1]{\todo[size=\tiny,inline,color=red!20]{#1
\\ \hfill --- Se-jin}}

\tikzset{tab/.style={matrix of math nodes,column sep=-.4, row
sep=-.4,text height=8pt,text width=8pt,align=center}}

\title[Cyclic sieving phenomenon on dominant maximal weights]
{Cyclic sieving phenomenon on dominant maximal weights over affine Kac-Moody algebras}
\author[Y.-H. Kim]{Young-Hun Kim}
\thanks{The research of Y.-H. Kim was supported by the National Research Foundation of Korea (NRF) Grant funded by the Korean Government (NRF-2018R1D1A1B07051048).}
\address{Department of Mathematics, Sogang University, Seoul 121-742, Republic of Korea}
\email{yhkim14@sogang.ac.kr}

\author[S.-j.~Oh]{Se-jin Oh}
\address{Ewha Womans University Seoul, 52 Ewhayeodae-gil, Daehyeon-dong, Seodaemun-gu, Seoul, South Korea}
\thanks{ S.-j.\ Oh was supported by the Ministry of Education of the Republic of Korea and the National Research Foundation of Korea (NRF-2019R1A2C4069647).}
\email{sejin092@gmail.com}
\urladdr{https://sites.google.com/site/mathsejinoh/}

\author[Y.-T.Oh]{Young-Tak Oh}
\thanks{The research of Y.-T. Oh was supported by the National Research Foundation of Korea (NRF) Grant funded by the Korean Government (NRF-2018R1D1A1B07051048).}
\address{Department of Mathematics, Sogang University, Seoul 121-742, Republic of Korea \& Korea Institute for Advanced Study, Seoul 02455, Republic of Korea}
\email{ytoh@sogang.ac.kr}

\date{\today}

\begin{document}

\maketitle

\begin{abstract}
We construct a (bi)cyclic sieving phenomenon on the union of dominant maximal weights for level $\ell$ highest weight modules over an affine Kac-Moody algebra with exactly one highest weight being taken for each equivalence class, in a way not depending on types, ranks and levels. 
In order to do that, we introduce $\svs$-evaluation on the set of dominant maximal weights for each highest modules, and
generalize Sagan's action in \cite{Sagan1} by considering the datum on each affine Kac-Moody algebra. As consequences, we obtain closed and recursive formulae 
for cardinality of the number of  dominant maximal weights for every highest weight module and observe level-rank duality on the cardinalities. 
\end{abstract}


\section*{Introduction}
Kac-Moody algebras were independently introduced by Kac \cite{Kac0} and Moody \cite{M}.
Among them, affine Kac-Moody algebras
have been particularly extensively studied for their beautiful representation theory as well as for their remarkable connections to
other areas such as mathematical physics, number theory, combinatorics, and so on.
Nevertheless, many basic questions are still unresolved.
For instance the behaviour of weight multiplicities and combinatorial features of dominant maximal weights are not fully understood (see \cite[Introduction]{KLO}).

Throughout this paper, $\g$ denotes an affine Kac-Moody algebra and $V(\La)$ the irreducible highest weight module with highest weight $\La \in P^+$,
where $\dwt$ denotes the set of dominant integral weights.
Due to Kac \cite{Kac}, all weights of $V(\La)$ are given by the disjoint union of $\delta$-strings attached to maximal weights and
every maximal weight is conjugate to a unique dominant maximal weight under Weyl group action.
So it would be quite natural to expect that better understanding of dominant maximal weights
makes a considerable contribution towards the study of representation theory of affine Kac-Moody algebras.

In \cite{Kac}, Kac established lots of fundamental properties concerned with ${\rm wt}(\La)$, the set of weights of $V(\La)$, using the orthogonal projection  $\bar{ \ } \colon \h^* \to \h_0^*$.
In particular, he showed that $\mx^+(\Lambda)$, the set of dominant maximal weights, is in bijection with $\ell \mathcal{C}_{{\rm af}} \cap (\overline{\Lambda}+\overline{Q})$
under this projection, thus it is finite. Here $\ell$ denotes the level of $\La$.
However, in the best knowledge of the authors, approachable combinatorial models, cardinality formulae and
structure on $\mx^+(\Lambda)$'s  have not been available up to now except for limited cases,
which motivates the present paper.

In 2014, Jayne and Misra \cite{JM} published noteworthy results about $\mx^+(\La)$ in $A_n^{(1)}$-case.
They give an explicitly parametrization of
$\mx^+((\ell - 1) \Lambda_0 + \Lambda_i)$ in terms of paths 
for $0 \le i \le n$ and $\ell \ge 2$,
and present the following conjecture:
\begin{align}\label{eq: conj}
|\mx^+(\ell\Lambda_0)| = \dfrac{1}{(n+1)+\ell} \sum_{d|(n+1,\ell)} \varphi(d) \matr{((n+1)+\ell)/d}{\ell/d},
\end{align}
where $\varphi$ is Euler's phi function. Notably this number gives the celebrated Catalan number when $\ell = n$. 
Soon after, this conjecture turned out to be affirmative in \cite{TW}. The proof therein largely depends on
Sagan's congruence on $q$-binomial coefficients \cite[Theorem 2.2]{Sagan1}.

The main purpose of this paper is to investigate $\mx^+(\Lambda)$ by constructing bijections with several combinatorial models and
a (bi)cyclic sieving phenomenon on the combinatorial models.  As applications, we can obtain closed 
formulae of  $\mx^+(\Lambda)$ for all affine types, and observe interesting symmetries by considering  $\mx^+(\Lambda)$ for all ranks and levels. 
\smallskip 

Set
\[\Pclp \seteq \dwt /  \Z \updelta  \quad \text{ and } \quad
\pclp{\ell} := \dwt_\ell / \Z \updelta  \quad \text{for } \ell \in \mathbb Z_{\ge 0}, \]
where $\dwt_\ell$  denotes the set of level $\ell$ dominant integral weights and $\updelta$ denotes the canonical null root of $\g$. 

Given a nonnegative integer $\ell$, we only consider classical dominant integral weights, that is, $\La$ in $\pclp{\ell}$
because there is a natural bijection between $\mx^+(\La)$ and $\mx^+(\La+k\delta)$ for every $k \in \Z$.
We begin with the observation that the set $\ellcov$ can be embedded into $\pclp{\ell}$ via the map
\begin{equation}\label{eq: embedding}
\begin{tikzcd}
&\iota_\Lambda: \ellcov \arrow[r]                          & \pclp{\ell}                                     \\[-4ex]
& \sum_{i=1}^n m_i \finfw_i \arrow[r, maps to] & m_0\Lambda_0 + \sum_{i=1}^n m_i\Lambda_i,
\end{tikzcd}
\end{equation}
where $Q$ denotes the root lattice, $\finfw_i:={\overline \La_i}$ and $m_0 = \ell - \sum_{i=1}^n a^\vee_i m_i$.

We then define an equivalence relation $\sim$ on $\pclp{\ell}$ by
$\Lambda \sim \Lambda' $ if and only if $\iota_\Lambda = \iota_{\Lambda'}$, equivalently $\ellcov = \ell \mathcal{C}_{{\rm af}} \cap (\overline{\Lambda'}+\overline{Q})$ (see Lemma~\ref{lem: sim equiv}).
By definition, if $\Lambda \sim \Lambda'$, then $|\mx^+(\La)|=|\mx^+(\La')|$. 
Note that this equivalence relation is defined in \cite{BFS} in a slightly different form.  We should remark that, in \cite{BFS}, the authors mainly investigated a membership condition of weights for highest weight module $V(\Lambda)$  modulo a certain lattice, while
we investigate $|\mx^+(\La)|$ and structures on the union of $\mx^+(\La)$'s. 

Under the relation $\sim$, it turns out that the image of $\iota_\La$ coincides with the equivalence class of $\Lambda$.
We provide a complete set of pairwise inequivalent representatives of the distinguished form $(\ell-1)\La_0 + \La_i$, denoted by ${\DR}(\pclp{\ell})$.
For instances, in case where $\g=A_{n}^{(1)}$, we have ${\DR}(\pclp{\ell})=\{(\ell-1)\La_0 + \La_i \ | \  0 \le i \le n\}$ and  in case where 
$\g=E_{6}^{(1)}$, we have ${\DR}(\pclp{\ell})=\{(\ell-1)\La_0 + \La_i \ | \   i=0,1,6\}$ (see Table \ref{table: PRep}). It follows that
$$    \bigsqcup_{ \Lambda \in  \DR (\pclp{\ell})} \pclp{\ell}(\Lambda)  = \pclp{\ell},$$ 
where $\pclp{\ell}(\Lambda)$ denotes the equivalence class of $\Lambda$ under $\sim$. It should be noticed that $|\pclp{\ell}(\Lambda)|=|\mx^+(\Lambda)|$.  

From this we derive a very significant consequence that the number of all equivalence classes is given by $\congN \seteq [\oP:\oQ]$, where 
$\oP/\oQ$ is isomorphic to the \emph{fundamental group} of the root system of $\g_0$ except for $\g=A^{(2)}_{2n}$  (see Table~\ref{table: oP/oQ}). Here $\g_0$ denotes the subalgebra of $\g$ which is of finite type.

Next, we introduce a new statistic $\evS$, called the \emph{$\svs$-evaluation}, on $\pclp{\ell}$.
Here $\svs$ is a certain set, called {\it a root sieving set}, which is characterized by a  minimal generating set of
the  $\Z_\congN$-kernel of the transpose of Cartan matrix associated $\g_0$ (see Convention~\ref{conv: Sieving} for details). 
In more detail, for all affine Kac-Moody algebras except for $D_n^{(1)}~(n\equiv_2 0)$, $\svs$ consists of a single element $(s_1, \ldots, s_n)$ and
$$\evS\left( \sum_{0 \le i \le n} m_i \Lambda_i\right):=\sum_{1 \le i \le n} s_im_i \qquad
 \text{ for } \Lambda =\sum_{0 \le i \le n}  m_i \Lambda_i.$$
In case where $\g=D_n^{(1)}~(n\equiv_2 0)$, we have $\svs=\{\svv^{(1)} = (0,0,\ldots,0,2,2), \svv^{(2)}= (2,0,2,0,\ldots,2,0,2,0)\}$.
For the $\svs$-evaluation of this type, see~\eqref{eq: S-evaluation}. 
Finally, exploiting this statistic, we characterize the equivalence class of $\La \in {\DR}(\pclp{\ell})$ 
in terms of $\svs$-evaluation (Theorem \ref{thm: eq rel and sieving}).

Quite interestingly, the \emph{$\svs$-evaluation} on $\pclp{\ell}$ leads us to construct a (bi)cyclic sieving phenomenon on it.
The {\it cyclic sieving phenomenon}, introduced by Reiner-Stanton-White in \cite{RSW}, are generalized and developed in various aspects including combinatorics and
representation theory (see \cite{ARR,BRS,EF,R,Sagan2} for examples).

Let us briefly recall the cyclic sieving phenomenon. 
Let $X$ be  a finite set, with an action of a cyclic group $C$ of order $m$, and $X(q) $ a polynomial in $q$ with nonnegative integer coefficients.
For $d \in \Z_{>0}$, let $\omega_d$ be a $d$th primitive root of the unity.
We say that $(X, C, X(q))$ exhibits the cyclic sieving phenomenon if, for all $g\in C$, we have
$|X^g| = X(\omega_{o(g)}),$
where $o(g)$ is the order of $g$ and $X^g$ is the fixed point set under the action of $g$.

Let us explain our initial motivation. It was shown in \cite[Theorem 1.1]{RSW} that $\left(\matr{[0,n]}{\ell}, C_{n+1}, \left[ \begin{matrix} n+\ell \\ \ell \end{matrix} \right]_q   \right) $ exhibits the cyclic sieving phenomenon. Here $\matr{[0,n]}{\ell}$ denotes the set of all $\ell$-multisets on $\{0,1,\ldots,n\}$, $C_{n+1}$ a fixed cyclic group of order ${n+1}$, and $\left[ \begin{matrix} n+\ell \\ \ell \end{matrix} \right]_q $ the $q$-binomial coefficient of $\matr{n+\ell}{\ell}$. We identify $\matr{[0,n]}{\ell}$ with $\pclp{\ell}$ in $A^{(1)}_{n}$-type as $C_{n+1}$-sets and let
$$
\pclp{\ell}(q) \seteq \left[ \begin{matrix} n+\ell \\ \ell \end{matrix} \right]_q.
$$
Then we observe that the generating function of $\pclp{\ell}(q)$ $(\ell \ge 0)$ can be expressed in terms of the root sieving set $\svs =\{(s_1,s_2,\ldots,s_n) = (1,2,\ldots,n)\}$ and the canonical center $c = h_0+h_1+h_2+ \ldots+h_n=\sum_{i=0}^n a_i^\vee h_i$ as follows:
\begin{align}\label{eq: q-binon series}
\sum_{\ell \ge 0} \pclp{\ell}(q)t^\ell  \seteq  \sum_{\ell \ge 0}  \left[ \begin{matrix} n+\ell \\ \ell \end{matrix} \right]_q  t^\ell 
= \displaystyle \prod_{0 \le i \le n} \dfrac{1}{1-q^{i}t^1}
=\displaystyle \prod_{0 \le i \le n} \dfrac{1}{1-q^{s_i}t^{a_i^\vee}},
\end{align}
where $s_0$ is set to be $0$. From this product identity it follows that $\pclp{\ell}(q) = \sum_{\Lambda \in \pclp{\ell}} q^{\evS(\Lambda)}$.
Furthermore, since $C_{n+1}$ is isomorphic to $\oP/\oQ$, we conclude that the triple $\left(\pclp{\ell}, \oP/\oQ, \pclp{\ell}(q) \right)$ also exhibits the cyclic sieving phenomenon.
 
 Then it is natural to ask whether 
there exists a triple for other affine Kac-Moody algebras \emph{exhibiting} the cyclic sieving phenomenon or not.
Canonically, one can construct the triple in uniform way for all affine Kac-Moody algebras as follows:  
We first take $\pclp{\ell}$ as the underlying set. 
Second, writing the canonical center as $c = \sum_{i=0}^n a_i^\vee h_i$,  
we take $\pclp{\ell}(q)$ from the following geometric series (by mimicking the $A_n^{(1)}$-case): 
 \begin{align*}
\begin{cases}
\sum_{\ell \ge 0} \pclp{\ell}(q)t^\ell \seteq \displaystyle \prod_{0 \le i \le n} \dfrac{1}{1-q^{s_i}t^{a_i^\vee}}, & \text{ if  $\g$ is not of type $D^{(1)}_{n}$ for even $n$}, \\
\sum_{\ell \ge 0} \pclp{\ell}(q_1,q_2)t^\ell \seteq \displaystyle \prod_{0 \le i \le n} \dfrac{1}{1-q_1^{\fks^{(1)}_i}q_2^{\fks^{(2)}_i}t^{a_i^\vee}}
& \text{ if  $\g$ is of type $D^{(1)}_{n}$  for even $n$},
\end{cases}
\end{align*}  
where $s_0$ is set to be $0$ (see~\eqref{eq: ge series not Dn1} and~\eqref{eq: ge series Dn1}).
Then we have  
$$
\begin{cases}
\pclp{\ell}(q) = \sum_{\Lambda \in \pclp{\ell}} q^{\evS(\Lambda)}   & \text{ if  $\g$ is not of type $D^{(1)}_{n}$  for even $n$}, \\
\pclp{\ell}(q_1,q_2) = \sum_{\Lambda \in \pclp{\ell}} q_1^{\ev_{\fks^{(1)}}(\Lambda)}q_2^{\ev_{\fks^{(2)}}(\Lambda)} 
& \text{ if  $\g$ is of type $D^{(1)}_{n}$  for even $n$}.
\end{cases}
$$
Finally, take $\oP/\oQ$ as the (bi)cyclic group, which completes the triple:
\begin{align}\label{eq: (bi)CSP triple}
(\pclp{\ell}, \oP/\oQ,  \pclp{\ell}(q)) \qquad \text{ (resp. } (\pclp{\ell}, \oP/\oQ,  \pclp{\ell}(q_1,q_2)) ).
\end{align}
We assign an appropriate $\oP/\oQ$-action on $\pclp{\ell}$ (see~\eqref{eq: CSP action} and~\eqref{eq: C2 times C2 action}),
and prove that the triple
exhibits the (bi)cyclic sieving phenomenon, which can be understood as a natural generalization of the cyclic sieving triple $\left( \matr{[0,n]}{\ell}, C_{n+1}, \left[ \begin{matrix} n+\ell \\ \ell \end{matrix} \right]_q   \right) $ in aspect of affine Kac-Moody algebras. 

For the proof, we employ the following strategy. 
For each divisor $d$ of $\congN$, we introduce a set $\bM_\ell(rd,d;\bnu,\bnu')$ equipped with a $C_d$-action obtained by generalizing Sagan's action on $(0,1)$-words in \cite{Sagan1}. Here, $r,\bnu,\bnu'$ are chosen so that $\bM_\ell(rd,d;\bnu,\bnu')$ can be identified with $\pclp{\ell}$ by permuting indices properly. Then we show that $|\bM_\ell(rd,d;\bnu,\bnu')^{C_d}| = \left|\left(\pclp{\ell}\right)^{g}\right|$ for all $g \in \oP/\oQ$ of order $d$.
We end the proof by showing $$|\bM_\ell(rd,d;\bnu,\bnu')^{C_d}| = \pclp{\ell}(\zeta_{\congN}^{\congN/d}).$$

From the above sieving phenomena, we derive closed formulae for $|\mx^+(\Lambda)|$ for all $\La\in \pclp{\ell}$ and for affine Kac-Moody algebras of arbitrary type. 
For the classical types, they are explicitly written as a sum of binomial coefficients (see Section~\ref{subsec: The number of dom max wts}).  For instance, in case where $A_n^{(1)}$ type, we obtain 
\begin{align}\label{eq: explicit formula for |mx| intro}
|\mx^+((\ell-1)\Lambda_0 + \Lambda_i)| =
\sum_{d \mid (n+1,\ell,i)} \dfrac{d}{(n+1)+\ell} \sum_{d'|(\frac{n+1}{d},\frac{\ell}{d})} \mu(d') \matr{((n+1)+\ell)/dd'}{\ell/dd'},
\end{align}
which is a vast generalization of~\eqref{eq: conj} (see also Theorem~\ref{thm: number of dom wt A}).  

Let us view $\{|\mx^+(\Lambda)|\}_{n,\ell}$ as a sequence expressed in terms of $n$ and $\ell$. Exploiting our closed formulae, we can also derive recursive formulae for $|\mx^+(\Lambda)|$ (except for type $A_n^{(1)}$) and their corresponding triangular arrays. It is quite interesting to observe that several triangular arrays are already known in different contexts. For example, when $\g$ is of affine $C$-type, our triangular arrays are known as \emph{Lozani\'{c}'s triangle} and its Pascal complement 
(see Subsection~\ref{subsubsec: inductive}).
Also, the triangular array for twisted affine  even $A$-type is Pascal triangle with duplicated diagonals 
(see Appendix~\ref{AppendiX: Recursive}).

Going further, we observe interesting interrelations among the triangular arrays of various affine Kac-Moody algebras (see Appendix~\ref{AppendiX: Recursive}). Surprisingly, 
all triangular arrays for classical affine type except for untwisted affine $C$-type can be constructed by {\it boundary conditions} and the triangular array of twisted affine  even $A$-type. Similarly, the triangular arrays for untwisted affine $C$-type can be constructed by boundary conditions and Pascal triangle. Considering that the triangular array of twisted affine even $A$-type can be obtained from Pascal triangle, we can conclude that all triangular arrays for classical affine types can be obtained from boundary conditions and Pascal triangle only.

As another byproduct of our closed formulae, we observe a symmetry which appears as level and rank are switched in a certain way.
For instance, 
 if $(n+1,\ell,i) = (\ell, n+1,j)$ for some $0 \le i \le n$ and $0 \le j \le \ell-1$, then
\begin{align*}
  \left|\mx_{A_n^{(1)}}^+((\ell-1)\Lambda_0 + \Lambda_i)\right| = \left|\mx_{A_{\ell-1}^{(1)}}^+(n\Lambda_0 + \Lambda_j)\right|.
\end{align*} 
This symmetry is \emph{compatible} with the classical level-rank duality for $A_n^{(1)}$ studied by Frenkel in \cite{F} (see Subsection~\ref{sec: LR-duality}). 
With the closed formulae of $\mx^+(\Lambda)$ in terms of binomial coefficients, we can observe interesting symmetries 
for all classical affine types. For instances, we have
\begin{align*}
\begin{cases}
\left|\mx_{B_n^{(1)}}^+(\ell\Lambda_0)\right| = \left|\mx_{B_{(\ell-1)/2}^{(1)}}^+((2n+1)\Lambda_0)\right|, & \text{ if $\ell$ is odd}, \\
\left|\mx_{B_n^{(1)}}^+((\ell-1)\Lambda_0+\Lambda_n)\right| = \left|\mx_{B_{\ell/2-1}^{(1)}}^+((2n+1)\Lambda_0+\Lambda_{\ell/2 -1 })\right|
& \text{ if $\ell$ is even},
\end{cases}
\end{align*}
by exchanging $n$ with $(\ell-1)/2$, and $n$ with $\ell/2-1$, respectively, since
\begin{align*}
\left|\mx_{B_n^{(1)}}^+(\ell\Lambda_0)\right| = \binom{n+\floor{\frac{\ell}{2}}}{n}+\binom{n+\floor{\frac{\ell-1}{2}}}{n},   \
\left|\mx_{B_n^{(1)}}^+((\ell-1)\Lambda_0 + \Lambda_n)\right| = \binom{n+\floor{\frac{\ell-1}{2}}}{n}+\binom{n+\floor{\frac{\ell}{2}}-1}{n}.
\end{align*}
 
This paper is organized as follows. In Section~\ref{Sec: pre}, we introduce necessary notations and backgrounds for affine Kac-Moody 
algebras, highest weight modules and classical results on dominant maximal weights. In Section~\ref{sec: Par param},
we define an equivalence relation $\sim$ on $\pclp{\ell}$
satisfying that the equivalence class of $\Lambda \in \pclp{\ell}$ has the same cardinality with $\mx^+(\Lambda)$. Then we provide the set ${\DR}(\pclp{\ell})$ of distinguished representatives, and characterize all equivalence classes in terms of $\svs$-evaluation with our sieving set $\svs$. 
In Section~\ref{sec: Sagan action and general}, we generalize Sagan's action with consideration on the result in Section~\ref{sec: Par param} and prove that the generalized action gives cyclic action on $\pclp{\ell}$
indeed.  In Section~\ref{Sec: CSP1}, we prove that our triple for affine Kac-Moody algebras except $D_{n}^{(1)}$ for even $n$
exhibits the cyclic sieving phenomenon. In Section~\ref{Sec: CSP2}, we prove the triple for $D^{(1)}_{n}$ for even $n$ exhibits bicyclic sieving phenomenon. In Section~\ref{Sec: application}, we derive closed formulae, recursive formulae, and level-rank duality for the sets of dominant maximal weights from the cyclic sieving phenomenon.
In Appendix~\ref{AppendiX: Recursive} and~\ref{AppendiX: Level-rank duality},
we list all triangular arrays and level-rank duality for affine Kac-Moody algebras, not dealt with in Section~\ref{Sec: application}.

\section{Preliminaries} \label{Sec: pre}
Let $I=\{ 0,1,...,n \}$ be an index set. An \emph{affine Cartan datum} $(\cmA,P,\Pi,P^{\vee},\Pi^{\vee})$ consists of the following quintuple:
\begin{enumerate}
    \item[({\rm a})] a matrix $\cmA=(\cma_{ij})_{i,j \in I}$ of corank $1$, called an {\it affine Cartan matrix} satisfying that, for $i,j \in I$,
    $$ ({\rm i}) \ \cma_{ii}=2 , \quad ({\rm ii}) \ \cma_{ij}  \in \Z_{\le 0} \text{ for $i \ne j \in I$,} \quad  ({\rm iii}) \ \cma_{ij}=0 \text{ if } \cma_{ji}=0,$$
    \item[({\rm c})] a free abelian group $P=\bigoplus_{i=0}^{n} \Z \Lambda_i \oplus \Z \updelta$, called the \emph{weight lattice},
    \item[({\rm e})] a linearly independent set $\Pi = \{ \affsr_i \mid i\in I \} \subset P$, called the set of {\it simple roots},
    \item[({\rm b})] a free abelian group $P^{\vee}= \mathrm{Hom}(P,\Z) $, called the \emph{coweight lattice},
    \item[({\rm d})] a linearly independent set $\Pi^{\vee} = \{ h_i \mid i\in I\} \subset P^{\vee}$, called the set of {\it simple coroots},
\end{enumerate}
subject to the condition
\[ \inn{h_i, \affsr_j} = \cma_{ij} \text{ and } \langle h_j, \Lambda_i \rangle = \delta_{ij} \text{ for all $i,j\in I$}. \]
We call $\Lambda_i$ the \emph{$i$th fundamental weight} and set $\h:= \Q \otimes_\Z P^{\vee}$. Let
\[\updelta= a_0 \affsr_0 + a_1 \affsr_1 + \cdots + a_n \affsr_n \]
be the \emph{null root} and
\[c=a^\vee_0 h_0 + a^\vee_1 h_1+ \cdots + a^\vee_n h_n  \]
be the \emph{canonical central element}.
%
We say that a weight $\Lambda \in P$ is of {\it level} $\ell$ if
$$ \langle c,\Lambda \rangle=\ell. $$
Then we have $a_i^\vee = \langle c,\Lambda_i \rangle$.

Note that there exists a non-degenerate symmetric bilinear form $( \,\cdot \, | \, \cdot \, )$ on $\h^*$ (\cite[(6.2.2)]{Kac}) such that
\begin{align}\label{eq: pairing}
(\Lambda_0|\Lambda_0) =0, \quad (\affsr_i|\affsr_j) = a_i^{\vee}a_i^{-1} \cma_{ij},  \quad (\affsr_i|\Lambda_0) =\delta_{i,0}a_0^{-1} \quad \text{for } i,j\in I,
\end{align}
and
\begin{align*}
 (\updelta |  \lambda)=\langle c,\lambda \rangle \quad \text{for } \ \lambda \in P.
\end{align*}

Set $\dwt \seteq \{\Lambda \in P \mid  \langle h_i,\Lambda \rangle \in \Z_{\ge 0},\  i \in I \}$. The elements of $\dwt$ are called the \emph{dominant integral weights}.
Also, for a nonnegative integer $\ell$, we set
$$\dwt_{\ell} \seteq \{\Lambda \in \dwt \mid \inn{c,\Lambda} = \ell\}.$$
We call the free abelian group $Q\seteq\bigoplus_{i \in I} \Z\affsr_i$ the \emph{root lattice} and set
$Q_{+}\seteq\bigoplus_{i \in I}\Z_{\ge 0}\affsr_i$.

\begin{definition}
    The {\em affine Kac-Moody algebra} $\g$ associated with an affine Cartan datum $(\cmA,P,\Pi,P^{\vee},\Pi^{\vee})$
    is the Lie algebra over $\Q$ generated by $e_i,f_i \ ( i \in I)$ and $h \in P^{\vee}$
    subject to the following defining relations:
    \begin{enumerate}
        \item $[h,h']=0, \ [h,e_i]=\inn{h,\affsr_{i}}e_i, \ \ [h,f_i]=-\inn{h,\affsr_{i}}f_i \ $ for $ h,h' \in P^{\vee}$,
        \item $[e_i,f_j]= \delta_{i,j}h_i\ $ for $ i,j \in I$,
        \item $({\rm ad}\; e_i)^{1-\cma_{ij}}(e_j)=({\rm ad}\; f_i)^{1-\cma_{ij}}(f_j)=0$ if $i \ne j$.
    \end{enumerate}
\end{definition}

Let $\g_0$ be the subalgebra of $\g$ generated by the $e_i$ and $f_i$ with $i \in I_0 \seteq I \setminus \{ 0 \}$.
Then $\g_0$ is the Lie algebra associated to the Cartan matrix $\Afin$ obtained from $\cmA$ by deleting the $0$th row and the $0$th column.
For a finite dimensional Lie algebra $\mathsf{g}$, let $\mathsf{g}^\dagger$ be the Lie algebra whose Cartan matrix is the transpose of the Cartan matrix of $\mathsf{g}$.
The following table lists $\g_0$ for each affine Kac-Moody algebra $\g$:
\begin{table}[h]
    \centering
    { \arraycolsep=1.6pt\def\arraystretch{1.5}
        \begin{tabular}{c||c|c|c|c|c|c|c|c|c|c|c}
            $\g$ &  $   A_{n}^{(1)}$ & $B_n^{(1)}, D_{n+1}^{(2)} $ & $C_n^{(1)}, A_{2n-1}^{(2)}, A_{2n}^{(2)}$ & $D^{(1)}_{n}$ & $E_6^{(1)}$ & $E_7^{(1)}$ & $E_8^{(1)}$ & $F_4^{(1)}$ & $E_6^{(2)}$ & $G_2^{(1)}$ & $D_4^{(3)}$ \\ \hline
            $\g_0$  & $A_{n} $ & $B_n$ & $C_n$ & $D_n$  & $E_6$ & $E_7$ & $E_8$ & $F_4$ & $F_4^\dagger$ & $G_2$ & $G_2^\dagger$
        \end{tabular}
    }\\[1.5ex]
    \caption{$\g_0$ for each type}
    \protect\label{og}
\end{table}

A $\g$-module $V$  is called a \emph{weight module} if it admits a \emph{weight space decomposition}
\[V= \bigoplus_{\mu \in P} V_{\mu},\quad
\text{where }V_{\mu}=\{ v \in V | \ h \cdot v =\inn{h,\mu} v \text{ for all }  h \in P^{\vee} \}.\]
If $V_\mu \neq 0$, $\mu$ is called a \emph{weight} of $V$ and $V_\mu$ is the \emph{weight space} attached to $\mu$.
A weight module $V$ over $\g$ is called {\it integrable} if $e_i$ and $f_i$ ($i \in I$) act locally nilpotent on $V$.

\begin{definition}
    The category $\mathcal{O}_{{\rm int}}$ consists of integrable $\g$-modules $V$ satisfying the following conditions:
    \begin{enumerate}
        \item $V$ admits a weight space decomposition $V=\bigoplus_{\mu \in P} V_{\mu}$ with $\dim V_\mu < \infty$ for all weights $\mu$.
        \item There exists a finite number of elements $\lambda_1,\ldots,\lambda_s \in P$ such that
        $$ {\rm wt}(V) \subset D(\lambda_1) \cup \cdots \cup D(\lambda_s).$$
        Here ${\rm wt}(V): =\{ \mu \in P \ | \ V_\mu \ne 0 \}$ and $D(\lambda):= \{ \lambda - \alpha \ | \ \alpha \in Q_+ \}$.
    \end{enumerate}
\end{definition}

It is well-known that $\mathcal{O}_{{\rm int}}$ is a semisimple tensor category such that every irreducible
objects is isomorphic to {\it the highest weight module $V(\Lambda)$} ($\Lambda \in P^+$).

A weight $\mu$ of $V(\Lambda)$ is {\it maximal} if $\mu+\updelta \not \in \mathrm{wt}(V(\Lambda))$ and the set of all maximal weights of $V(\Lambda)$ is denoted by $\mx_\g(\Lambda)$.

\begin{proposition}[{\cite[(12.6.1)]{Kac}}] For each $\Lambda \in P^+$, we have
    $$ {\rm wt}(V(\Lambda)) = \bigsqcup_{\mu \in \mx_{\g}(\Lambda)} \{ \mu-s\updelta \ | \ s \in \Z_{\ge 0} \}.$$
\end{proposition}
Denote by $\mx_\g^+(\Lambda)$ the set of all dominant maximal weights of $V(\Lambda)$, thus, $$\mx_\g^+(\Lambda) = \mx_\g(\Lambda) \cap P^+.$$
We will omit the subscript $\g$  for simplicity if there is no danger of confusion.
It is well-known that
$$  \mx(\Lambda) = W \cdot \mx^+(\Lambda),  \quad \text{ where $W$ is the Weyl group of $\g$}.$$

Let $\h_0$ be the vector space spanned by $\{ h_i \ | \ i \in I_0\}$.
Recall the {\it orthogonal projection} $\bar{ \ } \colon \h^* \to \h_0^*$, which is introduced in (\cite[(6.2.7)]{Kac}), by
$$  \mu \longmapsto \overline{\mu} =\mu - \inn{c,\mu}\Lambda_0 - (\mu | \Lambda_0) \updelta.$$
Let $\oQ$ (resp. $\oP$) be the image of $Q$ (resp. $P$) under this map.
We also use $\inn{\ , \ }$ and $( \ | \ )$ to denote bilinear forms for $\g_0$ since they can be obtained by restricting $\inn{\ , \ }$ and $( \ | \ )$ to $\h_0 \times \h_0^*$ and $\h_0^* \times \h_0^*$ (via $\bar{ \ }$) respectively.

Define
\begin{align}\label{eq: KCaf}
\ell \mathcal{C}_{{\rm af}} \seteq \{ \mu \in \h_0^* \ | \ \inn{h_i,\mu} \ge 0 \text{ for $i \in I_0$}, \ (\mu | \uptheta) \le \ell \} \quad \text{ where }
\uptheta \seteq \updelta-a_0\affsr_0.
\end{align}

\begin{proposition}[{\cite[Proposition 12.6]{Kac}}] \label{prop: Kac bijection}
    The map $\mu \longmapsto \overline{\mu}$ defines a bijection from $\mx^+(\Lambda)$ onto $\ell \mathcal{C}_{{\rm af}} \cap (\overline{\Lambda}+\overline{Q})$
    where $\Lambda$ is of level $\ell$. In particular, the set $\mx^+(\Lambda)$ is finite and described as follows:
    \begin{align}\label{eq: description}
    \mx^+(\Lambda) = \{ \lambda \in P^+ \ | \  \lambda \le \Lambda \text{ and } \Lambda-\lambda-\updelta \not\in Q_+ \}.
    \end{align}
\end{proposition}

For reader's understanding, let us collect notations required to develop our arguments.
\begin{itemize}
\item[$\diamond$] For $(n+1)$-tuples $\gamma = (\gamma_0, \gamma_1,\ldots, \gamma_n)$ and $\gamma' = (\gamma'_0, \gamma'_1,\ldots, \gamma'_{n'})$ of integers, $0 \le a \le b \le n$, we set
\begin{itemize}
\item[--] $\gamma_{[a,b]} \seteq (\gamma_a,\ldots,\gamma_b)$, $ \gamma_{\le a} \seteq \gamma_{[0,a]}$,  and $\gamma_{\ge b} \seteq \gamma_{[b,n]}$.
\item[--] $\gamma*\gamma' \seteq (\gamma_0,\gamma_1,\ldots,\gamma_n,\gamma'_0, \gamma'_1,\ldots, \gamma'_{n'})$.
\end{itemize}
\item[$\diamond$] For words $\bw = w_1w_2\cdots w_n$ and $\bw' = w'_1w'_2\cdots w'_n$, we set $\bw * \bw' \seteq w_1w_2\cdots w_n w'_1w'_2\cdots w'_n$.
\item[$\diamond$] For a nonnegative integer $m$ and a positive integer $k$, we denote by $m^k$ the sequence $\underbracket{m, m, \ldots, m}_{\text{$k$-times}}$.
\item[$\diamond$] Let $k$ be a positive integer.
\begin{itemize}
\item[--] For $m,m' \in \Z$, we write $m\equiv_k m'$ if $k$ divides $m-m'$, and $m \not\equiv_k m'$ otherwise.
\item[--] For $\mathbf{m}= (m_1,m_2,\ldots, m_n), \mathbf{m}' = (m_1',m_2',\ldots, m_n') \in \Z^n$, we write $\mathbf{m} \equiv_k \mathbf{m}'$ if $m_i \equiv_k m_i'$ for all $i = 1,2,\ldots,n$.
\end{itemize}
\item[$\diamond$] For a matrix $M$, we denote by $M_{(i)}$ the $i$th row of $M$ and by $M^{(i)}$ the $i$th column of $M$.
\item[$\diamond$] For an invertible matrix $M$, we denote by $\widetilde{M}$ the inverse matrix of $M$.
\item[$\diamond$] For a commutative ring $R$ with the unity and a positive integer $n$,
the \emph{dot product} on $R^n$ denotes the map $\bullet : R^n \times R^n \ra R$ defined by
\[(x_1,x_2,\ldots,x_n) \bullet (y_1,y_2,\ldots,y_n) = \sum_{1 \le i \le n} x_i y_i.\]
\item[$\diamond$] For a statement $P$, $\delta(P)$ is defined to be $1$ if $P$ is true and $0$ if $P$ is false.
\end{itemize}

\section{Sets in bijection with $\mx^+(\Lambda)$}\label{sec: Par param}

In this section, all affine Kac-Moody algebras will be affine Kac-Moody algebras other than $A_{2n}^{(2)}$. In fact, we exclude the case $A_{2n}^{(2)}$ for the simplicity of our statements. All the notations and terminologies in the previous section will be used without change.

Choose an arbitrary element $\Lambda \in \dwt_\ell$.
The purpose of this section is to understand a combinatorial structure of $\mx^+(\Lambda)$ by investigating sets in bijection with
$\mx^+(\Lambda)$ which are induced from certain restrictions of the orthogonal projection  $\bar{ \ } \colon \h^* \to \h_0^*$.

As seen in Proposition \ref{prop: Kac bijection}, the set  $\ell \mathcal{C}_{{\rm af}} \cap (\overline{\Lambda}+\overline{Q})$
plays  a key role in the study of $\mx^+(\Lambda)$. 
Hereafter we will assume that $\Lambda$ is of the form $\sum_{0 \le i \le n} p_i \Lambda_i$
because
\[ \ellcov = \ell \mathcal{C}_{{\rm af}} \cap (\overline{\Lambda + k \updelta}+\overline{Q}) \quad \text{for all $k \in \Z$}. \]
Set
\[ \Pclp \seteq \dwt / \Z \updelta.\]
We identify $\Pclp$ with $\sum_{0 \le i \le n} \Z_{\ge 0} \Lambda_i$ in the obvious manner.
As a set, $\Pclp$ coincides with the set of classical dominant integral weights arising in the context of quantum affine Lie algebra $U_q'(\g)$ (for details, see \cite{HK}).
We also set
\[ \pclp{\ell} := \dwt_\ell / \Z\updelta, \]
which is identified with $\dwt_\ell \cap \sum_{0 \le i \le n} \Z_{\ge 0} \Lambda_i$.

\subsection{Description of $\ell \mathcal{C}_{{\rm af}} \cap (\overline{\Lambda}+\overline{Q})$}\label{subsec: equiv rel}

As mentioned in the above, $\g$ denotes an affine Kac-Moody algebra other than $A_{2n}^{(2)}$.

Set
\begin{align*}
&\Uppi_0 \seteq \{\finsr_i \mid i \in I_0\} \quad \text{(the set of simple roots of $\g_0$)},\\
&\upvarpi \seteq \{\finfw_i \mid i \in I_0\} \quad \text{(the set of fundamental dominant weights of $\g_0$).}
\end{align*}
Both $\Uppi_0$ and $\upvarpi$ are bases for $\Q \upvarpi$,  and the transition matrix $[\mathrm{Id}]_{\Uppi_0}^{\upvarpi}$ is equal to Cartan matrix $\Afin$ of $\g_0$.
For reader's understanding, let us recall that
\[
\finsr_0 = -\sum_{1 \le i \le n}a_i \finsr_i,  \qquad
\oLa_i =\begin{cases}
\finfw_i & \text{if $i \neq 0$},\\
0 & \text{if $i = 0$},
\end{cases} \]
and
\[\finsr_i = \sum_{1 \le j \le n} a_{ji} \finfw_j ,  \qquad  \finfw_i = \sum_{1 \le j \le n} d_{ji} \finsr_j  \quad \text{ ($i \in I_0$)}.\]
Here $\Afin =(a_{ij})_{i,j \in I_0}$ and $\invA = (d_{ij})_{i,j \in I_0}$ is the inverse of $\Afin$.

Choose any element $\Lambda = \sum_{0 \le i \le n} p_i \Lambda_i \in \pclp{\ell}$, which will be fixed throughout this subsection.
Then we have
\begin{align}
\ell \mathcal{C}_{{\rm af}} \cap (\overline{\Lambda}+\overline{Q}) \nonumber
& = \left\{ \oLa + \sum_{0 \le j \le n} k_j \finsr_j \; \middle|\; k_j \in \Z, \langle h_i,\oLa + \sum_{0 \le j \le n} k_j \finsr_j\rangle \geq 0 \ (i \in I_0), \  \left(\oLa + \sum_{0 \le j \le n} k_j \finsr_j\, \middle|\, \uptheta \right) \leq \ell \right\} \nonumber \\
& = \left\{ \finsr \seteq  \oLa + \sum_{1 \le j \le n} x_j \finsr_j \; \middle|\;
\begin{array}{l}
{\rm (i)}   \ \mathbf{x} \seteq (x_1, x_2, \ldots, x_n)^t \in \Z^n \\
{\rm (ii)}  \  \langle h_i, \finsr \rangle \ge 0  ~(i \in I_0) \\
{\rm (iii)} \  \left( \finsr \; \middle| \; \sum_{1\le i \le n} a_i \finsr_i \right) \le \ell
\end{array} \right\}, \label{eq: equi condition1}
\end{align}
where the second equality can be obtained by substituting $x_j$ for $k_j - k_0 a_j$ for $j \in I_0$.
Since
\begin{align*}
\left\langle h_i, \oLa + \sum_{1 \le j \le n} x_j \finsr_j \right\rangle & = p_i + \sum_{1 \le j \le n} x_j \cma_{ij} = p_i + \Afin_{(i)} \mathbf{x},
\end{align*}
one can see that the condition (ii) is satisfied if and only if $\Afin_{(i)} \mathbf{x} \ge -p_i$.
For the condition (iii), notice that (see~\eqref{eq: pairing})
\begin{equation} \label{eq: equation}
\begin{aligned}
\left( \oLa + \sum_{1 \le j \le n} x_j \finsr_j \; \middle| \; \sum_{1 \le i \le n} a_i \finsr_i \right)
= \sum_{1 \le i,j \le n} p_j a_i (\finfw_j | \finsr_i) + \sum_{1 \le i,j \le n} x_j a_i (\finsr_j|\finsr_i) =  \sum_{1 \le i \le n} (a_i^\vee p_i + a_i^\vee \Afin_{(i)}\mathbf{x}).
\end{aligned}
\end{equation}
Since $\ell = \langle c,\Lambda \rangle = \sum_{0 \le i \le n} a_i^\vee p_i$, this computation implies that
the condition (iii) in~\eqref{eq: equi condition1} is satisfied if and only if $\sum_{1 \le i \le n} a_i^\vee \Afin_{(i)}\mathbf{x} \le a_0^\vee p_0$.
As a consequence, $\ell \mathcal{C}_{{\rm af}} \cap (\overline{\Lambda}+\overline{Q})$ can be written as
    \begin{align} \label{eq: desired set 1}
    \left\{ \oLa + \sum_{1 \le j \le n} x_j\finsr_j \; \middle| \;
    \begin{array}{l}
    {\rm (i)} \ \mathbf{x} :=  (x_1,x_2,\ldots,x_{n})^t \in \Z^n \\[0.7ex]
    {\rm (ii)} \ -p_i  \le \Afin_{(i)} \mathbf{x}  \text{ for $i \in I_0$} \\[0.7ex]
    {\rm (iii)}  \ \sum_{1 \le i \le n} a^\vee_i \Afin_{(i)} \mathbf{x} \le a_0^\vee p_0. \end{array}  \right\}
    \end{align}
Finally, using the substitution $m_j \seteq \Afin_{(j)} \mathbf{x} + p_j$ for $j\in I_0$, we obtain the description of
$\ell \mathcal{C}_{{\rm af}} \cap (\overline{\Lambda}+\overline{Q})$ in terms of the basis $\upvarpi$.
\begin{proposition}\label{lem: ellCaf}
     Let $\Lambda=\sum_{0 \le i \le n} p_i\Lambda_i \in \dwt_\ell$.  Then we have
    \begin{align}\label{eqn:omx}
        \ell \mathcal{C}_{{\rm af}} \cap (\overline{\Lambda}+\overline{Q})
        & = \left\{\sum_{1 \le i \le n} m_i \finfw_i\ \; \middle| \;
        \begin{array}{l}
        {\rm (i)} \ \sum_{1 \le i \le n} (m_i - p_i) \invA^{(i)} \in \Z^n  \\ [0.5ex]
        {\rm (ii)} \  (m_1,m_2,\ldots,m_n)^t\in \Z_{\ge 0}^{n} \\[0.5ex]
        {\rm (iii)} \ \sum_{1 \le i \le n} a_i^\vee m_i \le \ell
        \end{array}
        \right\}.
        \end{align}
\end{proposition}

\begin{proof}
Let $\mathbf{x} :=  (x_1,x_2,\ldots,x_{n})^t \in \Z^n$. Note that
\begin{align*}
\oLa + \sum_{1 \le i \le n}x_i \finsr_i & = \sum_{1 \le i \le n} \left( p_i +  \sum_{1 \le j \le n} x_j\cma_{ij} \right)\finfw_i = \sum_{1 \le i \le n}\left( p_i + \Afin_{(i)} \mathbf{x} \right) \finfw_i.
\end{align*}
Set $m_j := \Afin_{(j)} \mathbf{x} + p_j$.
Since $\mathbf{x} = \sum_{1 \le i \le n} \invA^{(i)}(m_i - p_i)$, (i) of~\eqref{eq: desired set 1} is equivalent to (i) of~\eqref{eqn:omx}.
By direct calculation, one can see that $\Afin_{(j)} \mathbf{x}$ is an integer. Thus, by the definition of $m_i$, (ii) of~\eqref{eq: desired set 1} is equivalent to {\rm (ii)} of~\eqref{eqn:omx}. 
For the condition (iii), observe that
\begin{align*}
    \sum_{1 \le i \le n} a_i^\vee \Afin_{(i)} \mathbf{x} = \sum_{1 \le i \le n}  a_i^\vee   \Afin_{(i)} \left( \sum_{1 \le j \le n} \invA^{(j)} (m_j - p_j) \right) = \sum_{1 \le i \le n} a_i^\vee(m_i - p_i) \le a_0^\vee p_0.
\end{align*}
This tells us that $\sum_{1 \le i \le n} a^\vee_i \Afin_{(i)} \mathbf{x} \le a_0^\vee p_0$ if and only if $\sum_{1 \le i \le n} a_i^\vee m_i \le \ell$, as required.
\end{proof}
\begin{example}\label{eg: Seq Lambda}
        Let $\g$ be the affine Kac-Moody algebra of type $A_2^{(1)}$ and $\Lambda = 2\Lambda_0 + \Lambda_1$. In this case, $a^\vee_0 = a^\vee_1 = a^\vee_2 = 1$,
        and $\invA = \left[\begin{matrix}
        \frac{2}{3} & \frac{1}{3} \\[0.5ex]
        \frac{1}{3} & \frac{2}{3}
        \end{matrix}\right]$.
        Hence, by Proposition \ref{lem: ellCaf}, we have
        \begin{align*}
        3\mathcal{C}_{{\rm af}} \cap (\overline{\Lambda}+\overline{Q}) &= \left\{m_1\finfw_1 + m_2 \finfw_2 \;\middle| \;
        \begin{array}{l}
        {\rm (i)} \ (m_1 - 1)\left[\begin{matrix}
        \frac{2}{3} \\[0.5ex]
        \frac{1}{3}
        \end{matrix}\right] + m_2\left[\begin{matrix}
        \frac{1}{3} \\[0.5ex]
        \frac{2}{3}
        \end{matrix}\right] \in \Z^2 \\[0.5ex]
        {\rm (ii)} \ m_1, m_2 \in \Z_{\ge 0}\\[0.5ex]
        {\rm (iii)} \ m_1 + m_2 \le 3 \\
        \end{array}\right\}\\
        & = \left\{\finfw_1, 2\finfw_2, 2\finfw_1 + \finfw_2 \right\}.
        \end{align*}
\end{example}

\subsection{Equivalence relation on $\pclp{\ell}$}

Let $\Lambda \in \pclp{\ell}$.
Consider the map $\iota_\Lambda: \ellcov \ra \pclp{\ell}$ defined by
\[\iota_\Lambda\left( \sum_{1 \le i \le n} m_i \finfw_i \right) = m_0\Lambda_0 + \sum_{1 \le i \le n} m_i\Lambda_i,  \]
where
$$m_0 = \ell - \sum_{1 \le i \le n} a^\vee_i m_i.$$
This map is well-defined since all $m_i$'s are nonnegative integers for all $0 \le i \le n$ by Proposition~\ref{lem: ellCaf} and $\sum_{0 \le i \le n} m_i \Lambda_i$ has level $\ell$. In particular, it is injective.

We now define a relation $\sim$  on $\pclp{\ell}$, called the \emph{sieving equivalence relation}, by
\begin{align}\label{eq: sieving relation}
\Lambda \sim \Lambda' \quad \text{if and only if} \quad \ellcov = \ell \mathcal{C}_{{\rm af}} \cap (\overline{\Lambda'}+\overline{Q}).
\end{align}
It is easy to see that $\sim$ is indeed an equivalence relation on $\pclp{\ell}$. 
The following lemma is straightforward.

\begin{lemma}\label{lem: sim equiv}
For $\Lambda, \Lambda' \in \pclp{\ell}$, the following are equivalent.
\begin{enumerate}
\item[{\rm (1)}] $\Lambda \sim \Lambda'$.
\item[{\rm (2)}] $\iota_\Lambda = \iota_{\Lambda'}$.
\item[{\rm (3)}] $\oLa + \oQ = \overline{\Lambda'} + \oQ$.
\item[{\rm (4)}] $\Lambda' \in \mathrm{Im}(\iota_\Lambda)$.
\end{enumerate}
\end{lemma}

\begin{proof}
The equivalence of (1) and (2) is straightforward from the definition, and that
of  (1) and (3) follows from the fact that
either $(\oLa + \oQ) \cap (\overline{\Lambda'} + \oQ) = \emptyset$  or $\oLa + \oQ = \overline{\Lambda'} + \oQ$
because $(\oLa + \oQ)$ and $(\overline{\Lambda'} + \oQ)$ are translations of $\oQ$.

Next, let us show that  (2) implies (4).
Suppose that $\iota_\Lambda = \iota_{\Lambda'}$. Then $\overline{\Lambda'} \in \ell \mathcal{C}_{{\rm af}} \cap (\overline{\Lambda'}+\overline{Q}) = \ellcov$ and so $\iota_{\Lambda}(\overline{\Lambda'}) = \iota_{\Lambda'}(\overline{\Lambda'}) = \Lambda'$.

Finally, let us show that  (4) implies (3).
Assume that $\Lambda' = \sum_{0 \le i \le n} m_i' \Lambda_i \in \mathrm{Im}(\iota_{\Lambda})$.
Since $\Lambda' \in \pclp{\ell}$, this gives $\overline{\Lambda'} = \sum_{1 \le i \le n} m_i'\finfw_i \in \ellcov$. Therefore, $\oLa + \oQ = \overline{\Lambda'} + \oQ$.
This completes the proof.
\end{proof}

Let $P_0:=\mathbb Z \upvarpi$ be the weight lattice of $\g_0$ and $Q_0:=\mathbb Z \Uppi_0$ the root lattice of $\g_0$.
Then $P_0/Q_0$ is known to be a finite group, called the {\it fundamental group of $\Phi_0$ (the set of roots of $\g_0$).}
Its structure is well-known in the literature. For instance, see \cite{Humph}.
\begin{table}[h]
    \centering
    { \arraycolsep=1.6pt\def\arraystretch{1.5}
        \begin{tabular}{c||c|c|c|c|c|c|c|c|}
            $ \g_0 $ &  $   A_{n}$ & $D_{n}$& $E_6$& $E_7$ & $E_8$ & $B_n \overset{t}{\leftrightarrow} C_{n}$   & $F_4$ & $G_2$\\ \hline
            $P_0/Q_0$  & $\Z_{n+1}$& $
            \begin{array}{cl}
            \Z_4 & \text{if $n$ is odd,} \\ \Z_2 \times \Z_2 & \text{if $n$ is even}
            \end{array}$& $\Z_3$ & $\Z_2$ & $\{e\}$
            & $\Z_2$
            & $\{e\}$& $\{e\}$
        \end{tabular}
    }\\[1.5ex]
    \caption{Fundamental groups}
    \protect\label{table: oP/oQ}
\end{table}

It should be noticed that, except for $A_{2n}^{(2)}$ type, it holds that $\oQ=Q_0$ and $\oP=P_0$.
Lemma~\ref{lem: sim equiv} (3)  shows that
there are at most $|P_0 / Q_0|$ equivalence classes on $\pclp{\ell}$.
In the following, we provide a complete list of representatives of very simple form.
For each type, let us define a set $\DRpclp$, called the set of {\it distinguished representatives}, as in Table~\ref{table: PRep}.
\begin{table}[h]
    \centering
    { \arraycolsep=1.6pt\def\arraystretch{1.5}
        \begin{tabular}{c|c|c}
            Type & $\DRpclp$ & $|\DRpclp| \; ( = |P_0/Q_0|)$ \\ \hline
            $   A_{n}^{(1)} $   & $\{(\ell - 1)\Lambda_0 + \Lambda_i \mid i = 0,1,\ldots, n \}$ & $n+1$ \\ \hline
            $   B_n^{(1)}, D_{n+1}^{(2)}, E_7^{(1)} $   & $\{(\ell - 1)\Lambda_0 + \Lambda_i \mid i = 0,n \}$ & $2$  \\ \hline
            $C_n^{(1)}, A_{2n-1}^{(2)}$ & $\{ (\ell - 1)\Lambda_0 + \Lambda_i \mid i = 0,1\} $ & $2$ \\ \hline
            $D_n^{(1)}$ &  $\{(\ell - 1)\Lambda_0 + \Lambda_i \mid i = 0,1,n-1,n\}$  & $4$ \\ \hline
            $E_6^{(1)}$ &  $\{ (\ell - 1)\Lambda_0 + \Lambda_i \mid i = 0,1,6\}$& $3$ \\ \hline
            & \\[-15pt]
            \makecell{$F_4^{(1)}$, $E_6^{(2)}$, $G_2^{(1)}$,\\ $D_4^{(3)}$, $E_8^{(1)}$}    & $\{\ell \Lambda_0\}$& $1$
        \end{tabular}
    }\\[1.5ex]
    \caption{Distinguished representatives}
    \protect\label{table: PRep}
\end{table}
One can prove in a direct way the following lemma.

\begin{lemma}
$\DRpclp$ is a complete set of pairwise inequivalent representatives of $\pclp{\ell}/\sim$, the set of equivalence classes of $\pclp{\ell}$ under the sieving equivalence relation.
In particular, the number of equivalence classes is given by $|P_0/Q_0|$.\
\end{lemma}

\begin{proof}
Here we will deal with $A_n^{(1)}$ type only since other types can be verified in the exactly same manner.
Since the number of elements in $\DRpclp$ is equal to $n+1$, it suffices to show that every element is pairwise inequivalent, that is, it is enough to show that
\[
\overline{(\ell-1)\Lambda_0 + \Lambda_i} - \overline{(\ell-1)\Lambda_0 + \Lambda_j} = \oLa_i - \oLa_j \notin \oQ \; ( = Q_0).
\]
Equivalently, it suffices to show that
\[
[\oLa_i - \oLa_j]_{\Uppi_0} \notin \Z^n \quad (0 \le i < j \le n).
\]
Note that
\[
 [\oLa_i]_{\Uppi_0} =
\begin{cases}
    0 & \text{if $i=0$}, \\
    \invA^{(i)} & \text{if $i >0$}
\end{cases}
\]
and the first coordinate of $\invA^{(i)}$ is $1-i/(n+1)$. Since $0 \le i < j \le n$, the first coordinate of $[\oLa_i - \oLa_j]_{\Uppi_0}$ is not an integer, as required.
\end{proof}

For $\Lambda \in  \DRpclp$, let $ \pclp{\ell}(\Lambda)$ denote the equivalence class of $\Lambda$, i.e., $\pclp{\ell}(\Lambda):= \{ \Lambda' \in \pclp{\ell}  \; | \; \Lambda \sim \Lambda'\}$.
Then
$$\iota_\Lambda: \ellcov \to    \pclp{\ell}(\Lambda) $$
is bijective,
and its inverse is given by $\bar{ \ } |_{ \pclp{\ell}(\Lambda)}$, the restriction of  $\bar{ \ } \colon \h^* \to \h_0^*$  to $\pclp{\ell}(\Lambda)$.
Notice that if $\Lambda \nsim \Lambda'$, then $\left( \ellcov  \right) \cap  \left( \ell\mathcal{C}_{{\rm af}} \cap (\overline{\Lambda'}+\overline{Q}) \right)=\emptyset$, so we have bijections:
\begin{equation}\label{eq: bijection between ellcov and pclp}
\begin{tikzcd}
\displaystyle \bigsqcup_{\Lambda \in \DRpclp}\ellcov \arrow[r,"1-1"] & \displaystyle \bigsqcup_{\Lambda \in \DRpclp} \pclp{\ell} (\Lambda)=\pclp{\ell}  \arrow[l]
\end{tikzcd}
\end{equation}

\subsection{Equivalence classes}
For each $\Lambda \in  \DRpclp$, we give a simple description of the equivalence class $ \pclp{\ell}(\Lambda)$.
For this purpose, we recall the following elementary fact from linear algebra.

\begin{lemma}\label{lem: Zsubmodule}
Let $\beta = \{ \beta_1, \beta_2,\ldots, \beta_n\}$ and $\gamma = \{\gamma_1, \gamma_2, \ldots, \gamma_n\}$ be bases for $\Q^n$ such that $\Z \gamma \subseteq \Z \beta$.
Let $M = [\mathrm{Id}]_{\gamma}^{\beta}$ be the change of coordinate matrix that change $\gamma$-coordinates into $\beta$-coordinates.
Then for any $v \in \Z \beta$, it holds that
$v \in \Z \gamma$ if and only if $\widetilde{M}_{(i)} [v]_{\beta} \in \Z$ for all $i =1,2,\ldots,n$.
\end{lemma}

Choose an arbitrary element $x\in P_0$.
Lemma \ref{lem: Zsubmodule} tells us that $x \in Q_0$ if and only if $\invA_{(i)} [x]_{\upvarpi} \in \Z$ for all $i =1,2,\ldots,n$.
Let $\{\stv_1, \stv_2, \ldots ,\stv_n \}$ be the standard basis of $\Z^n$.
Since
$$\stv_j = \Afin_{(j)} \invA = \sum_{1 \le k \le n} \Afin_{j,k}\invA_{(k)} \quad (j \in I_0),$$
$\Z^n$ is obviously a submodule of the $\Z$-span of $\{ {\invA}_{(i)} \mid i \in I_0 \}$,
denoted by $\Z\{ {\invA}_{(i)} \mid i \in I_0 \}$.
In the same manner, $\Z^n$ is a submodule of $\Z\{ {\invA}^{(i)} \mid i \in I_0 \}$ and
\begin{align*}
\Z\{\invA^{(i)} \mid i \in I_0 \} / \Z^n \cong P_0/ Q_0 \quad \text{ (as abelian groups) }
\end{align*}
since $[\finfw_i]_{\Uppi_0} = \invA^{(i)}$, for all $i\in I_0$. Going further,
using Table \ref{table: oP/oQ}, we can deduce that
\begin{align}\label{eq: oP/oQ and inv of Cartan}
\Z \{ \invA_{(i)} \mid i \in I_0 \} / \Z^n  \cong P_0/Q_0 \quad \text{ (as abelian groups) }.
\end{align}

Recall that $ P_0/Q_0$ is a cyclic group unless $\g$ is of the type $D_n^{(1)}$ ($n$ is even).
It is not difficult to see that there is an index $i_1$ (resp. $j_1)$, which may not be unique, such that
\[\invA_{(i_1)} + \Z^n \quad \text{(resp. $\invA^{(j_1)} + \Z^n$ )} \]
is a generator  of $\Z \{ \invA_{(i)} \mid i \in I_0 \} / \Z^n$
(resp. $\Z \{ \invA^{(i)} \mid i \in I_0 \} / \Z^n$).

In a similar way, in case where $\g$ is of the type $D_n^{(1)}$ ($n$ is even),  one can see that
there is a set of indices  $\{i_1, i_2 \}$ (resp. $\{j_1,j_2\})$, which may not be unique, 
such that
$$\{\invA_{(i_1)} + \Z^n, \,\, \invA_{(i_2)} + \Z^n \} \quad  \text{ (resp. $\{\invA^{(j_1)} + \Z^n, \,\, \invA^{(j_2)} + \Z^n \}$ )}$$
is a generating set  of $\Z \{ \invA_{(i)} \mid i \in I_0 \} / \Z^n$
(resp. $\Z \{ \invA^{(i)} \mid i \in I_0 \} / \Z^n$).
For the convenience of computation, from now on, we fix these indices as in the table below:

\begin{table}[h]
    \centering
    { \arraycolsep=1.6pt\def\arraystretch{1.5}
        \begin{tabular}{c||c|c|c|c|c|c}
            $\g$ &  $A_{n}^{(1)},  D_n^{(1)} \text{($n$:odd)}, E_7^{(1)}$  & $B_n^{(1)}, D_{n+1}^{(2)}$ & $C_{n}^{(1)}, A_{2n-1}^{(2)}$ & $E_6^{(1)}$ & $D_n^{(1)} \text{($n$:even)}$ &  \text{Other types} \\ \hline
            $i_k$  & $i_1 = n$ & $i_1 = 1$ & $i_1 = n$ & $i_1 = 1$ & $i_1 = 1,~ i_2 = n$ & $i_1 = 1$ \\ \hline
            $j_k$ & $j_1 = n$ & $j_1 = n$ & $j_1 = 1$& $j_1 = 1$  &$j_1 = 1,~ j_2 = n$ & $j_1 = 1$
        \end{tabular}
    }\\[1.5ex]
    \caption{$i_k,j_k$ for each type}
    \protect\label{table: i_12,j_12}
\end{table}
The above discussion shows that ${\invA}_{(i)} [\oLa]_{\upvarpi} \in \Z$ for all $i  \in I_0$ if and only if ${\invA}_{(i_k)} [\oLa]_{\upvarpi} \in \Z$ for $k = 1$ or $k = 1,2$ (up to types).
When $\g$ is of the type $D_n^{(1)}$ ($n\equiv_2 0$), the order of the coset ${\invA}_{(i_k)}+ \Z^n$ $(k=1,2)$ in $\Z\{ \widetilde{\Afin}_{(i)}  \mid i \in I_0\} /\Z^n$ is given by $2$.
For the other types, the order of the coset ${\invA}_{(i_1)}$ is $|P_0 / Q_0|$.
For simplicity of notation, we set
\begin{align}\label{eq: congN}
\congN \seteq |P_0 / Q_0|
\end{align}
With this notation, we have the following characterization.

\begin{lemma}\label{characterization of Q}
Let $\g$ be an affine Kac-Moody algebra. For $x\in P_0$, we have
\begin{align}\label{eq: root and sieving 1}
x \in Q_0 \quad \text{if and only if} \quad \adj(\Afin)_{(i_k)} [x]_{\upvarpi} \equiv_{\congN}  0
\end{align}
for all $k = 1$ or $k = 1,2$ up to types. Here, $\adj(\Afin)$ denotes the classical adjoint of $\Afin$.
\end{lemma}

Consider the $\Z$-linear map given by left multiplication by  $\Afin^t$ 
$$L_{\Afin^t}: \Z^n \to \Z^n, \quad \mathbf{x} \mapsto \Afin^t \mathbf{x}.$$
From reduction modulo $\congN$
\[\redr: \Z \ra \Z_{\congN}, \quad a \mapsto a + \congN\Z \]
 we can induce a $\Z_\congN $-linear map defined by
$$L_{{\overline{\Afin}}^t}: (\Z_\congN)^n \to (\Z_\congN )^n, \quad \mathbf{x} \mapsto {\overline{\Afin}}^t \mathbf{x},$$
where ${\overline{\Afin}}$ is obtained from $\Afin$ respectively by reading entries modulo $\congN $.
We simply write $\ker (\overline{\Afin}^t)$ to denote the kernel of $L_{{\overline{\Afin}}^t}$.

Since $[\finsr_i]_{\upvarpi} = \Afin^{(i)}$  for all $i = 1,2,\ldots,n$,  by Lemma \ref{characterization of Q}, we deduce that $\adj({\overline \Afin})_{(i_k)} \in \ker (\overline{\Afin}^t)$ for all $k = 1$ or $k = 1,2$ up to types.

\begin{lemma}\label{basis of ker C transpose}
In the above setting, $\{\adj({\overline \Afin})_{(i_k)} : \text{$k = 1$ {\rm  or} $k = 1,2$ {\rm up to types}}\}$
is a minimal generating set of $\ker (\overline{\Afin}^t).$
\end{lemma}

\begin{proof}
To prove the assertion, it suffices to show that $\ker (\overline{\Afin}^t) \cong P_0/Q_0$ as abelian groups, equivalently
$\ker (\overline{\Afin}^t) \cong \Z\{{\invA}^t_{(i)} \mid i \in I_0 \} / \Z^n$
by~\eqref{eq: oP/oQ and inv of Cartan}.

Define $f:\Z\{{\invA}^t_{(i)} \mid i \in I_0 \} / \Z^n \ra \ker (\overline{\Afin}^t) $ as follows:
For $\mathbf{m}=(m_1,m_2,\ldots,m_{n})^t \in \Z^n$ (so, $\invA^t \cdot \mathbf{m} + \Z^n\in \Z\{{\invA}^t_{(i)} \mid i \in I_0 \} / \Z^n$),  we define
\[f\left( \invA^t \cdot \mathbf{m} + \Z^n \right) = \redr\left( \congN \invA^t \cdot \mathbf{m} \right)  \in (\Z_{\congN})^n, \]
Since $\Afin^t \left(  \congN  \invA^t \cdot \mathbf{m}  \right) \in (\congN\Z)^n$, $f$ is well-defined.
Also, by definition, $f$ is a group homomorphism.

Next, assume that for $\mathbf{m},\mathbf{m}' \in \Z^n$ ,
\[ f\left( \invA^t \cdot \mathbf{m} + \Z^n \right) = f\left( \invA^t \cdot \mathbf{m}' + \Z^n \right)  \in (\Z_{\congN})^n. \]
Then
\[\congN \left(   \invA^t \cdot (\mathbf{m} - \mathbf{m} ')    \right)\in (\congN\Z)^n, \]
which implies that $ \invA^t \cdot (\mathbf{m} - \mathbf{m} ')  \in  \Z^n$.
Hence $f$ is injective.

For the surjectivity, take any $\overline{\mathbf{x}} = (x_1,x_2,\ldots,x_n) \in \ker (\overline{\Afin}^t)$. 
Then $\mathbf{m} = \dfrac{1}{\congN}\Afin^t \cdot \overline{\mathbf{x}} \in \Z^n$ by the definition of $\ker (\overline{\Afin}^t)$. Thus we have
\[
f\left( \invA^t \cdot \mathbf{m} + \Z^n \right)= \overline{\mathbf{x}}. \qedhere
\]
\end{proof}

\begin{convention}
If there is a danger of confusion, we will use $\redr(\mathbf{x})$ and $\redr (\Afin)$ instead of $\mathbf{\overline{x}} $ and ${\overline{\Afin}}$ to emphasize the modulo $\congN$.
\end{convention}

\begin{theorem}\label{thm: map phi}
    Let $\g$ be an affine Kac-Moody algebra  
    of rank $n\in \Z_{>0}$ and $\ell \in \Z_{\ge 0}$.
For each $\Lambda \in  \DRpclp$, we have
    \begin{align} \label{eq: image of phi}
    \pclp{\ell}(\Lambda) = \left\{ \Lambda' \in \pclp{\ell} \; \middle| \;
    {\rm red_{\congN} ( [\overline{\Lambda'}]_{\upvarpi})}  \in {\rm red_{\congN} ([\oLa]_{\upvarpi}) } + \ker (\overline{\Afin}^t)^\perp
    \right\},
    \end{align}
where
\[ \ker (\overline{\Afin}^t)^\perp  \seteq \left\{ \mathbf{x} \in \Z_{\congN}^n \; \middle| \; \mathbf{x} \bullet \mathbf{y} \equiv_{\congN} 0 \,\,\text{\rm for all }\,\,\mathbf{y} \in
\ker (\overline{\Afin}^t)  \right\}.\]
Here $\bullet$ is the dot product on $\Z_\congN^n$.
\end{theorem}

\begin{proof}
    The assertion follows from Lemma \ref{characterization of Q} together with Lemma \ref{basis of ker C transpose}.
\end{proof}

For a subset $S \subset \Z^n$, set
$\redr(S): = \{ \mathbf{\overline{s}}\subset (\Z_{\congN})^n \;  | \; \mathbf{s} \in S\}.$
Motivated by~\eqref{eq: root and sieving 1}, we introduce the following definition.

\begin{definition}\label{def: sieving}
    Let $\g$ be an affine Kac-Moody algebra. We call a subset $S \subset \Z^n$ a {\it root-sieving set} 
    if, for all $\bx \in P_0$, 
    \begin{enumerate}
        \item $\bx \in Q_0$ if and only if $\mathbf{s} \bullet [\bx]_{\upvarpi} \equiv_{\congN} 0$ for all $\mathbf{s} \in S$,
        \item the set $\redr(S)\subset (\Z_{\congN})^n$ is $\mathbb Z_{\congN}$- linearly independent, and
         \item  $|\redr(S)|=|S|.$
    \end{enumerate}
    An element in $S$ is called a \emph{root-sieving vector} of $S$.
\end{definition}

For instance, by~\eqref{eq: root and sieving 1}, we have an example of a root-sieving set:
\[\begin{cases}
\left\{ s^{(1)},s^{(2)}\right\}= \{ \adj(\Afin)_{(i_1)} , \adj(\Afin)_{(i_2)} \}& \text{when $\g = D_n^{(1)}$ (for even $n$),}\\[2ex]
\left\{ s \right\}= \{ \adj(\Afin)_{(i_1)} \} & \text{otherwise},
\end{cases}\]

Let $\Lambda = \sum_{0 \le i \le n} p_i \Lambda_i, \Lambda' = \sum_{0 \le i \le n} p_i' \Lambda_i \in \pclp{\ell}$. 
Combining Lemma~\ref{lem: sim equiv} with Definition \ref{def: sieving}, we can deduce the following characterization on the sieving equivalence relation $\sim$:
\begin{align}\label{eq: eqrel and sieving}
\Lambda \sim \Lambda' \quad \text{if and only if} \quad \mathbf{s} \bullet (p_1,p_2,\ldots,p_n) \equiv_{\congN} \mathbf{s} \bullet (p_1',p_2',\ldots,p_n')\quad \text{for all $\mathbf{s} \in S$}.
\end{align}
Here, $\bullet$ is the dot product on $\Z^n$.

\begin{example}
    Let $\g = A_3^{(1)}$. Then
    \[\Afin = \left[\begin{matrix} 2 & -1 & 0 \\ -1 & 2 & -1 \\ 0 & -1 & 2 \end{matrix}\right], \quad
     \invA = \left[\begin{matrix} \frac{3}{4} & \frac{1}{2} & \frac{1}{4} \\[1ex] \frac{1}{2} & 1 & \frac{1}{2} \\[1ex] \frac{1}{4} & \frac{1}{2} & \frac{3}{4} \end{matrix}\right] \quad \text{and} \quad \adj(\Afin) =
     \left[\begin{matrix}
     3 & 2 & 1 \\
     2 & 4 & 2 \\
     1 & 2 & 3
     \end{matrix}\right]. \]
    Let $i_1 = 3$. Note that
    \[2\invA_{(3)} + \Z^3 = \left[\begin{matrix}\frac{1}{2} & 1 & \frac{3}{2}\end{matrix}\right] + \Z^3 = \invA_{(2)} + \Z^3 \]
    and
    \[3\invA_{(3)} + \Z^3 = \left[\begin{matrix}\frac{3}{4} & \frac{3}{2} & \frac{9}{4}\end{matrix}\right] + \Z^3 = \invA_{(1)} + \Z^3. \]
    That is, for $\bx \in P_0$, $\invA_{(3)} [\bx]_{\upvarpi} \in \Z$ if and only if  $\invA_{(i)} [\bx]_{\upvarpi} \in \Z$ for all $i = 1,2,3$. It means that
    \[ \bx \in Q_0 \quad \text{if and only if} \quad 4\left(\invA_{(3)}\right)^t \bullet [\bx]_{\upvarpi} \equiv_{4} 0.\]
    Thus $\left\{ 4\left(\invA_{(3)} \right)^t \right\} = \left\{ \left(\adj(\Afin)_{(3)}\right)^t \right\} = \left\{\left[\begin{matrix}
    1 \\ 2 \\ 3
    \end{matrix}\right]\right\}$ is a root sieving set.
\end{example}

In the rest of this subsection, we classify all root sieving sets up to modulo $\congN$.

\begin{lemma}\label{prop: ker and root sieving set}
Let $S \subset \Z^n$ with $|S| = |\redr(S)|$. Then $S$ is a root sieving set if and only if $\redr(S)$ is a $\Z_{\congN}$-basis of $\ker (\overline{\Afin}^t)$.
\end{lemma}
\begin{proof}
(a)  Suppose that $\redr(S)$ is a $\Z_{\congN}$-basis of $\ker (\overline{\Afin}^t)$.
 Since $\ker (\overline{\Afin}^t) \subset (\Z_{\congN})^n$, $\redr(S)$ should be $\Z_\congN$-linearly  independent. Therefore, it suffices to show that $S$ satisfies the condition (1) in Definition \ref{def: sieving}.

We first show that  if $\bx \in Q_0$, then we have $ [\bx]_{\upvarpi} \bullet \mathbf{s}  \equiv_{\congN} 0$ for all $s \in S$. Take $\bx = \sum_{1 \le i \le n} t_i \finsr_i \in Q_0$.
Since $\redr(S) \subset \ker (\overline{\Afin}^t)$,
\[
 [\bx]_{\upvarpi} \bullet s = \sum_{1 \le i \le n} t_i [\finsr_i]_{\upvarpi} \bullet s = \sum_{1 \le i \le n} t_i (\Afin^t)_{(i)} \mathbf{s}  \equiv_{\congN} 0, \quad \text{for all $\mathbf{s}  \in S$.}
\]

Next, we assume that there is $\bx \notin Q_0$ satisfying $\mathbf{s}\bullet [\bx]_{\upvarpi} \equiv_{\congN} 0$ for all $\mathbf{s}  \in S$. Since $\bx \notin Q_0$, we have
\[\adj(\Afin)_{(i_{k'})}\bullet [\bx]_{\upvarpi} \not\equiv_{\congN} 0 \quad \text{for $k' = 1$ or $k' = 1,2$ up to types},\]
by Lemma~\ref{characterization of Q}. However,  since $\mathbf{s} \bullet [\bx]_{\upvarpi} \equiv_{\congN} 0$ for $\mathbf{s} \in S$, there are no $t_{s} \in \Z$ such that $\sum_{\mathbf{s}  \in S} t_{\mathbf{s}} \mathbf{s} \equiv_{\congN} \adj(\Afin)_{(i_{k'})}$. Since $\adj(\Afin)_{(i_k)} \in \ker (\overline{\Afin}^t)$, it contradicts to the assumption that $\redr(S)$ is a $\Z_\congN$-basis of $\ker (\overline{\Afin}^t)$.

(b) Suppose that $S$ is a root sieving set. By definition, $|S| = |\redr(S)|$, $\redr(S)   \subset \ker (\overline{\Afin}^t)$, and $\redr(S)$ is $\Z_\congN$-linearly independent. Therefore, we have to show that the $\Z_{\congN}$-span of $\redr(S)$ equals $\ker (\overline{\Afin}^t)$.

Note that $$\text{$P_0/Q_0= \Z\{\finfw_{j_1} + Q_0\}$ or $\Z\{\finfw_{j_1} + Q_0, \finfw_{j_2} + Q_0\}$ up to types.}$$ Therefore,
for any $\mathbf{y} = (y_1,y_2,\ldots, y_n)^t \in \ker (\overline{\Afin}^t)$ and $\Lambda\in P$, $(\mathbf{y}\bullet [\bx]_{\upvarpi})$ is determined by $\mathbf{y}\bullet [\finfw_{j_k}]_{\upvarpi} = y_{j_k}$, that is, $\mathbf{y}$ is determined by $y_{j_1}$
(resp. $y_{j_1}$ and $y_{j_2}$).
Now it suffices to show that for any $\mathbf{y} \in \ker (\overline{\Afin}^t)$, there are $\Z_{\congN}$-solutions for the following equation:
\begin{align}\label{eq: span of sieving}
\mathbf{y} = x \ \mathbf{s}
\quad \text{or} \quad \mathbf{y} = x^{(1)} \ \mathbf{s}^{(1)}+x^{(2)} \ \mathbf{s}^{(2)}
\end{align}
Here $\redr(S) = \{\mathbf{s}\}$ (resp. $\{\mathbf{s}^{(1)}, \mathbf{s}^{(2)}\}$).
Since $\mathbf{y}$ is determined by $y_{j_1}$ (resp. $y_{j_1}$ and $y_{j_2}$), the linearly independence of $\redr(S)$ implies the existence of the solution to~\eqref{eq: span of sieving}.
\end{proof}

Lemma \ref{prop: ker and root sieving set} implies that there are finitely many root sieving sets for each type up to modulo $\congN$.
Combining Lemma \ref{basis of ker C transpose} with Table \ref{table: oP/oQ}, we can complete the classification of root sieving sets up to modulo $\congN$,
which is presented in the table below.
\begin{table}[h]
    \centering
    { \arraycolsep=1.6pt\def\arraystretch{1.5}
        \begin{tabular}{c|c}
            Type &  Root sieving sets up to modulo $\congN$ \\ \hline
            $A_{n}^{(1)} $  & $\{k(1,2,\ldots,n)\}$,\quad for $(k,n+1)=1$  \\ \hline
            $B_n^{(1)}, D_{n+1}^{(2)}$  & $\{(0,0,\ldots,0,1)\}$  \\ \hline
            $C_n^{(1)}, A_{2n-1}^{(2)}$ & $\{(\delta(j\equiv_2 1))_{j=0,1,\ldots,n}\}$ \\ \hline
            $D_n^{(1)} (n \equiv_2 0)$  &
            $\begin{array}{ll}
            \{(0,0,\ldots,0,2,2), (2,0,2,0,\ldots,2,0,2,0)\}, \\
            \{(0,0,\ldots,0,2,2), (2,0,2,0,\ldots,2,0,0,2)\}, \\
            \{(2,0,2,0,\ldots,2,0,2,0), (2,0,2,0,\ldots,2,0,0,2)\}
            \end{array}$\\ \hline
            $D_n^{(1)} (n \equiv_2 1)$ &
            $\{k(2,0,2,0,\ldots,0,2,1,3)\}$, \quad for $k = 1,3$
            \\ \hline
            $E_6^{(1)}$ &  $\{k(1,0,2,0,1,2)\}$, \quad for $k = 1,2$ \\ \hline
            $E_7^{(1)}$ &  $\{(0,1,0,0,1,0,1)\}$ \\ \hline
            Remaining types & $\{(0,0,\ldots,0)\}$
        \end{tabular}
    }\\[1.5ex]
    \caption{Root sieving sets up to modulo $\congN$}
    \protect\label{table: sieving}
\end{table}

\vskip 30mm

\begin{convention}\label{conv: Sieving} \hfill
\begin{enumerate}
\item[{\rm (1)}] From now on, we choose a special root sieving set, denoted by $\svs$, as follows:
\[
\svs = \begin{cases}
\{\svv=(1,2,\ldots,n)\} & \text{ if } \g=A_n^{(1)}, \\
\{\svv=(1,0,2,0,1,2)\} & \text{ if } \g=E_6^{(1)}, \\
\{\svv=(2,0,2,0,\ldots,0,2,1,3)\} & \text{ if } \g=D_n^{(1)} \text{ and } n\equiv_2 1, \\
\{\svv^{(1)}=(0,0,\ldots,0,2,2), \svv^{(2)}=(2,0,2,0,\ldots,2,0,2,0)\} & \text{ if } \g=D_n^{(1)} \text{ and } n\equiv_2 0.
\end{cases}
\]
For the other types, we choose $\svs$ as in Table~\ref{table: sieving}.
\item[{\rm (2)}] For a root sieving vector $\svv=(s_1,s_2,\ldots,s_n)$, we denote $(0,s_1,s_2,\ldots,s_n)$ by $\wtsvv$.
\end{enumerate}
\end{convention}

With the root sieving sets $\svs$ given in Convention~\ref{conv: Sieving}, we define a new statistics $\evS$, called the \emph{$\svs$-evaluation},
\begin{align}\label{eq: S-evaluation}
\evS: \pclp{\ell} \ra \Z_{\ge 0 }^k, \quad \sum_{0\le i \le n} m_i \Lambda_i \mapsto \left( \wtsvv^{(k)} \bullet \mathbf{m} \right)_{k =1 \text{ or } 1,2}.
\end{align}
Here, $\wtsvv^{(1)} = \wtsvv$ in cases except for $D_n^{(1)}~(n\equiv_2 0)$ and $\bm = (m_0,m_1,\ldots,m_n)$.
For $\Lambda \in \pclp{\ell}$, we call $\evS(\Lambda)$ the \emph{$\svs$-evaluation} of $\Lambda$.
For later use, we list $\evS(\Lambda)$ for all $\Lambda = (\ell - 1)\Lambda_0 + \Lambda_i \in \DRpclp$ in Table~\ref{table: ev}.
\begin{table}[h]
    \centering
    { \arraycolsep=1.6pt\def\arraystretch{1.5}
        \begin{tabular}{c|c|c|c|c|c|}
            &  $    A_{n}^{(1)},C_n^{(1)}, A_{2n-1}^{(2)}$ & $  B_n^{(1)}, D_{n+1}^{(2)}, E^{(1)}_{7} $ & $D^{(1)}_{n}$ & $E_6^{(1)}$ \\ \hline
            \multirow{2}{*}{$\evS(\Lambda)$}  & $i$ & $\delta_{in}$ & $\begin{array}{ll}
            (2(\delta_{i,n-1}+\delta_{i,n}), 2(\delta_{i,1} + \delta_{i,n-1}))
            & (n\equiv_2 0), \\
            2\delta_{i,1} + \delta_{i,{n-1}} + 3\delta_{i,n} & (n \equiv_2 1) \end{array}$
            & $\delta_{i,n}+2\delta_{i,6}$  \\ \cline{2-5}
            &\multicolumn{4}{|c|}{\text{\rm For the remaining types, $\evS(\Lambda) = 0$}} \\ \hline
        \end{tabular}
    }\\[1.5ex]
    \caption{$\evS((\ell - 1)\Lambda_0 + \Lambda_i)$ for each type}
    \protect\label{table: ev}
\end{table}

The following theorem follows from~\eqref{eq: eqrel and sieving}.

\begin{theorem}\label{thm: eq rel and sieving}
Let $\svs$ be the root sieving set given in Convention~\ref{conv: Sieving}.
For any $\Lambda = (\ell - 1)\Lambda_0 + \Lambda_i \in \DRpclp$, we have
\begin{align}\label{eq: pis ell Lambda}
    \pclp{\ell}(\Lambda) = \left\{ \Lambda' \in \pclp{\ell} \; \middle| \; \evS(\Lambda') \equiv_{\congN} \evS(\Lambda)   \right\}.
\end{align}
\end{theorem}

\begin{example}
Let $\g = A_3^{(1)}$ and $\ell = 2$. In this case, $a_i^\vee = 1$, $\svv = (1,2,3)$ and $\evS((\ell - 1)\Lambda_0 + \Lambda_i) = i$ for $i = 0,1,2,3$. Then
\begin{align*}
\pclp{2} &= \left\{\sum_{0 \le i \le 3} m_i \Lambda_i  \in \dwt_2 \; \middle| \;
\sum_{0 \le j \le 3} m_j = 2 \right\}\\
&= \left\{ 2\Lambda_0, 2\Lambda_1, 2\Lambda_2, 2\Lambda_3, \Lambda_0 + \Lambda_1, \Lambda_0 + \Lambda_2, \Lambda_0 + \Lambda_3, \Lambda_1 + \Lambda_2, \Lambda_1 + \Lambda_3, \Lambda_2 + \Lambda_3 \right\}
\end{align*}
and, for each $i = 0,1,2,3$,
\begin{align*}
{\pclp{2}}(\Lambda_0+\Lambda_i)& = \left\{ \sum_{0 \le i \le 3} m_i \Lambda_i \in \dwt_2 \; \middle| \;
\sum_{0 \le j \le 3} m_j = 2 ~\text{and}~ \sum_{0 \le j \le 3} jm_j \equiv_4 i \right\}.
\end{align*}
For instance,
\begin{align*}
{\pclp{2}}(2\Lambda_0) & =\left\{ 2\Lambda_0, 2\Lambda_2, \Lambda_1 + \Lambda_3 \right\}.
\end{align*}
\end{example}

\begin{remark}
Even in case where $\g = A_{2n}^{(2)}$, we can define the sieving equivalence relation as in~\eqref{eq: sieving relation}. In this case, there is only one equivalence class and hence we may define the distinguished representative $\DRpclp$ as $\{\ell\Lambda_0\}$. Then we have the same bijection described in~\eqref{eq: bijection between ellcov and pclp}, which implies that for any $\Lambda \in \dwt_\ell$,
\[|\ellcov| = |\pclp{\ell}|.\]
\end{remark}

\section{Sagan's action and generalization}\label{sec: Sagan action and general}

From this section, we will investigate the structure and enumeration of $\pclp{\ell}(\Lambda)$ for all $\Lambda \in \DR(\pclp{\ell})$ in a viewpoint of (bi)cyclic sieving phenomena (\cite{RSW}). In order to do this, we give a suitable (bi)cyclic group action on $\pclp{\ell}$. This will be achieved by generalizing Sagan's action in \cite{Sagan1} under consideration on our results in the previous sections.
\smallskip

For each positive integer $m$, we fix a cyclic group $C_m$ of order $m$ and a generator $\siggen_m$ of $C_m$. Note that every $C_m$-action is completely determined by the action of $\siggen_m$.

In \cite[\S 2]{Sagan1}, Sagan introduced an interesting cyclic group action on sets consisting of $(0,1)$-words.
Here we provide a generalized version of this action, which will play a key role in our demonstration of cyclic sieving phenomena
associated with dominant maximal weights.
To do this, we first recall Sagan's action.

Let
\begin{align}\label{def of Wnl}
\wordset_{n,\ell} \seteq \left\{ \mathbf{w} = w_1 w_2 \cdots w_{n+\ell} \; \middle| \; w_i = 0,1 \quad \text{for $i= 1,2,\ldots, n+\ell$, and } \sum_{1 \le i \le n+\ell} w_i = \ell \right\},
\end{align}
which is in one to one correspondence with $\pclp{\ell}$ of type $A_n^{(1)}$ via 
\begin{align}\label{eq: W and pis corres}
\sum_{0 \le i\le n} m_i \Lambda_i \mapsto \underbracket{11\cdots 1}_{m_0} 0 \underbracket{1 1 \cdots 1}_{m_1} 0 \cdots\cdots 0 \underbracket{1 1 \ldots 1}_{m_n}.
\end{align}

For any $d \in \Z_{\ge 1}$, we define a $C_d = \inn{\siggen_d}$-action on $\wordset_{n,\ell}$ as follows:
Given a $(0,1)$-word $\mathbf{w} = w_1 w_2  \cdots  w_{n+\ell} \in \wordset_{n,\ell}$, break it into subwords of length $d$,
\begin{align*}
\mathbf{w} &= w_1 w_2 \cdots w_d \mid w_{d+1} w_{d+2} \cdots  w_{2d} \mid \cdots \mid w_{(t-1)d + 1}  w_{(t-1)d + 2}  \cdots  w_{td} \mid w_{td+1} \cdots  w_{n+\ell}\\
& = w^1 \mid w^2 \mid \cdots \mid w^t \mid w^0,
\end{align*}
where $t = \floor{\frac{n+\ell}{d}}$, 
\begin{align*}
&w^j \seteq w_{(j-1)d +1} w_{(j-1)d +2} \cdots w_{jd} \quad\text{ for }1\le j \le t,\ \text{ and } \ w^0 \seteq w_{td+1} \cdots  w_{n+\ell}.
\end{align*}
Note that $C_d$ acts on each subword $w^j$ by cyclic shift:
\[\siggen_d \cdot w^j \seteq w_{jd}, w_{(j-1)d +1} w_{(j-1)d +2} \cdots w_{jd -1}. \]
Assume that $j_0$ is the smallest integer such that $\siggen_d \cdot w^{j_0} \neq w^{j_0}$. Then Sagan's action $\sqact$ is defined by
\[ \siggen_d \sqact \mathbf{w} \seteq w^1 \mid w^2 \mid \cdots \mid w^{j_0 - 1} \mid \siggen_d \cdot w^{j_0} \mid w^{j_0+1} \mid \cdots \mid w^t \mid w^0. \]
If there is no such $j_0$ in \{1,2,\ldots,t\}, set $\siggen_d \sqact \mathbf{w} \seteq \mathbf{w}$.

\begin{example}\label{eg: hsig action}
Note that
    \[\wordset_{3,2}= \left\{ 
    11000,\  01100,\  00110,\  10010,\ \ \ 
    10001,\  01001,\  00101,\  00011,\ \ \
    10100,\  01010
\right\}.\]
Under the above $C_4$-action on $\wordset_{3,2}$, we have three orbits given by
    \[ 
    \{1100  |  0,\  0110  |  0,\  0011  |  0,\  1001  |  0 \},\quad 
    \{1000  |  1,\  0100  |  1,\  0010  |  1,\  0001  |  1\},\quad
    \{1010  |  0,\  0101  |  0 \}.
 \]
\end{example}

Via the correspondence in~\eqref{eq: W and pis corres}, we can transport Sagan's actions on $\wordset_{n,\ell}$ to $\pclp{\ell}$ of type $A_n^{(1)}$.
In the following, we will generalize this approach to other types. Although our setting is more general, basically we construct a set in bijection with $\pclp{\ell}$ and define cyclic group actions on it by mimicking Sagan's actions.

In this section, we assume that $d,k$ are positive integers and $\ell$ is a nonnegative integer.
Given a $kd$-tuple $\mathbf{m} = (m_0,m_1,\ldots,m_{kd-1}) \in \Z_{\ge 0}^{kd}$,
we set
$$\mathbf{m}[j;d]:=\sum_{0 \le t \le d-1} m_{jd +t}\quad  (0\le j \le k-1).$$
Also, given a $k$-tuple $\bnu=(\nu_0, \nu_1,\ldots, \nu_{k-1}) \in \Z^k_{> 0}$, we set
\begin{align}\label{eq: def Pelldeta}
\bM_\ell (d;\bnu) := \left\{ \mathbf{m} = (m_0,m_1,\ldots,m_{kd-1}) \in \Z_{\ge 0}^{kd} \; \middle| \;
\sum_{0 \le j \le k-1} \nu_j \mathbf{m}[j;d]= \ell
\right\}.
\end{align}
In particular, if $d=1$ then $\mathbf{m}[j;1] = m_j$ and
\begin{align}\label{eq: tuset d=1}
    \bM_\ell (1;\bnu)
    & = \left\{ \mathbf{m} = (m_0,m_1,\ldots,m_{k-1}) \in \Z_{\ge 0}^{k} \; \middle| \;
    \bnu \bullet \mathbf{m}= \ell
    \right\}.
\end{align}

To each $\mathbf{m} = (m_0,m_1,\ldots,m_{kd-1}) \in {\bM_\ell} (d;\bnu)$
we associate a word
$$\mathbf{w}(\mathbf{m};d;\bnu) := w_1w_2\cdots w_{u_{\mathbf{m}}}$$
with entries in $\{0,\nu_0,\nu_1,\ldots,\nu_{k-1}\}$
defined by the following algorithm:

\begin{algorithm}\label{Algorithm} {\rm (Algorithm for $\mathbf{w}(\mathbf{m};d;\bnu)$)} Assume we have a $kd$-tuple $\mathbf{m}= (m_0,m_1,\ldots,m_{kd-1}) \in \bM_\ell (d;\bnu)$.
\begin{itemize}
\item[(${\mathbf A}$1)] Set $\mathbf{w}$ to be the empty word and $j = 0, t = 0$. Go to $(\mathbf{A}$2).
\item[(${\mathbf A}$2)] Set $\mathbf{w}$ to be the word obtained by concatenating $m_{jd + t}$ $\nu_{j}$'s at the right of $\mathbf{w}$. If $j = k-1$ and $t = d-1$, return $\mathbf{w}$ and terminate the algorithm.
Otherwise, go to $(\mathbf{A}$3).
\item[(${\mathbf A}$3)] Set $\mathbf{w}$ to be the word obtained by concatenating $0$ at the right. Go to $(\mathbf{A}$4).
\item[(${\mathbf A}$4)] If $t \neq d-1$ then set $t = t + 1$ and go to $(\mathbf{A}$2). If $t = d-1$ set $j= j+1$ and $t = 0$, and go to $(\mathbf{A}$2).
\end{itemize}
\end{algorithm}

As seen in  Algorithm~\ref{Algorithm}, the length $u_{\mathbf{m}}$ determined by $\mathbf{m}$ and the formula for $u_{\mathbf{m}}$ is given as follows:
\[ u_{\mathbf{m}} = \left( \sum_{0 \le j \le kd-1}m_j \right) + (kd-1). \]
For $\mathbf{m},\mathbf{m}' \in {\bM_\ell} (d;\bnu)$, the lengths  $u_{\mathbf{m}}$ and $u_{\mathbf{m}'}$
are not necessarily equal to each other. On the contrary, for all $\bm \in \bM_\ell(d;\bnu)$, the number of $0$'s in $\bw(\bm;d;\bnu)$ is uniquely determined by $kd-1$ (see Example~\ref{eg: dwt to word} below).

\smallskip

Set
\[\wordset_\ell(d;\bnu) \seteq \{\mathbf{w}(\mathbf{m};d;\bnu) \mid \mathbf{m} \in \bM_\ell (d;\bnu)  \}, \]
which can be viewed as a generalization of $\wordset_{n,\ell}$ since
$\wordset_\ell(1;\bnu)$ recovers $\wordset_{n,\ell}$ when $k = n + 1$ and $\bnu = (1,1,\ldots, 1)\in \Z^{k}$.

\begin{lemma}\label{lem: injective}
The map
\begin{align}\label{eq: funcPW}
\Uppsi : {\bM_\ell} (d;\bnu) \ra \wordset_\ell(d;\bnu), \quad \mathbf{m} \mapsto \mathbf{w}(\mathbf{m};d;\bnu)
\end{align}
is injective and hence bijective.
\end{lemma}

\begin{proof}
For each $\mathbf{m} \in \bM_\ell(d;\bnu)$,  we have to apply  (${\mathbf A}$3) $(kd-1)$-times to obtain $\bwm$ via Algorithm~\ref{Algorithm}.
This says that every word $\mathbf{w} \in \wordset_\ell(d;\bnu)$ contains exactly $(kd-1)$-zero.

Define a map $\Uppsi^{-1}: \wordset_\ell(d;\bnu) \ra {\bM_\ell} (d;\bnu)$ as follows:
Let $\bw \in \wordset_\ell(d;\bnu)$. For each $1\le i \le kd-1$, let $z_i$ denote the position of the $i$th zero when we read $\bw$ from left to right, and we set $z_0 \seteq 0$ and $z_{kd} \seteq u_{\bm}+1$.
For each $0 \le j \le kd-1$, let $m_i = z_{i+1} - z_i - 1$. 
Define
\[\Uppsi^{-1}(\bw) = (m_0,m_1,\ldots,m_{kd-1}).\]
Recall that, for each $\bm = (m_0,m_1,\ldots,m_{kd-1}) \in \bM_\ell(d;\bnu)$, $\bwm$ is obtained by applying Algorithm~\ref{Algorithm} to $\bm$, which shows that there are exactly $m_i$ nonzero entries between the $i$th $0$ and the $(i+1)$st $0$ when we read $\bwm$ from left to right for $0 \le i \le kd-1$. Here the $0$th and $kd$th $0$'s are set to be the empty word (see Table~\ref{table: Cact eg} for details). Obviously it holds that
\[\Uppsi^{-1} \circ \Uppsi (\bm) = \bm \quad \text{ for each } \  \bm \in {\bM_\ell} (d;\bnu) \]
and hence $\Uppsi$ is injective.
\end{proof}

\begin{example}\label{eg: dwt to word}
Let $d=2,k=2,\ell = 4$, and $\bnu = (1,2)$. Then
\begin{align*}
\bM_4(2;(1,2)) &= \left\{ (m_0,m_1,m_2,m_3) \in \Z^{4}_{\ge 0} \; \middle| \; m_0 + m_1 + 2m_2 + 2m_3 = 4 \right\} \\
& = \left\{ \begin{array}{lll}
(4,0,0,0), & (2,0,1,0), & (0,0,2,0),    \\
(3,1,0,0), & (2,0,0,1), & (0,0,1,1),    \\
(2,2,0,0), & (1,1,1,0), & (0,0,0,2),    \\
(1,3,0,0), & (1,1,0,1),\\
(0,4,0,0), & (0,2,1,0),\\
& (0,2,0,1)
\end{array} \right\}.
\end{align*}
Using Algorithm \ref{Algorithm}, one can obtain $\Uppsi(\mathbf{m})$ for each $\mathbf{m} \in \bM_4(2;(1,2))$ as follows:
\[
\wordset_4(2;(1,2)) = \left\{ \begin{array}{lll}
\Uppsi((4,0,0,0)) = 1111000,    & \Uppsi((2,0,1,0)) = 110020,   & \Uppsi((0,0,2,0)) = 00220,    \\
\Uppsi((3,1,0,0)) = 1110100,    & \Uppsi((2,0,0,1)) = 110002,   & \Uppsi((0,0,1,1)) = 00202,    \\
\Uppsi((2,2,0,0)) = 1101100,    & \Uppsi((1,1,1,0)) = 101020,   & \Uppsi((0,0,0,2)) = 00022,    \\
\Uppsi((1,3,0,0)) = 1011100,    & \Uppsi((1,1,0,1)) = 101002,   \\
\Uppsi((0,4,0,0)) = 0111100,    & \Uppsi((0,2,1,0)) = 011020,   \\
& \Uppsi((0,2,0,1)) = 011002,
\end{array}\right\}.\]
In particular, $\Uppsi ((2,0,1,0)) = 110020$ can be computed as follows:
\begin{table}[h]
\centering
\begin{tabular}{c|c|c|c|c|c|c|c|c|c|c|c}
& (${\mathbf A}$1) & (${\mathbf A}$2) & \text{(${\mathbf A}$3)} & \text{(${\mathbf A}$4)} & \text{(${\mathbf A}$2)} & \text{(${\mathbf A}$3)} & \text{(${\mathbf A}$4)} & \text{(${\mathbf A}$2)} &
\text{(${\mathbf A}$3)} & \text{(${\mathbf A}$4)} & \text{(${\mathbf A}$2)}  \\ \hline
$\mathbf{w}$ & $\emptyset$ & 11 & 110 & 110 & 110 & 1100 & 1100 & 11002 & 110020 & 110020 & 110020  \\ \hline
$j$ & 0 & 0 & 0 & 0 & 0 & 0 & 1 & 1 & 1 & 1 & 1  \\ \hline
$t$ & 0 & 0 & 0 & 1 & 1 & 1 & 0 & 0 & 0 & 1 & 1
\end{tabular}
\caption{The process of obtaining $\Uppsi ((2,0,1,0))$ by Algorithm~\ref{Algorithm}}
\protect\label{table: Cact eg}
\end{table}
\end{example}

\begin{remark} \hfill
\begin{enumerate}
\item There are five words of length $7$, six words of length $6$, and three words of length $5$ in $\wordset_4(2;(1,2))$. This shows that the lengths of $\mathbf{m}$'s may be different.
\item The set $\wordset_\ell(d;\bnu)$ may be complicated to some extent.
The definition in~\eqref{def of Wnl} implies that all words of length $n+\ell$ consist of $n$ $0$'s and $\ell$ $1$'s are in $\wordset_{n,\ell}$, but which fails to characterize $\wordset_\ell(d;\bnu)$. For instance, although $110200, 110020$ and $110002$ have the same number of $i$'s ($i = 0,1,2$), Example~\ref{eg: dwt to word} shows that
\[110020, 110002 \in \wordset_4(2;(1,2)) \quad \text{but} \quad 110200 \not\in \wordset_4(2;(1,2)).\]
\end{enumerate}
\end{remark}

Now we define a $C_d = \inn{\siggen_d}$-action on $\wordset_\ell(d;\bnu)$.
First, we break $\mathbf{w} = w_1 w_2\ldots w_{u}$ into subwords of length $d$ as many as possible in order as follows:
\begin{equation}\label{eq: breakw}
\begin{aligned}
\mathbf{w}& =w_1w_2\cdots w_d \ | \ w_{d+1}w_{d+2}\cdots w_{2d} \ | \ \cdots \ | \ w_{(k-1)d+1}w_{(k-1)d+2}\cdots w_{td} \ | \ w_{td+1}\cdots w_u  \\
&= w^1 \ | \ w^2 \ | \ \cdots \ | \ w^t \ | \ w_{td+1}\cdots w_u,
\end{aligned}
\end{equation}
where $t = \lfloor u/d \rfloor$ and $w^j=w_{(j-1)d+1}w_{(j-1)d+2} \cdots w_{jd}$ for $1 \le j \le t$.
Note that $\siggen_d$ acts on each
subword $w^j$ by cyclic shift, i.e.,
$$  \siggen_d \cdot w^j \seteq w_{jd} w_{(j-1)d+1} w_{(j-1)d+2} \cdots w_{jd-1}.$$
Assume that $j_0$ is the smallest integer such that $ \siggen_d \cdot w^j_0 \ne w^j_0$.
Then we set
\begin{align}\label{eq: def of hsig}
\siggen_d \sqact \mathbf{w} \seteq w^1 \ | \ w^2 \ | \ \cdots \ | \ w^{{j_0}-1} \ | \ \siggen_d \cdot w^{j_0} \ | \ w^{j_0+1} \ | \ \cdots \ | \ w^t \ | \ w_{td+1}\cdots w_{u}.
\end{align}
If there is no such $j_0$, we set $ \siggen_d \sqact \mathbf{w} \seteq \mathbf{w}$.

\begin{theorem}\label{thm: hsigaction}
For any $\bnu = (\nu_0, \nu_1,\ldots, \nu_{k-1}) \in \Z^k_{> 0}$, the action defined as above is indeed a $C_d$-action on $\wordset_\ell(d;\bnu).$
\end{theorem}

\begin{proof}
From the definition in~\eqref{eq: def of hsig}, one can see that $e \sqact \bw = \underbracket{(\siggen_d \sqact (\siggen_d \sqact \cdots \sqact (\siggen_d}_{d} \sqact \bw)\cdots))  = \bw$ for all $\mathbf{w} \in \wordset_\ell(d;\nu)$.
Therefore, our assertion can be justified by showing that $\wordset_\ell(d;\nu)$ is closed under the action of $\siggen_d$. To do this, for any $\bw \in \wordset_\ell(d;\nu)$, we will find an element $\bm' \in \bM_\ell(d;\bnu)$ such that $\Uppsi (\bm') = \siggen_d \sqact \mathbf{w}$.

Let $\bw \in \wordset_\ell(d;\nu)$. We may assume that $\siggen_d \sqact \bw \neq \bw$.
Break $\bw$ into subwords
\begin{align*}
\bw & =w_1w_2\cdots w_d \ | \  w_{d+1}w_{d+2}\cdots w_{2d} \ | \ \cdots \ | \ w_{(t-1)d+1}w_{(t-1)d+2}\cdots w_{td} \ | \ w_{td+1}\cdots w_u  \\
&= w^1 \ | \ w^2 \ | \ \cdots \ | \ w^t \ | \ w_{td+1}\cdots w_u
\end{align*}
as in~\eqref{eq: def of hsig}.
Since $\siggen_d \sqact \mathbf{w} \neq \mathbf{w}$, there exists an index $1 \le j_0 \le t$ such that
\[\siggen_d \sqact \mathbf{w} = w^1 \ | \ w^2 \ | \ \cdots \ | \ w^{j_0-1} \ | \ \siggen_d \cdot w^{j_0} \ | \ w^{j_0+1} \ | \ \cdots \ | \ w^t \ | \ w_{td+1}\cdots w_u.\]
Note that for each $1 \le j \le j_0$, $w^j$ consists of $d$ $0$'s or $d$ $\nu_r$'s for some $0\le r < k$. Thus, the number of zeros in $w^1w^2\cdots w^{j_0-1}$ is $s\times d$ for some $s \in \Z_{\ge 0}$.
Moreover, from Algorithm~\ref{Algorithm}  and $\siggen_d \sqact w^{j_0} \neq w^{j_0}$, we see that
\begin{enumerate}[label = (\roman*)]
\item $\nu_{s+1}$ can appear in $\bw$ after $(s+1)\times d$ zeros occurrence, and
\item $w^{j_0}$ consists of $z$ $0$'s and $(d-z)$ $\nu_{s}$'s for some $z \ge 1$.
\end{enumerate}

By Lemma~\ref{lem: injective}, we can write $\Uppsi^{-1}(\mathbf{w})$ as $\bm = (m_0,m_1,\ldots,m_{kd-1})$.
Now we shall construct a tuple $\bm' = (m_0',m_1',\ldots, m_{kd-1}') \in \Z^{kd}_{\ge 0}$ satisfying that $\bm'  =  \Uppsi^{-1}(\siggen_d \sqact  \mathbf{w}) \in \bM_\ell(d;\bnu)$ in the following steps:
\begin{itemize}
\item[{\rm Step 1.}] Take  $m_i' = m_i$  for $0 \le i \le sd-1$.
\item[{\rm Step 2.}] Recall that $z$ denotes the number of $0$'s in $w^{j_0}$. Take \begin{align}\label{eq: step2}
\begin{cases}
m'_{sd} = m_{sd} + 1, \  m'_{sd +1}  = m_{sd +1}, \  m'_{sd + z} = m_{sd + z} - 1, \quad \text{and} \quad m'_i = m_i & \text{ if $w_{j_0 d} = \nu_{s}$}, \\
m'_{sd} = m_{sd} - p, \  m'_{sd +1} = p, \ m'_{sd+z} = m_{sd + z -1} + m_{sd + z} \quad \text{and} \quad m'_{i} =  m_{i-1}&\text{ if $w_{j_0 d} = 0$},
\end{cases}
\end{align}
for $sd+1  < i < sd+z$,  where $w^{j_0} = \underbracket{\nu_{s+1}\nu_{s+1}\cdots \nu_{s+1}}_{p}0**\cdots*w_{j_0 d}$.
\item[{\rm Step 3.}] Take $m_i' = m_i$ for $sd+z+1 \le i \le kd-1$.
\end{itemize}

By the construction of $\bm'$, we have
\[\sum_{0 \le j \le k-1} \nu_j \mathbf{m'}[j;d] = \sum_{0 \le j \le k-1} \nu_j \mathbf{m}[j;d] = \ell.\]
    Thus, we have $\bm' \in \bM_\ell(d;\bnu)$. Moreover,~\eqref{eq: step2} implies $\Uppsi (\bm') = \siggen_d \sqact \bw$, by Algorithm~\ref{Algorithm}.
\end{proof}

Now we define a $C_d$-action on $\bM_\ell (d;\bnu)$
by transporting the $C_d$-action $\sqact$ on $\wordset_\ell(d;\bnu)$ via the bijection $\Uppsi$, that is,
\begin{align}\label{eq: Sagact}
\siggen_d \sqact \mathbf{m} \seteq \Uppsi^{-1}( \siggen_d \sqact \Uppsi(\mathbf{m}))\quad \text{for all $\mathbf{m} \in \bM_\ell (d;\bnu)$}.
\end{align}

\begin{example}\label{eg: hsigaction}
    Let $d=2,k=2,\ell = 4$, and $\bnu = (1,2)$. For $\mathbf{m} = (3,1,0,0), \mathbf{m}' = (2,0,1,0), \mathbf{m}'' \in \bM_4(2;(1,2))$, we have the following commutative diagrams:
    \[\begin{tikzcd}[column sep=scriptsize]
    {\mathbf{m} = (3,1,0,0)} \arrow[r, "\Uppsi "] \arrow[d, "\siggen_2", dotted] & 11 | 10 | 10 | 0  \arrow[d, "\siggen_2"]        \\
    {\siggen_2 \sqact \mathbf{m} = (2,2,0,0)} & 11 | 01 | 10 | 0,  \arrow[l, "\Uppsi ^{-1}"]
    \end{tikzcd}
    \
    \begin{tikzcd}[column sep=scriptsize]
    {\mathbf{m}' = (1,1,1,0)} \arrow[r, "\Uppsi "] \arrow[d, "\siggen_2", dotted] & 10|10|20  \arrow[d, "\siggen_2"]        \\
    {\siggen_2 \sqact \mathbf{m}' = (0,2,1,0)} & 01|10|20,  \arrow[l, "\Uppsi ^{-1}"]
    \end{tikzcd}
    \
    \begin{tikzcd}[column sep=scriptsize]
    {\mathbf{m}'' = (0,0,1,1)} \arrow[r, "\Uppsi "] \arrow[d, "\siggen_2", dotted] & 00 | 20 | 2 \arrow[d, "\siggen_2"]        \\
    {\siggen_2 \sqact \mathbf{m}'' = (0,0,0,2)} & 00 | 02 | 2  \arrow[l, "\Uppsi ^{-1}"]
    \end{tikzcd}
    \]
\end{example}
\begin{remark}\label{rem: hsig remark} \hfill
\begin{enumerate}
\item[{\rm (1)}] Suppose that $C_d$ acts on $\bM_\ell(d;\bnu)$ as in~\eqref{eq: Sagact}.
 Then, for any $r \in \Z_{>0}$, $\bM_\ell(d;\bnu)$ is also equipped with a $C_{rd}$-action $\sqact_d$, which is given by
\begin{align}\label{eq: hsigaction multi}
\siggen_{rd} \sqact_d \mathbf{m} \seteq \siggen_d \sqact \mathbf{m}.
\end{align}
\item[{\rm (2)}] In~\eqref{eq: hsigaction multi}, if $d = 1$ then the $C_r$-action $\sqact_1$ on $\bM_\ell(1;\bnu)$ is trivial.
\end{enumerate}
\end{remark}

\medskip
Let us generalize the above setting a little further. Let $d, k, k'$ and $r$ be positive integers and $\ell$ a nonnegative integer. For $\bnu=(\nu_0, \nu_1,\ldots, \nu_{k-1}) \in \Z^k_{> 0}$ and $\bnu'=(\nu'_0, \nu'_1,\ldots, \nu'_{k'-1}) \in \Z^{k'}_{> 0}$, set
\begin{align}\label{eq: def of M}
    \bM_\ell(rd,d;\bnu,\bnu') \seteq \left\{ \mathbf{m} \in \Z_{\ge 0}^{krd + k'd} \; \middle| \;
    \sum_{0 \le j \le k-1} \nu_j \mathbf{m}[j;rd] + \sum_{0 \le j \le k'-1} \nu_j'\mathbf{m}[kr + j;d] =
    \ell \right\}.
\end{align}
Using the actions given in~\eqref{eq: def of hsig} and~\eqref{eq: hsigaction multi}, we define a new $C_{rd}$-action, denoted by $\sqact_{rd,d}$, on $\bM_\ell(rd,d;\bnu,\bnu')$ as follows:
Given $\mathbf{m} \in \bM_\ell(rd,d;\bnu,\bnu')$, we break it into
$\mathbf{m}_{\le krd-1} \in \bM_l(rd;\bnu)$ and $\mathbf{m}_{\ge krd} \in \bM_{l'}(d;\bnu')$, where $\ell = l + l'$.
Now, we define
\begin{align}\label{eq: hsigaction double}
\siggen_{rd} \sqact_{rd,d} \mathbf{m} \seteq \begin{cases}
(\siggen_{rd} \sqact \mathbf{m}_{\le krd-1})* \mathbf{m}_{\ge krd}  & \text{if $\siggen_{rd} \sqact \mathbf{m}_{\le krd-1} \neq \mathbf{m}_{\le krd-1}$,} \\
\mathbf{m}_{\le krd-1}* (\siggen_{rd} \sqact_d \mathbf{m}_{\ge krd}) & \text{otherwise.}
\end{cases}
\end{align}

\begin{example}
Let $d= 2, k = 1,k' = 2,r = 2, \ell = 8, \bnu = (1)$, and $\bnu' = (1,2)$. Then
\[\bM_{8}(4,2;(1),(1,2)) = \left\{ \mathbf{m} \in \Z_{\ge 0}^{8} \; \middle| \; (m_0 + m_1 + m_2 + m_3) + (m_4 + m_5) + 2(m_6 + m_7) = 8 \right\}, \]
where $\mathbf{m} = (m_0,m_1,m_2,m_3,m_4,m_5,m_6)$.

\begin{enumerate}[label = (\arabic*)]
\item For $\mathbf{m} = (6,0,0,0,1,1,0,0) \in \bM_{8}(4,2;(1),(1,2))$, break $\mathbf{m}$ into
\begin{align*}
\mathbf{m}_{\le 3} & = (6,0,0,0) \in \bM_{6}(4;(1)) \quad \text{ and } \quad  \mathbf{m}_{\ge 4} = (1,1,0,0) \in \bM_{2}(2;(1,2)).
\end{align*}
Since $\Uppsi((6,0,0,0)) = 1111|1100|0$, it follows that $\siggen_4 \sqact \Uppsi((6,0,0,0)) = 1111|0110|0$ and so
\[\siggen_4 \sqact (6,0,0,0) = (4,2,0,0).\]
Thus,~\eqref{eq: hsigaction double} shows that
\begin{align*}
\siggen_4 \sqact_{4,2} \mathbf{m} & = (\siggen_4 \sqact \mathbf{m}_{\le 3})* \mathbf{m}_{\ge 4} = (4,2,0,0 \; | \;1,1,0,0).
\end{align*}

\item For $\mathbf{m} = (4,0,0,0,1,1,1,0) \in \bM_{8}(4,2;(1),(1,2))$, break $\mathbf{m}$ into
\begin{align*}
\mathbf{m}_{\le 3} &= (4,0,0,0) \in \bM_{4}(4;(1)) \quad \text{ and } \quad \mathbf{m}_{\ge 4} = (1,1,1,0) \in \bM_{4}(1;(1,2)).
\end{align*}
Since $\Uppsi(\mathbf{m}_{\le 3}) = 1111|000$, one can see that $\siggen_4 \sqact(4,0,0,0) = (4,0,0,0)$. In Example \ref{eg: hsigaction}, we have already shown that $\siggen_2 \sqact \mathbf{m}_{\ge 4} = (0,2,1,0)$.
Thus, by~\eqref{eq: hsigaction double}, we have
\begin{align*}
\siggen_4 \sqact_{4,2} \mathbf{m} & = \mathbf{m}_{\le 3}*(\siggen_2 \sqact \mathbf{m}_{\ge 4})  = (4,0,0,0 \; | \; 0,2,1,0).
\end{align*}
\end{enumerate}
\end{example}

\begin{remark}\label{rem: Sagact dist}
    The expression for $\bM_\ell(rd,d;\bnu,\bnu')$ may not be unique. For instance, $\bM_\ell(4,2;(1),(2^k)) = \bM_\ell(2,1;(1^2),(2^{2k}))$ $(k \in \Z_{\ge 1})$ as sets. They should be distinguished since the former has a $C_4$-action $\sqact_{4,2}$, while the latter has a $C_2$-action $\sqact_{2,1}$, which is different from $(\sqact_{4,2})^2$.
\end{remark}

In the following, we will extend $\Uppsi$ in~\eqref{eq: funcPW} to $\bM_\ell(rd,d;\bnu,\bnu')$. The resulting map is denoted by $\funcMtoW$.
\begin{definition}\label{def: def of funcMtoW}
Let $\bm \in \bM_\ell(rd,d;\bnu,\bnu')$. Break $\bm$ into $\bm_{\le krd-1} \in \bM_l(rd;\bnu)$ and $\bm_{\ge krd} \in \bM_{l'}(d;\bnu')$, where $\ell = l + l'$. Let $\bw^{(1)} = \Uppsi(\bm_{\le krd-1} )$ and $\bw^{(2)} =  \Uppsi(\bm_{\ge krd})$. Define
\begin{align}\label{eq: definition of Psi}
\funcMtoW(\bm) = \bw^{(1)} * \blo * \bw^{(2)}.
\end{align}
Notice that $\blo$ denotes the $(krd)$-th zero in $\funcMtoW(\bm)$ when read from left to right.
\end{definition}

Using the injectivity of $\Uppsi$, one can easily see that $\funcMtoW$ is injective. Indeed, the inverse of $\funcMtoW$ is defined in the following way: For an element $\bw \in \funcMtoW(\bM_\ell(rd,d;\bnu,\bnu'))$, break it into $\bw^{(1)} *\blo* \bw^{(2)}$, where $\blo$ denotes the $krd$th zero. Then
\begin{align}\label{eq: definition inverse of Psi}
\funcMtoW^{-1}(\bw) \seteq \Uppsi^{-1}(\bw^{(1)}) * \Uppsi^{-1}(\bw^{(2)}).
\end{align}

\begin{convention}
Hereafter the blue zero $\blo$ will denote the $krd$-th zero in $\funcMtoW(\bm)$ when we read it from left to right.
\end{convention}

\begin{example}
Let $d=1,k=1,k'=2,r=2,\ell = 6,\bnu=(1)$, and $\bnu' = (1,2)$. Then
\[\bM_{6}(2,1;(1),(1,2)) = \left\{ \mathbf{m} \in \Z_{\ge 0}^{4} \; \middle| \; (m_0 + m_1) + m_2 + 2m_3 = 6 \right\}, \]
where $\mathbf{m} = (m_0,m_1,m_2,m_3)$. For $\bm = (1,2,1,1) \in \bM_{6}(2,1;(1),(1,2))$, break $\bm$ into
\[\bm_{\le 1} = (1,2) \in \bM_{3}(2;(1)) \quad \text{and} \quad \bm_{\ge 2} = (1,1) \in \bM_{3}(1;(1,2)). \]
Since $\bm_{\le 1} = (1,2) \in \bM_{3}(2;(1))$ (resp. $\bm_{\ge 2} = (1,1) \in \bM_{3}(1;(1,2))$), we have $\Uppsi((1,2)) = 1011$ (resp. $\Uppsi((1,1)) = 102$). Thus we have
\[ \funcMtoW((1,2,1,1)) = 1011\blo102. \]
\end{example}

\section{Cyclic sieving phenomena(except for $D_n^{(1)}(n\equiv_2 0)$)} \label{Sec: CSP1}
The {\it cyclic sieving phenomenon} was introduced by Reiner-Stanton-White in \cite{RSW}.
Let $X$ be  a finite set, with an action of a cyclic group $C$ of order $m$. Elements within a $C$-orbit share the same stabilizer subgroup, whose cardinality we will call the \emph{stabilizer-order} for the orbit.
Let $X(q)$ be a polynomial in $q$ with nonnegative integer coefficients.
For $d \in \Z_{>0}$, let $\omega_d$ be a $d$th primitive root of the unity.
We say that $(X, C, X(q))$ exhibits the  cyclic sieving phenomenon if, for all $c\in C$, we have
$$
|X^c| = X(\omega_{o(c)}),
$$
where $o(c)$ is the order of $c$ and $X^c$ is the fixed point set under the action of $c$.
Note that this condition is equivalent to the following:
$$
X(q)\equiv \sum_{0 \le i \le m-1}  b_i  q^i \pmod {q^m-1},
$$
where $b_i$ counts the number of $C$-orbits on $X$ for which the stabilizer-order divides $i$.

In this section, we suppose that $\g$ is any affine Kac-Moody algebra of rank $n$ except for $D_n^{(1)} (n\equiv_2 0)$ and $A_{2n}^{(2)}$.
For a nonnegative integer $\ell$, let $X=\pclp{\ell}$, $C = C_{\congN}$, with $\congN$  in~\eqref{eq: congN}, and we define
\begin{align}\label{eq: definition of X(q)}
X(q) = \pclp{\ell}(q) \seteq \sum_{\Lambda \in \pclp{\ell}} q^{\evS(\Lambda)},
\end{align}
where $\svs$ is the root-sieving set given in Convention~\ref{conv: Sieving}.  Then   Theorem~\ref{thm: eq rel and sieving} tells that
\begin{equation}\label{eq: level and series C equiv cond}
\begin{aligned}
\pclp{\ell}(q) &= \sum_{i \ge 0} \left| \left\{\Lambda \in \pclp{\ell} \; \middle| \; \evS(\Lambda) = i
\right\} \right|q^i \equiv \sum_{\Lambda \in  \DRpclp } \left|\pclp{\ell} (\Lambda) \right|q^{ \evS(\Lambda)}   \pmod{q^\congN-1}. 
\end{aligned}
\end{equation}

\begin{remark}
Let $\svv = (s_1,s_2,\ldots,s_n)$ and $s_0 \seteq 0$. Then $\pclp{\ell}(q)$ can be defined by the geometric series:
\begin{align} \label{eq: ge series not Dn1}
\sum_{\ell \ge 0} \pclp{\ell}(q)t^\ell & = \prod_{0 \le i \le n} \dfrac{1}{1-q^{s_i}t^{a_i^\vee}}.
\end{align}
Note that the coefficient of $t^\ell$ of the right hand side is given by
\[\sum_{j \ge 0} \left|\left\{\sum_{0\le i \le n}m_i\Lambda_i \; \middle| \; \sum_{0 \le i \le n} a_i^{\vee} m_i = \ell \text{ and } \sum_{0 \le i \le n} s_i m_i = j \right\}  \right|q^j. \]
\end{remark}

The purpose of this section is to show that $(\pclp{\ell}, C_{\congN}, \pclp{\ell}(q))$ exhibits the cyclic sieving phenomenon.

\medskip

Let us introduce some necessary notations.
When a finite group $G$ acts on $X$, we denote by $X^G$ the set of fixed points under the action of $G$.
For any $g \in G$, we let $X^g \seteq X^{\inn{g}}$.
For $n \in \Z_{\ge 0}$ and $k  \in \Z_{\ge 0} \cap \Z_{\le n}$, we let 
{\it $q$-binomial coefficient} which are defined as follows:
\[ [n]_q:=\dfrac{q^n-1}{q-1}, \quad [n]_q! \seteq \prod_{1 \le k \le n} [k]_q, \quad \text{ and } \quad \left[ \begin{matrix} n \\ k \end{matrix} \right]_q := \dfrac{[n]_q!}{[k]_q![n-k]_q!}. \]

Let $\clfw = \{\Lambda_0 , \Lambda_1, \ldots, \Lambda_n\}$ and let
\begin{align} \label{eq: funcPcltoZ}
\funcPcltoZ: \Pclp \ra \Z^{n+1}, \quad \Lambda \mapsto [\Lambda]_{\clfw}
\end{align}
be the map given by the matrix representation in terms of $\clfw$.

\subsection{$A_{n}^{(1)}$ type}

To begin with, we review a result in ~\cite{RSW}.
For a positive integer $N$, let $[0,N]:=\{0,1,2,\ldots,N\}$ and ${[0,N] \choose \ell} $ the set of all $\ell$-multisubsets of $[0,N]$.
Then the symmetric group $\Sym_{[0,N]}$ on $[0,N]$ acts on ${[0,N] \choose \ell}$. 
We say that a cyclic group $C$ of order $m$ acts {\it nearly freely} on $[0,N]$ if it is generated by a permutation $c\in \Sym_{[0,N]}$
whose cycle type is either
\begin{align}\label{eq: CSP in RSW}
\hspace{-15ex} \begin{array}{l}
\text{(1) $j$ cycles of size $m$ so that $N+1=jm$, or}\\
\text{(2) $j$ cycles of size $m$ and one singleton cycle, so that $N+1=jm+1$}
\end{array}
\end{align}
for some positive integer $j$.

\begin{lemma} \cite[Theorem 1.1 (a)]{RSW} \label{thm: CSP claissic 1}
Let a cyclic group $C$ of order $m$ act nearly freely on $[0,N]$.   Then
the triple $$\left({[0,N] \choose \ell},C,\left[ \begin{matrix} N+\ell \\ \ell \end{matrix} \right]_q \right)$$ exhibits the cyclic sieving phenomenon.
\end{lemma}

Recall that $\congN=n+1$ and $\langle c,\Lambda_i \rangle =1$ for all $i \in I$ when $\g=A^{(1)}_{n}$. Let us define a $C_{n+1}$-action on $\pclp{\ell}$ by
\begin{align}\label{eq: sigma action}
\siggen_{n+1} \cdot \sum_{0 \le i \le n} m_i \Lambda_i = \sum_{0 \le i \le n} m_{i+1} \Lambda_{i}, \qquad \text{where $m_{n+1} = m_0$.}
\end{align}

On the other hand, for the long cycle $\upsig=(0,1,2,\ldots, n) \in \Sym_{[0,n]}$ of order $n+1$,
the cyclic group $\inn{\upsig}$ acts  freely on ${[0,n] \choose \ell}$. For simplicity, let us use $0^{m_0}1^{m_1}\cdots n^{m_n}$ to denote the multiset with $m_i$ $i$'s for all $0\le i \le n$.
There is a natural bijection, say $\kappa$, between $\pclp{\ell}$ and ${[0,n] \choose \ell}$
\begin{align}\label{eq: kappa bijection}
\kappa: {\pclp{\ell}} \to \matr{[0,n]}{\ell}, \quad  \sum_{0 \le i \le n} m_i\Lambda_i \mapsto 0^{m_0}1^{m_1}\cdots n^{m_n}
\end{align}
preserving group actions, more precisely, satisfying that
\begin{align}\label{eq: group action preserve}
\kappa\left(\siggen_{n+1} \cdot \sum_{0 \le i \le n} m_i\Lambda_i\right) = \upsig \cdot 0^{m_0}1^{m_1}\cdots n^{m_n} ~( = 0^{m_n} 1^{m_0}\cdots n^{m_{n-1}} ).
\end{align}
Moreover, there is a bijection between $\pclp{\ell}$ and the set ${\rm Par}(n,\ell)$ of partitions contained in $n\times \ell$ rectangle defined by
\begin{align}\label{eq: wt par corres}
    \sum_{0 \le i \le n} m_i \Lambda_i \mapsto (n^{m_{n}} (n-1)^{m_{n-1}} \cdots 1^{m_1})'.
\end{align}
Here $(n^{m_{n}} (n-1)^{m_{n-1}} \cdots 1^{m_1})$ denotes the partition having $m_i$ part equal to $i~(1\le i \le n)$ and $\mu'$ the conjugate of $\mu$ for any partition $\mu$. It can be easily seen that $\evS(\Lambda)$ is equal to the size of the corresponding partition, thus \cite[Proposition 1.7.3]{EC1} says that
\begin{align}\label{eq: pclp An1 type}
\pclp{\ell}(q)&= \sum_{i \ge 0} \left| \left\{ \Lambda\in \pclp{\ell} \; \middle| \; \evS(\Lambda) = i \right\} \right|q^i
=  \sum_{i \ge 0} \left| \left\{ \lambda \in {\rm Par}(n,\ell) \; \middle| \; |\lambda| = i \right\} \right|q^i  = \left[ \begin{matrix} n + \ell \\ \ell \end{matrix} \right]_q.
\end{align}
\begin{theorem}\label{A TYPE csp}
Under the $C_{n+1}$-action in~\eqref{eq: sigma action}, the triple
\begin{align}\label{eq: CSP for An1}
\left(\pclp{\ell}, C_{n+1}, \pclp{\ell}(q) \right)
\end{align}
exhibits the cyclic sieving phenomenon.
\end{theorem}
\begin{proof}
Our assertion follows from Lemma \ref{thm: CSP claissic 1}, ~\eqref{eq: kappa bijection} and~\eqref{eq: pclp An1 type}.
\end{proof}
\begin{remark}\label{rem: dom wt and par}
    For $0 \le i \le n$, the image of $\pclp{\ell}((\ell-1)\Lambda_0 + \Lambda_i)$ under the correspondence in~\eqref{eq: wt par corres} is $\left\{ \lambda \in {\rm Par}(n,\ell) \; \middle| \; |\lambda| \equiv_{n+1} i \right\}$, which yields the identity:
\[
|\mx^+((\ell-1)\Lambda_0 + \Lambda_i)| = | \left\{ \lambda \in {\rm Par}(n,\ell) \; \middle| \; |\lambda| \equiv_{n+1} i \right\} |. 
 \]
This shows that $|\mx^+((\ell-1)\Lambda_0 + \Lambda_i)|~(0\le i \le n)$ appear as the coefficient of $\pclp{\ell}(q)$, more precisely,
    \begin{equation}\label{eq: coeff of pclp(q)}
    \begin{aligned}
    \pclp{\ell}(q)
    & \equiv \sum_{0 \le i \le n} |\mx^+((\ell-1)\Lambda_0 + \Lambda_i)|q^i \pmod{q^{n+1}-1}.
    \end{aligned}
    \end{equation}
\end{remark}

The congruence \eqref{eq: coeff of pclp(q)} implies that $|\mx^+((\ell-1)\Lambda_0 + \Lambda_i)|$ counts the $C_{n+1}$-orbits on $\pclp{\ell}$ for which the stabilizer-order divides $i$. In the following, we will give a closed formula for the number of $C_{n+1}$-orbits on $\pclp{\ell}$ for which the stabilizer-order divides $i$.

A $C_{n+1}$-orbit of $\pclp{\ell}$ can be considered as \emph{necklace} with black beads and white beads.
Fix $n,\ell \in \Z_{>0}$. Let $\beads(n,\ell)$ be the set of words of $n$ white beads and $\ell$ black beads. The cyclic group $C_{n+1} = \inn{\siggen_{n+1}}$ acts on $\beads(n,\ell)$ by
\begin{align*}
&\siggen_{n+1} \cdot \left (\underbracket{B,B,\ldots,B}_{m_0},W,\underbracket{B,B,\ldots,B}_{m_1},W,\ldots,W,\underbracket{B,B,\ldots,B}_{m_n} \right)\\
=& \left(\underbracket{B,B,\ldots,B}_{m_n},W,\underbracket{B,B,\ldots,B}_{m_0},W,\underbracket{B,B,\ldots,B}_{m_1},W,\ldots,W,\underbracket{B,B,\ldots,B}_{m_{n-1}} \right), \nonumber
\end{align*}
where $W$ denotes a white bead and $B$ denotes a black bead. We can realize a $C_{n+1}$-orbit of $\beads(n,\ell)$ as a necklace using $n+1$ white beads and $\ell$ black beads by (i) threading the beads into a necklace with the same order,
and (ii) adding a white bead between the last $B$ and the first $B$ as follows:
For $\scrS \seteq \underbracket{B,B,\ldots,B}_{m_0},W,\underbracket{B,B,\ldots,B}_{m_1},W,\ldots,W,\underbracket{B,B,\ldots,B}_{m_n} \in \beads(n,\ell)$,
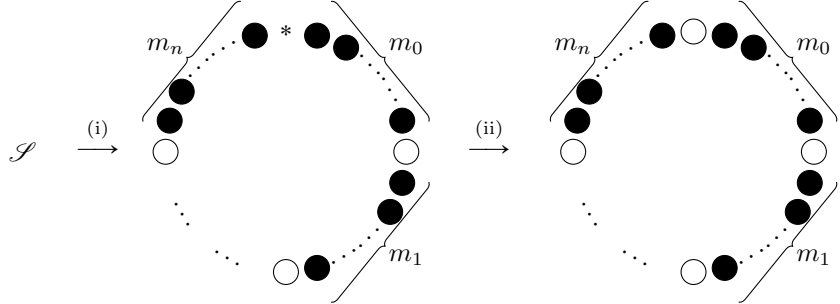
\begin{align*}
\begin{tikzpicture}
    \node[circle,draw,fill = black]     at ({1*(360/24)}:1.6){};
    \node[rotate = -60]                 at ({2*(360/24)}:1.6){$\cdots$};
    \node[rotate = -45]                 at ({3*(360/24)}:1.6){$\cdots$};
    \node[circle,draw,fill = black]     at ({4*(360/24)}:1.6){};
    \node[circle,draw,fill = black]     at ({5*(360/24)}:1.6){};
    \node                               at ({6*(360/24)}:1.6){$*$};
    \node[circle,draw,fill = black]     at ({7*(360/24)}:1.6){};
    \node[rotate = 30]                  at ({8*(360/24)}:1.6){$\cdots$};
    \node[rotate = 45]                  at ({9*(360/24)}:1.6){$\cdots$};
    \node[circle,draw,fill = black]     at ({10*(360/24)}:1.6){};
    \node[circle,draw,fill = black]     at ({11*(360/24)}:1.6){};
    \node[circle,draw]                  at ({12*(360/24)}:1.6){};
    \node[rotate = 120]                             at ({14*(360/24)}:1.6){$\cdots$};
    \node[rotate = 150]                 at ({16*(360/24)}:1.6){$\cdots$};
    \node[circle,draw]                  at ({18*(360/24)}:1.6){};
    \node[circle,draw,fill = black]     at ({19*(360/24)}:1.6){};
    \node[rotate = 30]                  at ({20*(360/24)}:1.6){$\cdots$};
    \node[rotate = 45]                  at ({21*(360/24)}:1.6){$\cdots$};
    \node[circle,draw,fill = black]     at ({22*(360/24)}:1.6){};
    \node[circle,draw,fill = black]     at ({23*(360/24)}:1.6){};
    \node[circle,draw]                  at ({24*(360/24)}:1.6){};
    \draw[decorate,decoration=brace] (0.6,2) -- (1.9,0.4);
    \node at (1.6,1.4) {$m_0$};
    \draw[decorate,decoration={brace,mirror}] (0.6,-2) -- (1.9,-0.4);
    \node at (1.6,-1.4) {$m_1$};
    \draw[decorate,decoration={brace,mirror}] (-0.6,2) -- (-1.9,0.4);
    \node at (-1.6,1.4) {$m_n$};
    \node at (2.7,0) {$\longrightarrow$};
    \node at (2.7,0.3) {\scriptsize (ii)};
    \node at (-2.5,0) {$\longrightarrow$};
    \node at (-2.5,0.3) {\scriptsize (i)};
\node at (-3.5,0) {$\scrS$};
\end{tikzpicture}\quad
\begin{tikzpicture}
\node[circle,draw,fill = black]     at ({1*(360/24)}:1.6){};
\node[rotate = -60]                 at ({2*(360/24)}:1.6){$\cdots$};
\node[rotate = -45]                 at ({3*(360/24)}:1.6){$\cdots$};
\node[circle,draw,fill = black]     at ({4*(360/24)}:1.6){};
\node[circle,draw,fill = black]     at ({5*(360/24)}:1.6){};
\node[circle,draw]                          at ({6*(360/24)}:1.6){};
\node[circle,draw,fill = black]     at ({7*(360/24)}:1.6){};
\node[rotate = 30]                  at ({8*(360/24)}:1.6){$\cdots$};
\node[rotate = 45]                  at ({9*(360/24)}:1.6){$\cdots$};
\node[circle,draw,fill = black]     at ({10*(360/24)}:1.6){};
\node[circle,draw,fill = black]     at ({11*(360/24)}:1.6){};
\node[circle,draw]                  at ({12*(360/24)}:1.6){};
\node[rotate = 120]                             at ({14*(360/24)}:1.6){$\cdots$};
\node[rotate = 150]                 at ({16*(360/24)}:1.6){$\cdots$};
\node[circle,draw]                  at ({18*(360/24)}:1.6){};
\node[circle,draw,fill = black]     at ({19*(360/24)}:1.6){};
\node[rotate = 30]                  at ({20*(360/24)}:1.6){$\cdots$};
\node[rotate = 45]                  at ({21*(360/24)}:1.6){$\cdots$};
\node[circle,draw,fill = black]     at ({22*(360/24)}:1.6){};
\node[circle,draw,fill = black]     at ({23*(360/24)}:1.6){};
\node[circle,draw]                  at ({24*(360/24)}:1.6){};
\draw[decorate,decoration=brace] (0.6,2) -- (1.9,0.4);
\node at (1.6,1.4) {$m_0$};
\draw[decorate,decoration={brace,mirror}] (0.6,-2) -- (1.9,-0.4);
\node at (1.6,-1.4) {$m_1$};
\draw[decorate,decoration={brace,mirror}] (-0.6,2) -- (-1.9,0.4);
\node at (-1.6,1.4) {$m_n$};
\end{tikzpicture}
\end{align*}
A necklace $C_{n+1} \cdot \scrS$ is called \emph{primitive} if the stabilizer subgroup of $\scrS$ is trivial.

\begin{lemma}[{\cite[Theorem 7.1]{Reu}}]\label{thm: Necklaces} The number of primitive necklaces using $n+1$ white beads and $\ell$ black beads is
    \[\dfrac{1}{(n+1)+\ell} \sum_{d \mid (n+1,\ell)} \mu(d) \matr{((n+1)+\ell)/d}{\ell/d},\]
    where $\mu$ is the classical M\"{o}bius function.
\end{lemma}

There is a natural $C_{n+1}$-set isomorphism between $\pclp{\ell}$ and $\beads(n,\ell)$ defined by
\begin{align*}
(m_0,m_1,\ldots,m_n) \longleftrightarrow \underbracket{B,B,\ldots,B}_{m_0},W,\underbracket{B,B,\ldots,B}_{m_1},W,\cdots,W,\underbracket{B,B,\ldots,B}_{ m_{n}}.
\end{align*}
Combining Theorem \ref{A TYPE csp} with \eqref{eq: coeff of pclp(q)} and Lemma~\ref{thm: Necklaces}, we derive the following closed formula.

\begin{theorem}\label{thm: number of dom wt A}
For any $(\ell-1)\Lambda_0 + \Lambda_i \in \DRpclp$, we have
\begin{align}\label{eq: explicit formula for |mx|}
|\mx^+((\ell-1)\Lambda_0 + \Lambda_i)| =
\sum_{d \mid (n+1,\ell,i)} \dfrac{d}{(n+1)+\ell} \sum_{d'|(\frac{n+1}{d},\frac{\ell}{d})} \mu(d') \matr{((n+1)+\ell)/dd'}{\ell/dd'}.
\end{align}
\end{theorem}

\begin{remark}
In \cite{JM}, Jayne-Misra conjectured that
\begin{align*}
|\mx^+(\ell\Lambda_0)| = \dfrac{1}{(n+1)+\ell} \sum_{d|(n+1,\ell)} \varphi(d) \matr{((n+1)+\ell)/d}{\ell/d},
\end{align*}
where $\varphi$ is Euler's phi function. This is the case where $i = 0$ in~\eqref{eq: explicit formula for |mx|}, which was proven in \cite{TW}.
\end{remark}

It should be remarked that the cardinality of $\left\{ \lambda \in {\rm Par}(n,\ell) \; \middle| \; |\lambda| \equiv_{n+1} i \right\}$ have already appeared in~\cite{EJP}. Hence, we can also derive Theorem~\ref{thm: number of dom wt A} using Remark~\ref{rem: dom wt and par}.

\begin{corollary}
    Let $(\ell-1)\Lambda_0 + \Lambda_i, (\ell-1)\Lambda_0 + \Lambda_j \in \DRpclp$. Then
    \[ |\mx^+((\ell-1)\Lambda_0 + \Lambda_i)| = |\mx^+((\ell-1)\Lambda_0 + \Lambda_j)| \]
    if and only if $(n+1,\ell,i) = (n+1,\ell,j)$.
\end{corollary}
\begin{proof}
It is a direct consequence of Theorem~\ref{thm: number of dom wt A}.
\end{proof}

\subsection{$B_{n}^{(1)},C_{n}^{(1)},A_{2n-1}^{(2)},D_{n+1}^{(2)},E_6^{(1)}, E_{7}^{(1)}$ types}\label{subsec : CSPs}
In this subsection, we assume that $\g$ is an affine Kac-Moody algebra other than $A_n^{(1)}$ and $D_n^{(1)}$. In $A_n^{(1)}$-case, $\congN$ is a composite unless $n+1$ is a prime. In $D_n^{(1)}$-case, $\congN=4$ is a composite for all $n \in \Z_{\ge 4}$.

Note that $\mathfrak{S}_{[0,n]}$ acts on $\pclp{\ell}$ by permuting indices of coefficients, that is,
$$  \sigma \cdot \sum_{0 \le i \le n} m_i\Lambda_i = \sum_{0 \le i \le n} m_{\sigma(i)}\Lambda_i \quad \text{for $\sigma \in \mathfrak{S}_{[0,n]}$}.$$

Let us take  $\upsig \in \mathfrak{S}_{[0,n]}$ of order $\congN$ as in Table~\ref{table: Cact}.
\begin{table}[h]
    \centering
    { \arraycolsep=1.6pt\def\arraystretch{1.5}
        \begin{tabular}{c|c|c}
            Types & $\upsig$  & $\congN$  \\ \hline
            $B_{n}^{(1)}, D_{n+1}^{(2)}$ & $(0,n)$ & 2   \\ \hline
            $C_{n}^{(1)}, A_{2n-1}^{(2)} ~(n\equiv_2 1)$ & $(0,1)(2,3)\cdots(n-1,n)$ &2    \\ \hline
            $C_{n}^{(1)}, A_{2n-1}^{(2)}~(n\equiv_2 0)$ & $(0,1)(2,3)\cdots(n-2,n-1)$  &2 \\ \hline
            $E_{6}^{(1)}$ & $(0,1,6)(2,3,5)$ &3  \\ \hline
            $E_{7}^{(1)}$ & $(0,7)(1,6)(3,5)$ &2
        \end{tabular}
    }\\[1.5ex]
    \caption{$\upsig$ and $\congN$ for other types}
    \protect\label{table: Cact}
\end{table}
We define a $C_\congN = \inn{\siggen_\congN}$-action on $\pclp{\ell}$ by
\begin{align}\label{eq: CSP action}
\siggen_\congN \cdot \sum_{0 \le i \le n} m_i \Lambda_i \seteq \sum_{0 \le i \le n} m_{\upsig(i)} \Lambda_i
\end{align}
for any $\sum_{0 \le i \le n} m_i \Lambda_i \in \pclp{\ell}$.

\begin{theorem}\label{thm: CSP}
    Let $\g$ be of type $B_{n}^{(1)},C_{n}^{(1)},A_{2n-1}^{(2)},D_{n+1}^{(2)},E_6^{(1)}$, or $E_{7}^{(1)}$. Then, under the $C_\congN$-action given in~\eqref{eq: CSP action}, the triple
    \begin{align}\label{eq: CSP}
    \left(\pclp{\ell}, C_{\congN}, \pclp{\ell}(q) \right)
    \end{align}
    exhibits the cyclic sieving phenomenon.
\end{theorem}

Since the method of proof for each type is essentially same, we only deal with $E_6^{(1)}$ type. Recall that
$$(a_i^\vee)_{i=0}^6 = (1,1,2,2,3,2,1), \quad \wtsvv = (s_i)_{i=0}^6 = (0,1,0,2,0,1,2)  \quad \text{ and } \quad   \sigma=(0,1,6)(2,3,5).$$
Then one can see that, for all $j = 0,1,\ldots, 6$, we have
\begin{align}\label{eq: E61 Cact reason}
a_j^\vee = a_{\upsig(j)}^\vee \quad \text{and} \quad \{s_j, s_{\upsig(j)}, s_{\upsig^2(j)}\} = \begin{cases}
\{0,1,2\}  & \text{ if $i \neq 4$,}\\
\{0\} & \text{ if $i=4$.}
\end{cases}
\end{align}
By Theorem~\ref{thm: eq rel and sieving} and ~\eqref{eq: level and series C equiv cond}, we have
\begin{equation}\label{eq: E61 dwt_ell(q)}
\begin{aligned}
\pclp{\ell}(q) &= \sum_{i \ge 0} \left| \left\{ \Lambda \in \pclp{\ell} \; \middle| \; \evS(\Lambda) = i \right\} \right|q^i  \equiv \sum_{0 \le i \le 2} \left| \left\{ \Lambda \in \pclp{\ell} \; \middle| \; \evS(\Lambda) \equiv_3 i \right\} \right|q^i \pmod{q^3-1}.
\end{aligned}
\end{equation}

We will prove Theorem~\ref{thm: CSP} by providing a set $X$ with a $C_3$-action such that $X$ is isomorphic to $\pclp{\ell}$ as $C_3$-sets and $(X,C_3,\pclp{\ell}(q))$ exhibits the cyclic sieving phenomenon.
More precisely, we will take $X$ as $\bM_\ell(3,1;(1,2),(3))$.

Recall that $\Sym_{[0,6]}$ acts on $\Pclp$ by permuting indices of coefficients. Let $\tau = (4,3,2,6) \in \Sym_{[0,6]}$. Since
\[\tau \cdot (a_0^\vee,a_1^\vee,a_2^\vee,a_3^\vee,a_4^\vee,a_5^\vee,a_6^\vee)
 =(a_{\tau(0)}^\vee,a_{\tau(1)}^\vee,a_{\tau(2)}^\vee,a_{\tau(3)}^\vee,a_{\tau(4)}^\vee,a_{\tau(5)}^\vee,a_{\tau(6)}^\vee),\]
we have
\[\tau \cdot (1,1,2,2,3,2,1) = (1,1,1 \ | \ 2,2,2 \ | \ 3).\]
Thus the image of $\tau \cdot \pclp{\ell}$ under $\funcPcltoZ$ in~\eqref{eq: funcPcltoZ} is the same as $\bM_\ell(3,1;(1,2),(3))$. For the definition of $\bM_\ell(3,1;(1,2),(3))$, see~\eqref{eq: def of M}.
\begin{example}
Let $\g = E_6^{(1)}$ and $\ell = 3$. Then, we have
\begin{align*}
\pclp{3} = \left\{
\begin{array}{lll}
3\Lambda_0,                         &\Lambda_0 + \Lambda_2, & \Lambda_4, \\
2\Lambda_0 + \Lambda_1,             &\Lambda_0 + \Lambda_3,\\
2\Lambda_0 + \Lambda_6,             &\Lambda_0 + \Lambda_5,\\
\Lambda_0 + 2\Lambda_1,             &\Lambda_1 + \Lambda_2,\\
\Lambda_0 + \Lambda_1 + \Lambda_6,  &\Lambda_1 + \Lambda_3,\\
\Lambda_0 + 2\Lambda_6,             &\Lambda_1 + \Lambda_5,\\
3\Lambda_1,                         &\Lambda_2 + \Lambda_6,\\
2\Lambda_1 + \Lambda_6,             &\Lambda_3 + \Lambda_6,\\
\Lambda_1 + 2\Lambda_6,             &\Lambda_5 + \Lambda_6,\\
3\Lambda_6
\end{array}
\right\}
\quad \text{and} \quad
\tau \cdot \pclp{3} = \left\{
\begin{array}{lll}
3\Lambda_{0},                           &\Lambda_{0} + \Lambda_{3}, & \Lambda_{6}, \\
2\Lambda_{0} + \Lambda_{1},             &\Lambda_{0} + \Lambda_{4},\\
2\Lambda_{0} + \Lambda_{2},             &\Lambda_{0} + \Lambda_{5},\\
\Lambda_{0} + 2\Lambda_{1},             &\Lambda_{1} + \Lambda_{3},\\
\Lambda_{0} + \Lambda_{1} + \Lambda_{2},    &\Lambda_{1} + \Lambda_{4},\\
\Lambda_{0} + 2\Lambda_{2},             &\Lambda_{1} + \Lambda_{5},\\
3\Lambda_{1},                           &\Lambda_{2} + \Lambda_{3},\\
2\Lambda_{1} + \Lambda_{2},             &\Lambda_{2} + \Lambda_{4},\\
\Lambda_{1} + 2\Lambda_{2},             &\Lambda_{2} + \Lambda_{5},\\
3\Lambda_{2}
\end{array}
\right\}.
\end{align*}
Note that the image of $\tpclp{3}$ under $\funcPcltoZ$ in~\eqref{eq: funcPcltoZ} is
\begin{align*}
\bM_3(3,1;(1,2),(3))
= \left\{
\begin{array}{lll}
(3,0,0,0,0,0,0), & (1,0,0,1,0,0,0), & (0,0,0,0,0,0,1), \\
(2,1,0,0,0,0,0), & (1,0,0,0,1,0,0), \\
(2,0,1,0,0,0,0), & (1,0,0,0,0,1,0),\\
(1,2,0,0,0,0,0), & (0,1,0,1,0,0,0),\\
(1,1,1,0,0,0,0), & (0,1,0,0,1,0,0),\\
(1,0,2,0,0,0,0), & (0,1,0,0,0,1,0),\\
(0,3,0,0,0,0,0), & (0,0,1,1,0,0,0),\\
(0,2,1,0,0,0,0), & (0,0,1,0,1,0,0),\\
(0,1,2,0,0,0,0), & (0,0,1,0,0,1,0),\\
(0,0,3,0,0,0,0)
\end{array}
 \right\}.
\end{align*}
\end{example}
In Table~\ref{table: tau and tuset}, we list $\tau$ and $\bM_\ell(rd,d;\bnu,\bnu')$ for all types 
(except for $A_n^{(1)}$ and $D_n^{(1)}$).
\begin{remark} \label{rem: choice}  \hfill
\begin{enumerate}[label = (\arabic*)]
\item All $\bM_\ell(rd,d;\bnu,\bnu')$ in Table~\ref{table: tau and tuset} are contained in $\Z_{\ge 0}^{n+1}$ as subsets.
\item In Table~\ref{table: tau and tuset}, we choose $\tau,r,d,\bnu,\bnu'$ to be satisfied that the image of $\tau \cdot \pclp{\ell}$ under $\funcPcltoZ$ in~\eqref{eq: funcPcltoZ} is the same as $\bM_\ell(rd,d;\bnu,\bnu')$.
\end{enumerate}
\end{remark}

\begin{table}[h]
    \centering
    { \arraycolsep=1.6pt\def\arraystretch{1.5}
        \begin{tabular}{c|c|c}
            Types &  $\tau$ & $\bM_\ell(rd,d;\bnu,\bnu')$ \\ \hline
            $B_{n}^{(1)}$ &  $(n,n-1,\ldots,1)$ & $\bM_\ell(2,1;(1),(1,2^{n-2}))$   \\ \hline
            $C_{n}^{(1)}~(n\equiv_2 1)$ & $\mathrm{id}$ & $\bM_\ell(2;(1^{(n+1)/2}))$   \\ \hline
            $C_{n}^{(1)}~(n\equiv_2 0)$ &  $\mathrm{id}$ & $\bM_\ell(2,1;(1^{n/2}),(1))$   \\ \hline
            $A_{2n-1}^{(2)}~(n\equiv_2 1)$  & $\mathrm{id}$  & $\bM_\ell(2;(1,2^{(n-1)/2}))$  \\ \hline
            $A_{2n-1}^{(2)}~(n\equiv_2 0)$ & $\mathrm{id}$ & $\bM_\ell(2,1;(1,2^{(n-2)/2}),(2))$   \\ \hline
            $D_{n+1}^{(2)}$ & $(n,n-1,\ldots,1)$ & $\bM_\ell(2,1;(1),(2^{n-1}))$ \\ \hline
            $E_{6}^{(1)}$ & $(4,3,2,6)$ & $\bM_\ell(3,1;(1,2),(3))$  \\ \hline
            $E_{7}^{(1)}$ & $(1,7,4,3,2,6)$ & $\bM_\ell(2,1;(1,2,3),(2,4))$
        \end{tabular}
    }\\[1.5ex]
    \caption{$\tau$ and $\bM_\ell(rd,d;\bnu,\bnu')$ for all types}
    \protect\label{table: tau and tuset}
\end{table}

Now, we have a $C_3$-action $\sqact_{3,1}$ on $\tau \cdot \pclp{\ell}$ defined as follows: For $\Lambda \in \tau \cdot \pclp{\ell}$,
\begin{equation}\label{eq: E61 action on tau}
\begin{aligned}
\siggen_3 \sqact_{3,1} \Lambda & = \funcPcltoZ^{-1}( \siggen_3 \sqact_{3,1} \funcPcltoZ(\Lambda)) = (\funcMtoW \circ \funcPcltoZ)^{-1} \left( \siggen_3 \sqact_{3,1} (\funcMtoW \circ \funcPcltoZ)(\Lambda)\right).
\end{aligned}
\end{equation}

\begin{example}
Let $\g = E_6^{(1)}$. For $2\Lambda_0 + \Lambda_1 \in \tpclp{3}$, we have the following commutative diagram:
\[
\begin{tikzcd}
2\Lambda_0 + \Lambda_1 \arrow[r, "\funcPcltoZ"] \arrow[d, "\siggen_3", dotted] & {(2,1,0,0,0,0)} \arrow[r, "\funcMtoW"] \arrow[d, "\siggen_3"]        & 110|100|00\blo \arrow[d, "\siggen_3"]                             \\
3\Lambda_1 \arrow[d, "\siggen_3", dotted]                                      & {(0,3,0,0,0,0)} \arrow[l, "\funcPcltoZ^{-1}"] \arrow[d, "\siggen_3"] & 011|100|00\blo \arrow[l, "\funcMtoW^{-1}"] \arrow[d, "\siggen_3"] \\
\Lambda_0 + 2\Lambda_1 \arrow[bend left=90, "\siggen_3", dotted]{uu}                          & {(1,2,0,0,0,0)} \arrow[l, "\funcPcltoZ^{-1}"]  & 101|100|00\blo \arrow[l, "\funcMtoW^{-1}"] \arrow[bend right=90,swap, "\siggen_3"]{uu} \\
\end{tikzcd}
\]
\end{example}

\begin{lemma}\label{lem: Sagan CSP E61}
    Under the $C_3$-action $\sqact_{3,1}$ on $\bM_\ell(3,1;(1,2),(3))$ given in \eqref{eq: hsigaction double},
    \[\left(\bM_\ell(3,1;(1,2),(3)), C_3, \pclp{\ell}(q) \right)\]
    exhibits the cyclic sieving phenomenon.
\end{lemma}
\begin{proof}
    Note the following:
    \begin{itemize}
        \item $\left| \bM_\ell(3,1;(1,2),(3))\right| = \left|\pclp{\ell}\right|$.
        \item $C_3$-orbits are of length $1$ or $3$.
        \item For any $\Lambda \in \pclp{\ell}$, $\evS(\Lambda) =  \widehat{\svv} \bullet  \funcPcltoZ(\tau \cdot \Lambda)$, where $\widehat{\svv} \seteq \tau \cdot \wtsvv = (0,1,2,0,1,2,0)$.
    \end{itemize}

    For simplicity, we set $X \seteq \tau \cdot \pclp{\ell}$ and
    $X(i) \seteq \tau \cdot \left\{ \Lambda \in \pclp{\ell} \; \middle| \; \evS(\Lambda) \equiv_3 i \right\}$.
    Suppose that
    the following claims hold (which will be proven in the below):\\[-1ex]

    {\it Claim 1.} Let $\mathbf{m}=(m_0,m_1,\ldots,m_5,m_6) \in \bM_\ell(3,1;(1,2),(3))^{C_3}$. Then
$$ \widehat{\svv} \bullet \mathbf{m}= \widehat{\svv} \bullet  \funcPcltoZ(\tau \cdot \Lambda) = 0$$ (see Remark~\ref{rem: choice}).    \\[-1ex]

    {\it Claim 2.} Let $O = \{\mathbf{m}^{(1)},\mathbf{m}^{(2)},\mathbf{m}^{(3)}\}$ be a free $C_3$-orbit of $\bM_\ell(3,1;(1,2),(3))$. For each $\mathbf{m}^{(j)}~(j=1,2,3)$,
$\whsvv \bullet \mathbf{m}^{(j)}$ are distinct up to modulo $3$.\\[-1ex]

    By {\it Claim 1}, we have
    \begin{align}\label{eq: fixed in ellLambda0}
    X^{C_3} \subset X(0),
    \end{align}
    under the $C_3$-action on $X$ given in~\eqref{eq: E61 action on tau}.

    By {\it Claim 2}, we have
    \begin{align}\label{eq: free oribts distribution}
    \left| \left( X \setminus X^{C_3} \right) \cap X(0) \right|
    & = \left| \left( X \setminus X^{C_3} \right) \cap X (1) \right|
    = \left| \left( X \setminus X^{C_3} \right) \cap X(2) \right|.
    \end{align}
    By~\eqref{eq: fixed in ellLambda0}, for $i = 1,2$, we have
    \[ \left| \left( X \setminus X^{C_3} \right) \cap X (i) \right| =
    \left|X(i) \right|  = \left| \left\{ \Lambda \in \pclp{\ell} \; \middle| \; \evS(\Lambda) \equiv_3 i \right\} \right|,\]
    which is equal to the number of free $C_3$-orbits, by~\eqref{eq: free oribts distribution}. Moreover, since
    \begin{align*}
    \left| \left\{ \Lambda \in \pclp{\ell} \; \middle| \; \evS(\Lambda) \equiv_3 0 \right\} \right| & = \left|X(0) \right|
    = \left| X^{C_3} \right| + \left| \left( X \setminus X^{C_3} \right) \cap X(0) \right|\\
    & = \text{(the number of fixed points)} + \text{(the number of free orbits)}\\
    & = \text{(the number of all orbits)},
    \end{align*}
    our assertion holds.

    To complete the proof, we have only to verify {\it Claim 1} and {\it Claim 2}.
\smallskip

    \noindent (a) For {\it Claim 1}, suppose that $\mathbf{m} \in \bM_\ell(3,1;(1,2),(3))^{C_3}$. Recall the function $\Uppsi$ from~\eqref{eq: funcPW}. Let $\mathbf{w} = w_1 w_2\cdots w_u \seteq \Uppsi(\mathbf{m}_{\le 5})$.
    Break $\mathbf{w}$ into subwords
    \[w^1 \ | \ w^2 \ | \ \cdots \ | \ w^t \ | \ w_{3t+1}\cdots w_u,  \qquad  (t = \lfloor u/3 \rfloor) \]
    of length $3$ as (\ref{eq: breakw}).
    Since $\mathbf{m}$ is a fixed point, Algorithm~\ref{Algorithm} and the definition of $\siggen_3 \sqact$ in~\eqref{eq: Sagact} say that
    \begin{align}\label{eq: fix WWF}
    \mathbf{w} = \underbracket{11\cdots 1}_{3k_1} \; | \; 000 \; | \; \underbracket{22\ldots 2}_{3k_2} \; | \; 00,
    \end{align}
    for $k_1,k_2 \in \Z_{\ge 0}$ such that $\ell - 3k_1 - 6k_2 = 3k_3$ for some $k_3 \in \Z_{\ge 0 }$. Therefore,
    \begin{align}\label{eq: fix dwt}
    \mathbf{m} = (3k_1,0,0,3k_2,0,0,k_3).
    \end{align}
    Thus,
    \[   \widehat{\svv} \bullet \mathbf{m} =    (0,1,2,0,1,2,0)  \bullet (3k_1,0,0,3k_2,0,0,k_3)= 0. \]

    \noindent (b) For {\it Claim 2}, choose any element $\mathbf{m} \in O$. Then we have
    $$\text{$\mathbf{m}_{\le 5} \in \bM_l(3;(1,2))$ and $\mathbf{m}_{6} \in \bM_{l'}(1;(3))$ for some $l$ and $l'$ with  $l + l' = \ell$.}$$
    Since $C_3$ acts on  $\mathbf{m}_{6}$  in a trivial way (see Remark \ref{rem: hsig remark} (2)), it suffices to consider the $C_3$-action on $\mathbf{m}_{\le 5}$.

    Let $\mathbf{w} = w_1 w_2\cdots w_u \seteq \Uppsi(\mathbf{m}_{\le 5})$. Break $\mathbf{w}$ into subwords
    \[w^1 \ | \ w^2 \ | \ \cdots \ | \ w^t \ | \ w_{3t+1}\cdots w_u\]
    of length $3$ as (\ref{eq: breakw}). Since $O$ is a free orbit, there exists the smallest $1 \le j_0 \le t$ such that $\siggen_3 \cdot w^{j_0} \ne w^{j_0}$.
    Note that the definition of the $C_3$-action in~\eqref{eq: hsigaction double} says that, for all $1\le j <j_0$, $w^j$'s are $000,111$, or $222$.

    Now $w^{j_0}$ should be one of
    \begin{align*}
    & 100,~~ 010,~~ 001,\qquad 110,~~ 011,~~ 101,    \qquad  200,~~ 020,~~ 002,\qquad 220,~~ 022,~~ 202.
    \end{align*}
    We shall only give a proof for the case $w^{j_0} = 100$ since the other cases can be proved by the same argument.

    \smallskip

    Note that for all $1\le j <j_0$, $w^j$ is $000$ or $111$ under our assumption. Assume that there is $1\le j <j_0$ such that $w^j = 000$.
    Then, by Algorithm~\ref{Algorithm}, $w^{j_0}$ is not able to contain 1 since $\bm_{\le 5} \in \bM_l(3;(1,2))$. Now  we have $w^j = 111$, for all $1\le j <j_0$.  Then
    \begin{align*}
    \siggen^i_3 \sqact \mathbf{w} &=\hspace{.63em} w^1 \hspace{.93em} | \hspace{.93em} w^2 \hspace{.93em} | \ \cdots \ |\hspace{.29em} w^{j_0-1} \hspace{.29em} | \hspace{.35em} \siggen_3^i \cdot w^{j_0} \hspace{.35em} | \ w^{j_0+1} \ | \ \cdots \ | \ w^t \ | \ w_{3t+1} \cdots w_u\\
    & = \hspace{.3em} 111 \hspace{.6em} \ |\hspace{.6em} 111 \hspace{.6em}\ | \ \cdots \ | \ \hspace{.3em} 111 \hspace{.6em} \ | \  \hspace{.95em} \begin{matrix} 010 \\ 001 \end{matrix}   \hspace{.95em} \ | \ w^{j_0+1} \ | \ \cdots \ | \ w^t \ | \ w_{3t+1}\cdots w_u
\quad \begin{matrix} \text{ if } i=1, \\  \text{ if } i=2. \end{matrix}
    \end{align*}
    Therefore, by the construction of $\Uppsi ^{-1}$, we have
    \begin{align*}
    \siggen^i_3 \sqact_{3,1} \mathbf{m}
    & = \Uppsi ^{-1}(\siggen_3^i \sqact \mathbf{w}) * \mathbf{m}_{6} = \begin{cases} (m_0 - 1, m_1 + 1, m_2 \ \ \ \ \  ,m_3,m_4,m_5,m_6) & \text{ if } i=1, \\  (m_0 - 1, m_1 \ \ \  \ \  , m_2 + 1 ,m_3,m_4,m_5,m_6) & \text{ if } i=2. \end{cases}
    \end{align*}
    Hence, we have
\[
    \widehat{\svv} \bullet (\siggen_3^i \sqact_{3,1} \mathbf{m} - \mathbf{m}) = \begin{cases}   (0,1,2,0,1,2,0)  \bullet (-1,1,0,0,0,0,0)   \equiv_3 1  & \text{ if } i=1, \\    (0,1,2,0,1,2,0)  \bullet (-1,0,1,0,0,0,0) \equiv_3 2   & \text{ if } i=2. \end{cases}
\qedhere \]
\end{proof}

 Now we are ready to prove Theorem \ref{thm: CSP}.

\begin{proof}[Proof for Theorem \ref{thm: CSP}]
Since $|\pclp{\ell}| = |\bM_\ell(3,1;(1,2),(3))|$, by Lemma \ref{lem: Sagan CSP E61}, we have only to see that
\[|(\pclp{\ell})^{C_3}| = |\bM_\ell(3,1;(1,2),(3))^{C_3}|.  \]
Note that
\begin{align}\label{eq: E61 Cact fixed}
\left( \pclp{\ell} \right)^{C_3} = \left\{(k_1,k_1,k_2,k_2,k_3,k_2,k_1) \in \Z_{\ge 0}^{7} \; \middle| \; 3k_1 + 6k_2 + 3k_3 = \ell   \right\}.
\end{align}
We claim that
\begin{align}\label{eq: bM fixed}
\bM_\ell(3,1;(1,2),(3))^{C_3} = \left\{(3k_1,0,0,3k_2,0,0,k_3) \in \Z_{\ge 0}^{7} \; \middle| \; 3k_1 + 6k_2 + 3k_3 = \ell   \right\} =: Y.
\end{align}
By~\eqref{eq: fix WWF} and~\eqref{eq: fix dwt}, $\bM_\ell(3,1;(1,2),(3))^{C_3}$ is contained in $Y$.

For the reverse inclusion, recall the function $\funcMtoW$ from Definition~\ref{def: def of funcMtoW}. For $\bm \in Y$, we have
\[ \funcMtoW(\bm) = \underbracket{1 \cdots 1}_{3k_1} \ | \ 000 \ | \ \underbracket{2 \cdots 2}_{3k_2} \ | \ 0 0 \blo \ | \ \underbracket{3\cdots 3}_{k_3}.\]
Thus, by \eqref{eq: hsigaction double} and~\eqref{eq: definition of Psi}, $\siggen_3 \sqact_{3,1} \bm = \bm$ and hence $Y$ is contained in $\bM_\ell(3,1;(1,2),(3))^{C_3}$.

By~\eqref{eq: E61 Cact fixed} and~\eqref{eq: bM fixed}, $\left( \pclp{\ell} \right)^{C_3}$ and $\bM_\ell(3,1;(1,2),(3))^{C_3}$ have the same cardinality, as required.
\end{proof}

\subsection{$D_n^{(1)}$ type ($n\equiv_2 1$)}\label{subsec: CSP Dn1 odd}
Throughout this subsection, we set
$$\eta \seteq  \frac{n-3}{2},$$
which is an integer since $n$ is odd.
Recall $\congN = 4$,
$$(a_i^\vee)_{i=0}^n = (1,1,2,2,\ldots,2,1,1) \quad \text{and} \quad \wtsvv = (s_i)_{i=0}^n = (0,2,0,2,0,\ldots,0,2,1,3).$$

Let
\[\upsig = (0,n,1,n-1)(2,3)(3,4)\cdots(n-3,n-2) \in \Sym_{[0,n]},\]
which is of order $4$. Since $a_j^\vee = a_{\upsig(j)}^\vee$ for any $j \in I$, we can define a $C_4 = \inn{\siggen_4}$-action on $\pclp{\ell}$ as follows:
\begin{align}\label{eq: C4 action on D_n}
\siggen_4 \cdot \sum_{0 \le i \le n} m_i \Lambda_i \seteq \sum_{0 \le i \le n} m_{\upsig(i)} \Lambda_i \quad \text{for $\sum_{0 \le i \le n} m_i \Lambda_i \in \pclp{\ell}$.}
\end{align}

\begin{theorem}\label{thm: D TYPE csp}
    Under the $C_4$-action given in \eqref{eq: C4 action on D_n},
    \begin{align}\label{eq: CSP for Dn1}
    \left(\pclp{\ell}, C_{4}, \pclp{\ell}(q) \right)
    \end{align}
    exhibits the cyclic sieving phenomenon.
\end{theorem}

For reader's understanding, let us briefly explain our strategy for proving Theorem~\ref{thm: D TYPE csp}. First, as in  Table~\ref{table: tau and tuset}, we take
\begin{align}
\tau = (n-1,n-3,\ldots,4,2,1)(n,n-2,\ldots,5,3)\in \Sym_{[0,n]}
\end{align}
so that
\[
\tau \cdot (a_0^\vee,a_1^\vee,\ldots,a_n^\vee)= \tau \cdot (1,1,2,2,\ldots,2,2,1,1) = (1,1,1,1,2,2, \ldots,2).
\]

\noindent
By breaking $(1,1,1,1,2,2, \ldots,2)$ into $(1,1,1,1 \ | \ 2,2 \ | \ 2,2 \ | \  \cdots \ | \ 2,2)$, we can identify the image of $\tau \cdot \pclp{\ell}$ under $\funcPcltoZ$ with $\bM_\ell(4,2;(1),(2^{\eta}))$. Thus, we can define the $C_4$-action $\sqact_{4,2}$ on $\tau \cdot \pclp{\ell}$ by
\begin{align}\label{eq: 42action on tau}
\siggen_4 \sqact_{4,2} \Lambda = \funcPcltoZ^{-1}( \siggen_4 \sqact_{4,2} \funcPcltoZ(\Lambda)) \quad \text{for $\Lambda \in \tau \cdot \pclp{\ell}$}.
\end{align}
On the other hand, by breaking $(1,1,1,1,2,2, \ldots,2)$ into  $(1,1 \ | \ 1,1 \ | \ 2 \ | \ 2 \ | \ \cdots \ | \ 2)$,
we can also identify the image of $\tau \cdot \pclp{\ell}$ under $\funcPcltoZ$ with $\bM_\ell(2,1;(1^2),(2^{2\eta}))$. Thus, we can define the $C_2$-action $\sqact_{2,1}$ on $\tau \cdot \pclp{\ell}$  by
\begin{align}\label{eq: 21action on tau}
\siggen_2 \sqact_{2,1} \Lambda = \funcPcltoZ^{-1}( \siggen_2 \sqact_{2,1} \funcPcltoZ(\Lambda)) \quad \text{for $\Lambda \in \tau \cdot \pclp{\ell}$}.
\end{align}
Then, we will show
\begin{itemize}
\item $\left| \left( \pclp{\ell} \right)^{C_4} \right| = \pclp{\ell}(\zeta_4^j) \quad (j=1,3)$ using the $C_4$-action defined in~\eqref{eq: 42action on tau},
\item $\left| \left( \pclp{\ell} \right)^{\siggen_4^2} \right|  = \pclp{\ell}(-1)$ using the $C_2$-action defined in~\eqref{eq: 21action on tau}.
\end{itemize}

\medskip

By~\eqref{eq: level and series C equiv cond}, we have
\begin{equation}\label{eq: Dn1 dwt_ell(q)}
\begin{aligned}
\pclp{\ell}(q) & = \sum_{i \ge 0} \left| \left\{\Lambda \in \pclp{\ell} \; \middle| \; \evS(\Lambda) = i
\right\} \right|q^i \equiv \sum_{0 \le i \le 3} \left| \left\{\Lambda \in \pclp{\ell} \; \middle| \; \evS(\Lambda) \equiv_4 i
\right\} \right|q^i \pmod{q^4 -1}.
\end{aligned}
\end{equation}
For simplicity, we let
\[ \pclp{\ell}(q) \equiv \sum_{0 \le i \le 3} b_i q^i \pmod{q^4 - 1}.\]

Before proving Theorem~\ref{thm: D TYPE csp}, let us introduce four key lemmas.

\begin{lemma}\label{lem: Dn1 fixed by hsig}
Under the $C_4$-action $\sqact_{4,2}$ on $\bM_\ell(4,2;(1),(2^\eta))$ given in \eqref{eq: hsigaction double}, we have
\[\left|  \bM_\ell(4,2;(1),(2^\eta))^{C_4} \right| = b_0 - b_2 \quad \text{ and } \quad b_1 =b_3.\]
\end{lemma}

\begin{proof}
For simplicity, we set $X = \tau \cdot \pclp{\ell}$ and $X(i) \seteq \tau \cdot \left\{ \Lambda \in \pclp{\ell} \; \middle| \; \evS(\Lambda) \equiv_4 i \right\}$.
Note the following:
\begin{itemize}
\item $C_4$-orbits are of length $1,2$ or $4$.
\item Under the $C_4$-action on $X$ given in~\eqref{eq: 42action on tau}, $\left|  \bM_\ell(4,2;(1),(2^\eta))^{C_4} \right| = \left| X^{C_4} \right|$.
\item For any $\Lambda \in \pclp{\ell}$, $\evS(\Lambda) =  \widehat{\svv} \bullet  \funcPcltoZ(\tau \cdot \Lambda),$ where $\widehat{\svv} \seteq \tau \cdot \wtsvv = (0,1,2,3,0,2,0,2,\ldots,0,2)$.
\end{itemize}

Suppose that the following claims hold (which will be proven below): \\[-1ex]

{\it Claim 1.} Let $\Lambda \in X^{C_4}$ and $\mathbf{m} = \funcPcltoZ(\Lambda)$. Then $\whsvv \bullet \mathbf{m} = 0$ and hence $X^{C_4} \subset X(0)$.\\[-1ex]

{\it Claim 2.} Let $\Lambda \in X \setminus X^{C_4}, \mathbf{m} = \funcPcltoZ(\Lambda)$, and $i,j \in \{0,1,2,3\}$. Then
\[\whsvv \bullet (\siggen_4^i \sqact_{4,2} \mathbf{m} - \siggen_4^{i+1} \sqact_{4,2} \mathbf{m}) \equiv_4
\whsvv \bullet (\siggen_4^j \sqact_{4,2} \mathbf{m} - \siggen_4^{j+1} \sqact_{4,2} \mathbf{m}) \not \equiv_4 0,\]

By {\it Claim 1}, we have
\begin{align}\label{eq: fixed by C4}
X(2) \subset X \setminus X^{C_4}.
\end{align}
Let $\Lambda \in X(0) \cap \left(X \setminus X^{C_4} \right)$ and $\mathbf{m}=\funcPcltoZ(\Lambda)$. Then we have the following cases:
\begin{enumerate}
\item[{\it Case 1.}] $\siggen_4^2 \sqact_{4,2} \Lambda = \Lambda$,
\item[{\it Case 2.}] $\siggen_4^2 \sqact_{4,2} \Lambda \neq \Lambda$ and $\whsvv \bullet \funcPcltoZ(\Lambda - \siggen_4 \sqact_{4,2} \Lambda) \equiv_4 1 \text{ or } 3$,
\item[{\it Case 3.}] $\siggen_4^2 \sqact_{4,2} \Lambda \neq \Lambda$ and $\whsvv \bullet \funcPcltoZ(\Lambda - \siggen_4 \sqact_{4,2} \Lambda) \equiv_4 2$.
\end{enumerate}
\begin{itemize}
\item In {\it Case 1}, we have $\whsvv \bullet  (\siggen_4 \sqact_{4,2} \bm) = 2$ by {\it Claim 2} and thus $\siggen_4 \sqact_{4,2} \Lambda \in X(2)$. This shows that one can correspond $\Lambda$ to $\siggen_4 \sqact_{4,2} \Lambda$ in a bijective way.
\item In {\it Case 2}, we have $\whsvv \bullet (\siggen_4^2 \sqact_{4,2} \bm) = 2$  by {\it Claim 2} and thus $\siggen_4^2 \sqact_{4,2} \Lambda \in X(2)$. This shows that one can correspond $\Lambda$ to $\siggen_4^2 \sqact_{4,2} \Lambda$ in a bijective way.
\item In {\it Case 3}, we have $\whsvv \bullet  (\siggen^i_4 \sqact_{4,2} \bm)  \equiv_4 2i$ for $i=1,2,3$. Therefore $\siggen_4 \sqact_{4,2} \Lambda, \siggen_4^3 \sqact_{4,2} \Lambda \in X(2)$ and $ \siggen_4^2 \sqact_{4,2} \Lambda \in X(0)$. This shows that one can correspond $\Lambda,  \siggen_4^2 \sqact_{4,2} \Lambda $ to $\siggen_4 \sqact_{4,2} \Lambda, \siggen_4^3 \sqact_{4,2} \Lambda$ in a bijective way.
\end{itemize}
In this way, we obtain a bijection between $X(0) \cap \left(X \setminus X^{C_4} \right)$ and $X(2) \cap \left(X \setminus X^{C_4} \right)$ and thus, by~\eqref{eq: fixed by C4},
\begin{align*}
\left|X(0) \cap \left(X \setminus X^{C_4} \right)  \right|
& = \left|X(2) \cap \left( X \setminus X^{C_4} \right)  \right| = \left|X(2)\right|.
\end{align*}
Consequently we have
\begin{align*}
\left|  \bM_\ell(4,2;(1),(2^\eta))^{C_4} \right|
&= \left|X^{C_4} \right|
= \left|X(0)\right| - \left|X{(0)} \cap \left(X \setminus X^{C_4} \right) \right|\\
& = \left|X(0)\right| - \left|X(2)\right| = b_0 - b_2.
\end{align*}

In the same manner, by taking $\Lambda \in X(1)$ or $X(3) \subset X \setminus X^{C_4}$, one can see that
$$   |X(1)| = |X(3)| \iff b_1 =b_3.$$

To complete the proof, we have only to verify {\it Claim 1} and {\it Claim 2}.

\noindent (a) For {\it Claim 1}, suppose that $\Lambda \in X^{C_4}$. Recall $\funcMtoW$ in~\eqref{eq: definition of Psi} and $\funcPcltoZ$ in~\eqref{eq: funcPcltoZ}. Let $\bm = \funcPcltoZ(\Lambda)$ and $\bw = \funcMtoW(\bm)$. Break $\bw$ into subwords
\begin{align}\label{eq: Dn1 word break}
w^1 \ | \ w^2 \ | \ \cdots \ | \ w^{t_1} \ | \ w^{t_1 + 1}\ | \ \cdots \ | \ w^{t_1+t_2} \ | \ w^{0},
\end{align}
where
\begin{itemize}
\item $w^j$ is of length $4$ for $1 \le j \le t_1$,
\item $w^{t_1}$ contains the $4$th zero when we read $\mathbf{w}$ from left to right,
\item $w^j$ is of length $2$ for $t_1+1 \le j \le t_1 + t_2$,
\item  $w^0$ is the empty word or of length $1$.
\end{itemize}
Here such $w^{t_1}$ exists since the number of $0$ in $\bw$  is $n \ge 4$ by Algorithm~\ref{Algorithm} and~\eqref{eq: definition of Psi}.

Since $\Lambda \in X^{C_4}$, we have $\bm \in \bM_\ell(4,2;(1),(2^\eta))^{C_4}$. Then $\siggen_4 \sqact_{4,2}$ in $\eqref{eq: hsigaction double}$ and~\eqref{eq: definition of Psi}  say that
\begin{align}\label{eq: Dn1 fixed w}
\underbracket{1 \cdots 1}_{4k_1} \ | \ 000\blo \ | \ \underbracket{2 \cdots 2}_{2k_2} \ | \ 0 0 \ | \ \underbracket{2\cdots 2}_{2k_3} \ | \ 00 \ | \ \cdots \ | \ \underbracket{2 \cdots 2}_{2k_{\eta+1}} \ | \ 0
\end{align}
for some $k_1,k_2,\ldots, k_{\eta+1} \in \Z_{\ge 0}$ such that $\sum_{1 \le j \le \eta+1} 4k_j = \ell$.
Therefore,
\begin{align}\label{eq: Dn1 fixed dwt}
\mathbf{m} = \left(4k_1, 0 ,0,0,2k_2,0,2k_3,0,\ldots,2k_{\eta+1},0\right)
\end{align}
and hence
\[\whsvv  \bullet \mathbf{m} = (0,1,2,3,0,2,\ldots,0,2) \bullet \left(4k_1, 0 ,0,0,2k_2,0,2k_3,0,\ldots,2k_{\eta+1},0\right) = 0.\]

\noindent (b) For {\it Claim 2}, suppose that $\Lambda \in X \setminus X^{C_4}$. Let $\bm = \funcPcltoZ(\Lambda)$ and $\bw = \funcMtoW(\bm)$. Break $\mathbf{w}$ into subwords
\[w^1 \ | \ w^2 \ | \ \cdots \ | \ w^{t_1} \ | \ w^{t_1 + 1}\ | \ \cdots \ | \ w^{t_1+t_2} \ | \ w^{0},\]
as~\eqref{eq: Dn1 word break}. Since $\Lambda \in X \setminus X^{C_4}$, there exists the smallest $1 \le j_0 \le t_1 + t_2$ such that $\siggen_4 \cdot w^{j_0} \neq w^{j_0}$. Note that $w^{i_0}$ should be one of
\begin{align*}
& 1000,~~ 0100,~~ 0010,~~ 0001, \qquad 1100,~~ 0110,~~ 0011,~~1001,\\
& 1110,~~ 0111,~~ 1011,~~ 1101, \qquad 1010,~~ 0101,~ \qquad 20,~~ 02.
\end{align*}
We shall only give a proof for the case where $w^{j_0} = 1000$ since the other cases can be proven by the same argument.
Since $\bm \in \bM_\ell(4,2;(1),(2^\eta))$ and $w^{j_0}$ contains $1$, by~\eqref{eq: definition of Psi},  $w^{j}$ is $1111$ for all $1\le j < j_0$. Thus we have
\begin{align*}
\siggen^i_4 \sqact_{4,2} \mathbf{w} &=\hspace{.87em} w^1 \hspace{1.17em} | \hspace{1.17em} w^2 \hspace{1.17em} | \ \cdots \ |\hspace{.52em} w^{j_0-1} \hspace{.51em} | \hspace{.62em} \siggen \cdot w^{j_0} \hspace{.52em} | \ w^{j_0+1} \ | \ \cdots \ | \ w^t \ | \ w_{4t+1} \cdots w_u\\
& = \hspace{.3em} 1111 \hspace{.6em} \ |\hspace{.6em} 1111 \hspace{.6em}\ | \ \cdots \ | \ \hspace{.3em} 1111 \hspace{.6em} \ | \  \hspace{.6em} \begin{matrix} 0100 \\ 0010 \\ 0001 \end{matrix} \hspace{.7em} \ | \ w^{j_0+1} \ | \ \cdots \ | \ w^t \ | \ w_{4t+1}\cdots w_u
\quad \begin{matrix} \text{ if } i=1, \\ \text{ if } i=2, \\ \text{ if } i=3, \end{matrix}
\end{align*}
which implies
\begin{align*}
\siggen_4^i \sqact_{4,2} \mathbf{m} = (m_0 - 1, m_1 + \delta_{i,1}, m_2+ \delta_{i,2}, m_3+ \delta_{i,3}, m_4\ldots,m_n)\quad \text{for $i = 1,2,3$},
\end{align*}
by the construction of $\funcMtoW^{-1}$ in \eqref{eq: definition inverse of Psi}.  Hence, we have
\begin{align*}
& \widehat{\svv} \bullet (\mathbf{m} - \siggen_4 \sqact_{4,2} \mathbf{m}) = (0,1,2,3,0,2,\ldots,0,2)  \bullet (1,-1,0,0,0,\ldots,0)   \equiv_4 3, \\
& \widehat{\svv}  \bullet (\siggen_4 \sqact_{4,2}\mathbf{m} - \siggen_4^2 \sqact_{4,2} \mathbf{m})= (0,1,2,3,0,2,\ldots,0,2)  \bullet (0,1,-1,0,0,\ldots,0) \equiv_4 3, \\
& \widehat{\svv}  \bullet (\siggen_4^2 \sqact_{4,2}\mathbf{m} - \siggen_4^3 \sqact_{4,2} \mathbf{m})= (0,1,2,3,0,2,\ldots,0,2)  \bullet (0,0,1,-1,0,\ldots,0) \equiv_4 3, \\
& \widehat{\svv}  \bullet (\siggen_4^3 \sqact_{4,2}\mathbf{m} - \mathbf{m})= (0,1,2,3,0,2,\ldots,0,2)  \bullet (-1,0,0,1,0,\ldots,0) \equiv_4 3.
\end{align*}
This completes the proof.
\end{proof}

\begin{lemma}\label{lem: Dn1 hsig and Cact}
    Under the $C_4$-action on $\pclp{\ell}$ given in~\eqref{eq: C4 action on D_n} and the $C_4$-action on $\bM_\ell(4,2;(1),(2^\eta))$ given in \eqref{eq: hsigaction double}, we have
    \[\left| \left( \pclp{\ell} \right)^{C_4} \right| = \left| \bM_\ell(4,2;(1),(2^\eta))^{C_4} \right|.\]
\end{lemma}

\begin{proof}
By~\eqref{eq: C4 action on D_n}, one can see that
\begin{align}\label{eq: Dn1 fixed by Cact}
\left( \pclp{\ell} \right)^{C_4} = \left\{ \sum_{0 \le i \le n} m_i\Lambda_i \in \pclp{\ell} \; \middle| \; \begin{array}{l}
m_0 = m_1 = m_{n-1} = m_n,\\
m_{2j} = m_{2j+1}\quad \text{for $1\le j \le \eta$},\\
m_0 + m_1 + \sum_{2 \le j \le n-2}2m_j + m_{n-1} + m_n = \ell
\end{array} \right\}.
\end{align}

We claim that
$$\bM_\ell(4,2;(1),(2^\eta))^{C_4}= \left\{ \left(4k_1, 0 ,0,0,2k_2,0,2k_3,0,\ldots,2k_{\eta+1},0\right) \ \middle| \ k_i \in \Z_{\ge 0}, \ \sum_{1 \le i \le \eta+1} 4k_i =\ell  \right\} =: Y.$$

By~\eqref{eq: Dn1 fixed w} and \eqref{eq: Dn1 fixed dwt}, $\bM_\ell(4,2;(1),(2^\eta))^{C_4}$ is contained in $Y$.

\smallskip

On the contrary, for an element  $\mathbf{m} \in Y$, we have
\begin{align*}
\funcMtoW(\mathbf{m}) &
= \underbracket{1 \cdots 1}_{4k_1} \ | \ 000\blo \ | \ \underbracket{2 \cdots 2}_{2k_2} \ | \ 0 0 \ | \ \underbracket{2\cdots 2}_{2k_3} \ | \ 00 \ | \ \cdots \ | \ \underbracket{2 \cdots 2}_{2k_{\eta+1}} \ | \ 0.
\end{align*}
Thus, by \eqref{eq: hsigaction double} and~\eqref{eq: definition of Psi}, $\siggen_4 \sqact_{4,2} \mathbf{m} = \mathbf{m}$ and hence our claim follows.

Next, we have an obvious bijection $\Theta: \left( \pclp{\ell} \right)^{C_4} \ra \left( \bM_\ell(4,2;(1),(2^\eta)) \right)^{C_4}$ defined by
\[ \Theta\left( \sum_{0 \le i \le n} m_i\Lambda_i \right) = \left( 4m_0,0,0,0,2m_2,0,2m_4,0,\ldots,2m_{n-3},0\right). \]
This completes the proof.
\end{proof}

\begin{lemma}\label{lem: Dn1 fixed by hsig'}
Under the $C_2$-action given in \eqref{eq: hsigaction double}, we have
\[\left|\bM_\ell(2,1;(1^2),(2^{2\eta}))^{C_2} \right| = b_0 - b_1 + b_2 - b_3.\]
\end{lemma}

\begin{proof}
As we did in the proof of Lemma~\ref{lem: Dn1 fixed by hsig}, we set $X = \tau \cdot \pclp{\ell}$ and $X(i) \seteq \tau \cdot \left\{ \Lambda \in \pclp{\ell} \; \middle| \; \evS(\Lambda) \equiv_4 i \right\}$.
Note the following:
\begin{itemize}
\item Under the $C_2$-action on $X$ given in~\eqref{eq: 21action on tau}, $\left|\bM_\ell(2,1;(1^2),(2^{2\eta}))^{C_2} \right| = \left| X^{C_2} \right|$.
\item For any $\Lambda \in \pclp{\ell}$, $\evS(\Lambda) =  \widehat{\svv} \bullet  \funcPcltoZ(\tau \cdot \Lambda),$ where $\widehat{\svv} \seteq \tau \cdot \wtsvv = (0,1,2,3,0,2,0,2,\ldots,0,2)$.
\end{itemize}

Suppose that the following claims hold: \\[-1ex]

{\it Claim 1.} Let $\Lambda \in X^{C_2}$ and $\mathbf{m} = \funcPcltoZ(\Lambda)$. Then $\whsvv \bullet \mathbf{m} \equiv_2 0$ and hence  $X^{C_2} \subset X(0) \cup X(2)$.\\[-1ex]

{\it Claim 2.} Let $\Lambda \in X \setminus X^{C_2}$ and $\mathbf{m} = \funcPcltoZ(\Lambda)$. Then
\[\whsvv \bullet \left(\mathbf{m} - \siggen_2 \sqact_{2,1} \mathbf{m}\right) \equiv_2 1.\]

By {\it Claim1}, we have
\begin{align}\label{eq: Dn1 fixed supset}
X(1)\sqcup X(3) \subset X \setminus X^{C_2}.
\end{align}

Let $\Lambda \in \left(X(0)\sqcup X(2)\right) \cap X \setminus X^{C_2}$ and $\bm = \funcPcltoZ(\Lambda)$. By {\it Claim2}, we have $\whsvv \bullet (\siggen_2 \sqact_{2,1} \mathbf{m})\equiv_2 1$ and thus $\siggen_2 \sqact_{2,1} \Lambda \in \left(X(1)\sqcup X(3)\right) \cap X \setminus X^{C_2}$. So, we obtain a bijection from $\left(X(0)\sqcup X(2)\right) \cap X \setminus X^{C_2}$ to $\left(X(1)\sqcup X(3)\right) \cap X \setminus X^{C_2}$ by mapping $\Lambda$ to $\siggen_2 \sqact_{2,1} \Lambda$. By~\eqref{eq: Dn1 fixed supset}, we have
\begin{align*}
\left|\left(X(0)\sqcup X(2)\right) \cap X \setminus X^{C_2}\right| & = \left|\left(X(1)\sqcup X(3)\right) \cap X \setminus X^{C_2}\right| = \left|X(1)\sqcup X(3)\right|\\
& = \left|X(1)\right| + \left|X(3)\right|.
\end{align*}
Finally we have
\begin{align*}
\left| \bM_\ell(2,1;(1^2),(2^{2\eta}))^{C_2} \right|
& = \left| X^{C_2} \right|
 = \left| X(0)\sqcup X(2)\right| - \left| \left(X(0)\sqcup X(2)\right)\cap X \setminus X^{C_2}\right| \\
& = \left(\left| X(0) \right| + \left|X(2) \right|\right) - \left(\left|X(1)\right| + \left|X(3)\right|\right)
= b_0 - b_1 + b_2 - b_3.
\end{align*}

We omit the proof of {\it Claim1} and {\it Claim 2} since they can be proven in the same manner as those in the proof of Lemma~\ref{lem: Dn1 fixed by hsig}.
\end{proof}

\begin{lemma}\label{lem: Dn1 hsig' and Cact^2}
Under the $C_4 = \inn{\siggen_4}$-action on $\pclp{\ell}$ given in~\eqref{eq: C4 action on D_n} and the $C_2$-action on $\bM_\ell(2,1;(1^2),(2^{2\eta}))$ given in \eqref{eq: hsigaction double}, we have
\[\left| \left( \pclp{\ell} \right)^{\siggen_4^2} \right| = \left|  \bM_\ell(2,1;(1^2),(2^{2\eta}))^{C_2} \right|.\]
\end{lemma}

\begin{proof}
By~\eqref{eq: C4 action on D_n}, one can see that
\begin{align}\label{eq: Dn1 fixed by Cact square}
\left( \pclp{\ell} \right)^{\siggen_4^2} =  \left\{ \sum_{0 \le i \le n} m_i \Lambda_i \in \pclp{\ell} \; \middle| \; m_0 = m_{n-1}, m_1 = m_n~\text{and}~ 2m_0 + 2m_1 + \sum_{2 \le j \le n-2}m_j = \ell \right\}.
\end{align}

We claim that
\[\bM_\ell(2,1;(1^2),(2^{2\eta}))^{C_2} =  \left\{ (2k_1,0,2k_2,0,m_4,m_5,\ldots,m_n) \; \middle| \; \begin{array}{l}
k_1,k_2,m_4,m_5,\ldots,m_n \in \Z_{\ge 0 },\\
2k_1 + 2k_2 + \sum_{4 \le j \le n} m_j = \ell
\end{array}\right\} = : Y. \]
Suppose $\bm \in \bM_\ell(2,1;(1^2),(2^{2\eta}))^{C_2}$. Recall the function $\funcMtoW$ from~\eqref{eq: definition of Psi}. Let $\mathbf{w} = \funcMtoW(\mathbf{m})$.
Break $\mathbf{w}$ into subwords
\begin{align}\label{eq: Dn1 subword break}
w^1 \ | \ w^2 \ | \ \cdots \ | \ w^{t_1} \ | \ w_{2t_1 + 1}w_{2t_1 + 2}\cdots w_{u},
\end{align}
where
\begin{itemize}
\item $w^j$ is of length $2$ for $1 \le j \le t_1$,
\item $w^{t_1}$ contains the $4$th zero when we read $\bw$ from left to right.
\end{itemize}
By~\eqref{eq: hsigaction double} and~\eqref{eq: definition of Psi}, we have
\begin{align*}
\mathbf{w} = \underbracket{1\cdots 1}_{2k_1} \ | \ 00 \ | \ \underbracket{1\cdots 1}_{2k_2} \ | \ 0 \blo \ | \ w_{2t_1 + 1}w_{2t_1 + 2}\cdots w_{u}
\end{align*}
for some $k_1,k_2, \in \Z_{\ge 0}$; i.e., $w^{t_1} = 00$.
Therefore, $\mathbf{m} \in Y$.

On the contrary, for $\mathbf{m}'\in Y$, we have
\[\funcMtoW(\mathbf{m}') = \underbracket{1\cdots 1}_{2k_1} \ | \ 00 \ | \ \underbracket{1\cdots 1}_{2k_2} \ | \ 0 \blo \ | \ \underbracket{2\cdots 2}_{m_4'} 0 \underbracket{2\cdots 2}_{m_5'} 0 \cdots\cdots 0\underbracket{2\cdots 2}_{m_n'} . \]
Therefore, $\siggen_2 \sqact_{2,1} \mathbf{m}' = \mathbf{m}'$ and hence our assertion follows.

Hence, we have an obvious bijection $\Theta : \left( \pclp{\ell} \right)^{\siggen_4^2} \ra \bM_\ell(2,1;(1^2),(2^{2\eta}))^{C_2}$ defined by
\[ \Theta\left( \sum_{0 \le i \le n} m_i \Lambda_i \right) = ( 2m_0, 0, 2m_1, 0, m_2, m_3, \ldots, m_{n-2} ). \]
This completes the proof.
\end{proof}

\begin{proof}[Proof of Theorem \ref{thm: D TYPE csp}]
Let $\zeta_4$ be a $4$th primitive root of unity. We will see that
\[\left|\left( \pclp{\ell} \right)^{\siggen_4^j} \right| = \pclp{\ell}(\zeta_4^j) \quad \text{for $j = 0,1,2,3$.}\]
\begin{itemize}
\item When $j = 0$, since $\left|\left( \pclp{\ell} \right)\right| =  \pclp{\ell}(1)$, it is trivial.
\item For the case $j \in \{1,3\}$, note that
\[\left( \pclp{\ell} \right)^{C_4} = \left( \pclp{\ell} \right)^{\siggen_4^j} \quad \text{and} \quad \pclp{\ell}(\zeta_4^j) =  b_0 + b_1\zeta^j_4 - b_2 - b_3\zeta^j_4.\]
Lemma~\ref{lem: Dn1 fixed by hsig} and Lemma~\ref{lem: Dn1 hsig and Cact} say that
\[\left| \left( \pclp{\ell} \right)^{C_4} \right| =\left|  \bM_\ell(4,2;(1),(2^\eta))^{C_4} \right| = b_0 - b_2 = \pclp{\ell}(\zeta_4^j). \]
\item For the case $j = 2$, note that
\[ \pclp{\ell}(-1) = b_0 - b_1 + b_2 - b_3. \]
Lemma~\ref{lem: Dn1 fixed by hsig'} and Lemma~\ref{lem: Dn1 hsig' and Cact^2} say that
\[\left| \left( \pclp{\ell} \right)^{\siggen_4^2} \right| = \left|  \bM_\ell(2,1;(1^2),(2^{2\eta}))^{C_2} \right| = b_0 - b_1 + b_2 - b_3 = \pclp{\ell}(-1).\]
\end{itemize}
Thus our assertion holds.
\end{proof}

\section{Bicyclic sieving phenomenon for $D_n^{(1)}$} \label{Sec: CSP2}
We start with reviewing the notion of bicyclic sieving phenomenon. For details, see \cite[Section 3]{BRS} or \cite[Section 9]{Sagan2}.

Let $X$ be a finite set with a permutation action of a finite \emph{bicyclic group}, that is, a product $C_k \times C_{k'}$ for some $k,k' \in \Z_{>0}$. Fix embeddings $\omega : C_k  \ra \C^{\times}$ and $\omega': C_{k'}\ra \C^{\times}$ into the complex roots of unity. Let $X(q_1,q_2) \in \Z_{\ge 0 }[q_1,q_2]$.

\begin{proposition}[\cite{BRS}, Proposition 3.1]\label{prop: biCSP}
In the above situation, the following two conditions on the triple $(X,C_k\times C_{k'},X(q_1,q_2))$ are equivalent:
\begin{enumerate}
\item[{\rm (1)}] For any $(c,c') \in C_k \times C_{k'}$,
\[X(\omega(c), \omega(c')) = \left| \left\{ x \in X \; \middle| \; (c,c')\,x = x \right\} \right|.\]
\item[{\rm (2)}] The coefficients $a(j_1,j_2)$ uniquely defined by the expansion
\[X(q_1,q_2) \equiv \sum_{\substack{0 \le j_1 <k \\ 0 \le j_2 < k'}} a(j_1,j_2)q_1^{j_1}q_2^{j_2} \pmod{q_1^{k} - 1, q_2^{k'} -1} \]
have the following interpretation: $a(j_1,j_2)$ is the number of orbits of $C_k \times C_{k'}$ on $X$ for which the $C_k \times C_{k'}$-stabilizer subgroup of any element in the orbit lies in the kernel of the $C_k \times C_{k'}$-character $\rho^{(j_1,j_2)}$ defined by
\[ \rho^{(j_1,j_2)}(c,c') = \omega(c)^{j_1}\omega'(c')^{j_2}. \]
\end{enumerate}
\end{proposition}

\begin{definition}
When either of the above two conditions holds, we say that the triple $(X,C_k\times C_{k'},X(q_1,q_2))$ exhibits the \emph{bicyclic sieving phenomenon}.
\end{definition}

In this section, we let $\g = D_n^{(1)}~(n\equiv_2 0)$.
In this case,
$(a_i^\vee)_{i=0}^n = (1,1,2,2,\ldots,2,1,1)$,
\[\wtsvv^{(1)} = (s_i^{(1)})_{i=0}^n = (0,0,\ldots,0,2,2) \quad \text{and} \quad \wtsvv^{(2)} = (s_i^{(2)})_{i=0}^n = (0,2,0,2,0,\ldots,2,0).\]
We set
$$\fks^{(1)} \seteq \frac{1}{2}\wtsvv^{(1)} = (0,0,\ldots,0,1,1) \quad \text{and} \quad \fks^{(2)} \seteq \frac{1}{2}\wtsvv^{(2)} = (0,1,0,1,0,\ldots,1,0).$$

Let
\[\upsig_1 = (0,n)(1,n-1) \in \Sym_{[0,n]} \quad \text{and} \quad \upsig_2 = (0,1)(2,3)\cdots(n-4,n-3)(n-1,n) \in \Sym_{[0,n]}.\]
Note that $\upsig_1$ and $\upsig_2$ commute to each other in $\Sym_{[0,n]}$, so $\inn{\upsig_1, \upsig_2} \simeq C_2 \times C_2$. Thus, we can define a $C_2 \times C_2 = \inn{\siggen_2} \times \inn{\siggen_2}$-action on $\pclp{\ell}$ by
\begin{align}\label{eq: C2 times C2 action}
(\siggen_2,e) \cdot \sum_{0 \le i \le n} m_i \Lambda_i \seteq \sum_{0 \le i \le n} m_{\upsig_1(i)} \Lambda_{i} \quad \text{and} \quad (e,\siggen_2) \cdot \sum_{0 \le i \le n} m_i \Lambda_i \seteq \sum_{0 \le i \le n} m_{\upsig_2(i)} \Lambda_{i}.
\end{align}
Here $e$ denotes the identity of $C_2$.
Note that
\begin{align}\label{eq: Dn1 even Cact_1 reason}
a_j^\vee = a_{\upsig_k(j)}^\vee \quad \text{and} \quad \begin{cases}
\left\{\fks^{(k)}_j, \fks^{(k)}_{\upsig_k(j)}\right\} = \{ 0,1 \} & \text{if $\upsig_k(j) \neq j$,}\\[1ex]
\fks^{(k)}_j = 0 & \text{if $\upsig_k(j) = j$,}
\end{cases}
\end{align}
for any $k=1,2$ and $j = 0,1,\ldots,n$.

Let
\[
\pclp{\ell}(q_1,q_2) \seteq \sum_{\Lambda \in \pclp{\ell}} q_1^{\ev_{\fks^{(1)}}(\Lambda)}q_2^{\ev_{\fks^{(2)}}(\Lambda)},  
\]
where $\ev_{\fks^{(t)}}(\Lambda): \pclp{\ell} \ra \Z_{\ge 0 } ~(t=1,2)$ is defined as follow:
\[ \sum_{0 \le i \le n} m_i \Lambda_i \mapsto  \fks^{(t)} \bullet \mathbf{m}. \]
Alternatively, $\pclp{\ell}(q_1,q_2)$ can be defined by the geometric series as in the other affine types:
\begin{align}\label{eq: ge series Dn1}
\sum_{\ell \ge 0} \pclp{\ell}(q_1,q_2)t^\ell \seteq \prod_{0 \le i \le n} \dfrac{1}{1-q_1^{\fks^{(1)}_i}q_2^{\fks^{(2)}_i}t^{a_i^\vee}}.
\end{align}
We set
$$\ev_{\fks} (\Lambda) \seteq (\evfks{1}(\Lambda),\evfks{2}(\Lambda)).$$
Note that all components of $\svv^{(1)}$ and $\svv^{(2)}$ are even.
Therefore, by Theorem~\ref{thm: eq rel and sieving}, we have
\begin{equation}\label{eq: Dnotcong_n40}
\begin{aligned}
    \pclp{\ell}(q_1,q_2)
    & = \sum_{i_1, i_2 \ge 0} \left|\left\{\Lambda \in \pclp{\ell} \; \middle| \; \evfks{1}(\Lambda) = i_1 \text{ and } \evfks{2}(\Lambda) = i_2 \right\}\right|q_1^{i_1} q_2^{i_2} \\
    & \equiv \sum_{0 \le i_1, i_2 \le 1} \left|\left\{\Lambda \in \pclp{\ell} \; \middle| \; \evfks{1}(\Lambda) \equiv_2 i_1 \text{ and } \evfks{2}(\Lambda) \equiv_2 i_2 \right\}\right|
    q_1^{i_1} q_2^{i_2} \pmod{q_1^{2} - 1, q_2^{2} - 1}.
\end{aligned}
\end{equation}
For simplicity, we let
\[ \pclp{\ell}(q_1,q_2) \equiv \sum_{0\le i_1,i_2 \le 1} b_{(i_1,i_2)} q^{i_1}q^{i_2} \pmod{q_1^2 - 1, q_2^2 - 1}.\]
Throughout this subsection, we set
\[ \eta_1' = n-3 \quad \text{and} \quad \eta_2' = \frac{n-4}{2}. \]

\begin{theorem}\label{eq: Dn1 even biCSP}
Under the $C_2 \times C_2$-action given in~\eqref{eq: C2 times C2 action},
\[\left( \pclp{\ell},C_2\times C_2, \pclp{\ell}(q) \right)\]
exhibits the bicyclic sieving phenomenon.
\end{theorem}

Take
\begin{align*}
\tau_1 &= (n, n-2,\ldots,2,1)(n-1,n-3,\ldots,3)\in \Sym_{[0,n]} \text{ and } \\
\tau_2 &= (n-1, n-3, \ldots, 3, n, n-2,\ldots, 2) \in \Sym_{[0,n]}.
\end{align*}
Then we have
\begin{align*}
\tau_i \cdot (a_0^\vee,a_1^\vee,\ldots,a_n^\vee) & = \tau_i \cdot (1,1,2,2,\ldots,2,1,1) = (1,1,1,1,2,2,\ldots,2) \qquad \text{ for } i=1,2.  
\end{align*}
(a) By breaking $(1,1,1,1,2,2,\ldots,2)$ into $(1,1 \ | \ 1,1 \ | \ 2,2,\ldots,2)$,
the image of $\tau_1 \cdot \pclp{\ell}$ under $\funcPcltoZ$ can be identified with $\bM_\ell(2,1;(1^2),(2^{\eta'_1})) \subset \Z^{n+1}_{\ge 0}$ and we can define the $C_2$-action on $\tau_1 \cdot \pclp{\ell}$ by
\begin{align}\label{eq: siggen tau_1 act}
\siggen_2 \sqact_{2,1} \Lambda =  \funcPcltoZ^{-1}( \siggen_2 \sqact_{2,1} \funcPcltoZ(\Lambda)) \quad \text{for $\Lambda \in \tau_1 \cdot \pclp{\ell}$}.
\end{align}
(b) By breaking $(1,1,1,1,2,2,\ldots,2)$ into $ (1,1 \ | \ 1,1 \ | \ 2,2 \ | \ 2,2 \ | \ \cdots \ | \ 2,2 \ | \ 2)$,
the image of $\tau_2 \cdot \pclp{\ell}$ under $\funcPcltoZ$ can be identified with $\bM_\ell(2,1;(1^2, 2^{\eta'_2}),(2))\subset \Z^{n+1}_{\ge 0}$ and we can define the $C_2$-action on $\tau_2 \cdot \pclp{\ell}$ by
\begin{align}\label{eq: siggen tau_2 act}
\siggen_2 \sqact_{2,1} \Lambda =  \funcPcltoZ^{-1}( \siggen_2 \sqact_{2,1} \funcPcltoZ(\Lambda)) \quad \text{for $\Lambda \in \tau_2 \cdot \pclp{\ell}$}.
\end{align}

\begin{lemma}\label{lem: Dn1 even fixed by hsig1}
Under the $C_2$-action on $\bM_\ell(2,1;(1^2),(2^{\eta'_1}))$ given in \eqref{eq: hsigaction double},
\[\left| \bM_\ell(2,1;(1^2),(2^{\eta'_1}))^{C_2} \right|
= b_{(0,0)} - b_{(1,0)} + b_{(0,1)} - b_{(1,1)} = \pclp{\ell}(-1,1).
\]
\end{lemma}

\begin{proof}
For simplicity, we set $X_1 \seteq \tau_1 \cdot \pclp{\ell}$ and $X_1(i_1,i_2) \seteq \tau_1 \cdot \left\{\Lambda \in \pclp{\ell} \; \middle| \; \evfks{1}(\Lambda) \equiv_2 i_1 \text{ and } \evfks{2}(\Lambda) \equiv_2 i_2 \right\}$. Note the following:
\begin{itemize}
\item $C_2$-orbits are of length $1$ or $2$.
\item Under the $C_2$-action on $X_1$ given in~\eqref{eq: siggen tau_1 act}, $\left| \bM_\ell(2,1;(1^2),(2^{\eta'_1}))^{C_2} \right| = \left| X_1^{C_2} \right|$.
\item For any $\Lambda \in \pclp{\ell}$, $\evfks{1}(\Lambda) =  \whfksone \bullet  \funcPcltoZ(\tau \cdot \Lambda),$ where $\whfksone \seteq \tau \cdot \fks^{(1)} = (0,1,0,1,0,0,\ldots,0)$.
\end{itemize}

Suppose that the following claims hold (which will be proven below):\\[-1ex]

{\it Claim 1.} Let $\Lambda \in X_1^{C_2}$ and $\mathbf{m} = \funcPcltoZ(\Lambda)$. Then $\whfksone \bullet \mathbf{m} = 0$ and hence $X_1^{C_2} \subset X_1(0,0)\cup X_1(0,1)$.\\[-1ex]

{\it Claim 2.} Let $\Lambda \in X_1 \setminus X_1 ^{C_2}$ and $\mathbf{m} = \funcPcltoZ(\Lambda)$. Then
\[\whfksone \bullet (\mathbf{m} -\siggen_2 \sqact_{2,1} \mathbf{m}) \equiv_2 1.\]

By {\it Claim 1}, we have
\begin{align}\label{eq: Dn1 even fixed supset}
X_1(1,1) \cup X_1(1,0)\subset X_1 \setminus X_1^{C_2}.
\end{align}

Let $\Lambda \in \left(X_1(0,0) \cup X_1(0,1)\right) \cap X_1 \setminus X_1^{C_2}$ and $\bm = \funcPcltoZ(\Lambda)$. Note that
\[ \whfksone \bullet  \funcPcltoZ(\tau \cdot \Lambda) \equiv_2 \begin{cases}
0 & \text{if $\Lambda \in X_1(0,0) \cup X_1(0,1)$},\\
1 & \text{if $\Lambda \in X_1(1,0) \cup X_1(1,1)$}.
\end{cases}  \]
By {\it Claim 2}, we have $\whfksone \bullet (\siggen_2 \sqact_{2,1} \mathbf{m}) \equiv_2 1$ and thus $\siggen_2 \sqact_{2,1} \Lambda \in X_1(1,0) \cup X_1(1,1)$.
So, we obtain a bijection from $\left(X_1(0,0) \cup X_1(0,1)\right) \cap X_1 \setminus X_1^{C_2}$ to $\left(X_1(1,0) \cup X_1(1,1)\right) \cap X_1 \setminus X_1^{C_2}$ by mapping $\Lambda$ to $\siggen_2 \sqact_{2,1} \Lambda$.
By~\eqref{eq: Dn1 even fixed supset}, we have
\begin{align*}
\left|\left(X_1(0,0) \cup X_1(0,1)\right) \cap X_1 \setminus X_1^{C_2}   \right|
&= \left|\left(X_1(1,0) \cup X_1(1,1)\right) \cap X_1 \setminus X_1^{C_2}   \right|\\
&= \left|X_1(1,0) \cup X_1(1,1) \right|
= \left|\pclp{\ell}(1,0)\right| + \left|\pclp{\ell}(1,1)\right|.
\end{align*}
Finally we have
\[ \left| \bM_\ell(2,1;(1^2),(2^{\eta'_1}))^{C_2} \right| = \left| X_1^{C_2} \right|  = b_{(0,0)} - b_{(1,0)} + b_{(0,1)} - b_{(1,1)}. \]

To complete the proof, we have only to verify {\it Claim 1} and {\it Claim 2}.

\noindent
(a) For {\it Claim 1,} suppose $\Lambda \in X_1^{C_2}$. Let $\mathbf{m} = \funcPcltoZ(\Lambda)$ and $\mathbf{w} = \funcMtoW(\mathbf{m})$. Break $\mathbf{w}$ into subwords
\begin{align}\label{eq: Dn1 even word break1}
w^1 \ | \ w^2 \ | \ \cdots \ | \ w^{t} \ | \ w_{2t+1}w_{2t+2}\cdots w_u,
\end{align}
where
\begin{itemize}
\item $w^j$ is of length $2$ for $1 \le j \le t$,
\item $w^{t}$ contains the $4$th zero when we read $\bw$ from left to right,
\item $w_j = 0~\text{or}~2$ for $2t+1 \le j \le u$.
\end{itemize}
Since $\Lambda \in X_1^{C_2}$, $\mathbf{w}$ should be of the form
\begin{align*}
\underbracket{1 \cdots 1}_{2k_1} \ | \ 00 \ | \ \underbracket{1 \cdots 1}_{2k_2} \ | \ 0 \blo \ | \  w_{2t+1}w_{2t+2}\cdots w_u
\end{align*}
for $k_1,k_2 \in \Z_{\ge 0}$. Therefore,
\begin{align}\label{eq: Dn1 even fixed dwt}
\mathbf{m} = (2k_1,0,2k_2,0,m_4, m_5,\ldots,m_n)
\end{align}
for some $m_4,m_5,\ldots,m_n \in \Z_{\ge 0}$ such that $2k_1 + 2k_2 + \sum_{4 \le j \le n} 2m_j = \ell$. Thus, we have
\[ \whfksone \bullet \mathbf{m} = (0,1, 0,1, 0,0,\ldots,0) \bullet (2k_1,0,2k_2,0,m_4, m_5,\ldots,m_n) = 0. \]

\noindent
(b) For {\it Claim 2}, suppose $\Lambda \in X \setminus X^{C_2}$. Let $\mathbf{m} = \funcPcltoZ(\Lambda)$ and $\mathbf{w} = \funcMtoW(\mathbf{m})$. Break $\mathbf{w}$ into subwords
\[ w^1 \ | \ w^2 \ | \ \cdots \ | \ w^{t} \ | \ w_{2t+1}w_{2t+2}\cdots w_u, \]
as~\eqref{eq: Dn1 even word break1}. Since $\Lambda \in X \setminus X^{C_2}$, there exists the smallest integer $1 \le j_0 \le t$ such that $\siggen_2 \cdot w^{j_0} \neq w^{j_0}$. Note that $w^{j_0}$ should be one of
\[10,~~01.\]

Note that if there are $1 \le j_1 < j_2 < j_0$ such that $w^{j_1} = w^{j_2} = 00$ then, by~\eqref{eq: definition of Psi}, $1$ can not appear in $w^{j_0}$ since $\bm \in \bM_\ell(2,1;(1^2),(2^{\eta'_1}))$. Therefore, there is at most one $j \in\{1,2,\ldots, j_0-1\}$ such that $w^j = 00$. Thus, we have four cases as follows:
\[\begin{array}{l}
w^{j_0} = 10 ~\text{and}~\text{there is no $j \in \{1,2,\ldots, j_0-1\}$ such that $w^j = 00$},\\
w^{j_0} = 10 ~\text{and}~\text{there is one $j \in \{1,2,\ldots, j_0-1\}$ such that $w^j = 00$},\\
w^{j_0} = 01 ~\text{and}~\text{there is no $j \in \{1,2,\ldots, j_0-1\}$ such that $w^j = 00$, and}\\
w^{j_0} = 01 ~\text{and}~\text{there is one $j \in \{1,2,\ldots, j_0-1\}$ such that $w^j = 00$}.\\
\end{array}\]
We shall only give a proof for the case that $w^{j_0} = 10$ and there is no $j \in \{1,2,\ldots, j_0-1\}$ such that $w^j = 00$ since the other cases can be proved by the same argument. In this case, $\mathbf{w}$ is of the form
\[ \underbracket{1\cdots1}_{2k} \ | \ 10 \ | \ w^{j_0 +1} \ | \ \cdots \ | \ w^{t} \ | \  w_{2t+1}w_{2t + 2}\cdots w_{u} \]
for some $k \in \Z_{\ge 0 }$. Thus,
\[\siggen_2 \sqact_{2,1} \mathbf{w} = \underbracket{1\cdots1}_{2k} \ | \ 01 \ | \ w^{j_0 +1} \ | \ \cdots \ | \ w^{t_1} \ | \  w_{2t_1 + 1}w_{2t_1 + 2}\cdots w_{u} \]
and hence
\[ \siggen_2 \sqact_{2,1} \mathbf{m} = (m_0-1,m_1+1,m_2,m_3,\ldots,m_n). \]
Thus we have
\[ \whfksone \bullet (\mathbf{m} - \siggen_2 \sqact_{2,1} \mathbf{m}) = (0,1,0,1,0,0,\ldots,0) \bullet(1,-1,0,0,0,0,\ldots,0,0) \equiv_2 1. \]
This completes the proof.
\end{proof}

\begin{lemma}\label{lem: Dn1 even hsig1 and Cact1}
Under the $C_2$-action on $\bM_\ell(2,1;(1^2),(2^{\eta'_1}))$ given in \eqref{eq: hsigaction double} and the $\inn{(\siggen_2,e)}$-action on $\pclp{\ell}$ given in~\eqref{eq: C2 times C2 action}, we have
\[\bM_\ell(2,1;(1^2),(2^{\eta'_1}))^{C_2} = \left| \left( \pclp{\ell} \right)^{(\siggen_2,e)} \right|.\]
\end{lemma}

\begin{proof}
By~\eqref{eq: C2 times C2 action}, one can see that
\begin{align}\label{Dn1 even fixed by Cact_1}
\left( \pclp{\ell} \right)^{(\siggen_2,e)} = \left\{ (m_0,m_1,\ldots,m_n) \in \Z_{\ge 0}^{n+1} \; \middle| \; \begin{array}{l}
m_0 = m_n,~m_{1} = m_{n-1},~\text{and}\\
m_0 + m_1 + \sum_{2 \le j \le n-2}2m_j + m_{n-1} + m_n = \ell
\end{array} \right\}.
\end{align}

We claim that
\[\bM_\ell(2,1;(1^2),(2^{\eta'_1}))^{C_2} = \left\{ \left(2k_1,0,2k_2,0,m_4, \ldots,m_n\right) \; \middle| \; \begin{array}{l}
k_1,k_2,m_4\ldots, m_{n} \in \Z_{\ge 0} \\
2k_1 + 2k_2 + \sum_{4 \le j \le n} 2m_j = \ell
\end{array} \right\} =: Y.\]
By~\eqref{eq: Dn1 even fixed dwt}, $\bM_\ell(2,1;(1^2),(2^{\eta'_1}))^{C_2}$ is contained in $Y$.

On the contrary, for $\mathbf{m}' = (2k'_1,0,2k'_2,0,m'_4, m'_5,\ldots,m'_n) \in Y$, we have
\begin{align*}
\funcMtoW(\mathbf{m'}) &
= \underbracket{1 \cdots 1}_{2k'_1} \ | \ 00 \ | \ \underbracket{1 \cdots 1}_{2k'_2} \ | \ 0 0 \ | \ \underbracket{2\cdots 2}_{2m'_4} 0  \underbracket{2\cdots 2}_{2m'_5} 0 \cdots \cdots 0 \underbracket{2 \cdots 2}_{2m'_{n}}.
\end{align*}
Thus, $\siggen_2 \sqact_{2,1} \mathbf{m}' = \mathbf{m}'$ and hence the claim follows.

Thus, we have an obvious bijection $\Theta: \left( \pclp{\ell} \right)^{(\siggen_2,e)} \ra \bM_\ell(2,1;(1^2),(2^{\eta'_1}))^{C_2}$ defined by
\[ \Theta\left( \sum_{0 \le i \le n} m_i \Lambda_i \right) = ( 2m_0, 0, 2m_1, 0,m_2,m_3,\ldots,m_{n-2} ).\]
This completes the proof.
\end{proof}

\begin{lemma}\label{lem: Dn1 even fixed by hsig2}
Under the $C_2$-action on $\bM_\ell(2,1;(1^2, 2^{\eta'_2}),(2))$ given in \eqref{eq: hsigaction double},
\[\left| \bM_\ell(2,1;(1^2, 2^{\eta'_2}),(2))^{C_2} \right|
= b_{(0,0)} + b_{(1,0)} - b_{(0,1)} - b_{(1,1)} = \pclp{\ell}(1,-1).\]
\end{lemma}

\begin{proof}
For simplicity, we write $X_2 \seteq \tau_2 \cdot \pclp{\ell}$ and $X_2(i_1,i_2) \seteq \tau_2 \cdot \left\{\Lambda \in \pclp{\ell} \; \middle| \; \evfks{1}(\Lambda) \equiv_2 i_1 \text{ and } \evfks{2}(\Lambda) \equiv_2 i_2 \right\}$.
Note the following:
\begin{itemize}
\item $C_2$-orbits are of length $1$ or $2$.
\item Under the $C_2$-action on $X_2$ given in~\eqref{eq: siggen tau_2 act}, $\left| \bM_\ell(2,1;(1^2, 2^{\eta'_2}),(2))^{C_2} \right| = \left| X_2^{C_2} \right|$.
\item For any $\Lambda \in \pclp{\ell}$, $\evfks{2}(\Lambda) =  \whfkstwo \bullet  \funcPcltoZ(\tau \cdot \Lambda),$ where $\whfkstwo \seteq \tau \cdot \fks^{(2)} = (0,1,0,1,\ldots,0,1,0)$.
\end{itemize}
Suppose that the following claims hold:\\[-1ex]

{\it Claim 1.} Let $\Lambda \in X_2^{C_2}$ and let $\mathbf{m} = \funcPcltoZ(\Lambda)$. Then $\whfkstwo \bullet \mathbf{m} = 0$ and hence $X_2^{C_2} \subset X_2(0,0) \cup X_2(1,0)$.\\[-1ex]

{\it Claim 2.} Let $\Lambda \in X_2 \setminus X_2^{C_2}$ and let $\mathbf{m} = \funcPcltoZ(\Lambda)$. Then
\begin{align*}
\whfkstwo \bullet (\mathbf{m} - \siggen_2 \sqact_{2,1} \mathbf{m}) \equiv_2 1.
\end{align*}

By {\it Claim 1}, we have
\begin{align}\label{eq: Dn1 even fixed supset2}
X_2(0,1) \cup X_2(1,1)\subset  X_2 \setminus X_2^{C_2}.
\end{align}

Let $\Lambda \in \left(X_2(0,0) \cup X_2(1,0)\right) \cap X_2 \setminus X_2^{C_2}$ and $\bm = \funcPcltoZ(\Lambda)$. Note that
\[ \whfkstwo \bullet  \funcPcltoZ(\tau \cdot \Lambda) \equiv_2 \begin{cases}
0 & \text{if $\Lambda \in X_2(0,0) \cup X_2(1,0)$},\\
1 & \text{if $\Lambda \in X_2(0,1) \cup X_2(1,1)$}.
\end{cases}  \]
Therefore, by {\it Claim 2}, we have $\whfkstwo \bullet (\siggen_2 \sqact_{2,1} \mathbf{m}) \equiv_2 1$ and thus $\siggen_2 \sqact_{2,1} \Lambda \in X_2(0,1) \cup X_2(1,1)$.
So, we obtain a bijection from $\left(X_2(0,0) \cup X_2(1,0)\right) \cap X_2 \setminus X_2^{C_2}$ to $\left(X_2(0,1) \cup X_2(1,1)\right) \cap X_2 \setminus X_2^{C_2}$ by mapping $\Lambda$ to $\siggen_2 \sqact_{2,1} \Lambda$.
By~\eqref{eq: Dn1 even fixed supset2}, we have
\begin{align*}
\left|\left(X_2(0,0) \cup X_2(1,0)\right) \cap X_2 \setminus X_2^{C_2}   \right|
&= \left|\left(X_2(0,1) \cup X_2(1,1)\right) \cap X_2 \setminus X_2^{C_2}   \right|\\
&= \left|X_2(0,1) \cup X_2(1,1) \right|
= \left|\pclp{\ell}(0,1)\right| + \left|\pclp{\ell}(1,1)\right|.
\end{align*}
Finally we have
\[ \left| \bM_\ell(2,1;(1^2, 2^{\eta'_2}),(2))^{C_2} \right| = \left| X_2^{C_2} \right|  = b_{(0,0)} + b_{(1,0)} - b_{(0,1)} - b_{(1,1)}. \]

We omit the proof of {\it Claim1} and {\it Claim 2} since they can be proven in the same manner as those in the proof of Lemma~\ref{lem: Dn1 even fixed by hsig1}.
\end{proof}

\begin{lemma}\label{lem: Dn1 even hsig2 and Cact2}
Under the $C_2$-action on $\bM_\ell(2,1;(1^2, 2^{\eta'_2}),(2))$ given in \eqref{eq: hsigaction double} and the $\inn{(e,\siggen_2)}$-action on $\pclp{\ell}$ given in~\eqref{eq: C2 times C2 action}, we have
\[\bM_\ell(2,1;(1^2, 2^{\eta'_2}),(2))^{C_2} = \left| \left( \pclp{\ell} \right)^{(e,\siggen_2)} \right|.\]
\end{lemma}

\begin{proof}
By~\eqref{eq: C2 times C2 action}, one can see that
\begin{align}\label{eq: fixed by Cact2}
\left( \pclp{\ell} \right)^{(e,\siggen_2)} = \left\{ \sum_{0 \le i \le n} m_i\Lambda_i \in \pclp{\ell} \; \middle| \; \begin{array}{l}
m_{2j} = m_{2j+1},~\text{for $j = 0,1,\ldots, \frac{n-4}{2}$},~ m_{n-1} = m_n,\\
m_0 + m_1 + \sum_{2 \le j \le n-2}2m_j + m_{n-1} + m_n = \ell
\end{array} \right\}.
\end{align}

We claim that
\[\bM_\ell(2,1;(1^2, 2^{\eta'_2}),(2))^{C_2} = \left\{ \left(2k_1,0,2k_2,0,\ldots,2k_{\frac{n}{2}},0,k_0\right) \; \middle| \; \begin{array}{l}
k_0,\ldots,k_{\frac{n}{2}} \in \Z_{\ge 0}\\
2k_0 + 2k_1 + 2k_2 + \sum_{3 \le j \le \frac{n}{2}} 4k_j = \ell
\end{array} \right\} =: Y. \]

Let $\mathbf{m} = \bM_\ell(2,1;(1^2, 2^{\eta'_2}),(2))^{C_2}$ and $\mathbf{w} = \funcMtoW(\mathbf{m})$. Break $\mathbf{w}$ into subwords
\begin{align}\label{eq: Dn1 even word break2}
w^1 \ | \ w^2 \ | \ \cdots \ | \ w^{t} \ | \ w_{2t+1}w_{2t+2}\cdots w_u,
\end{align}
where
\begin{itemize}
\item $w^j$ is of length $2$ for $1 \le j \le t$,
\item $w^{t}$ contains the $n$th zero when we read $\bw$ from left to right.
\end{itemize}
Since $\bm \in \bM_\ell(2,1;(1^2, 2^{\eta'_2}),(2))^{C_2}$, $\mathbf{w}$ should be of the form
\begin{align*}
\underbracket{1 \cdots 1}_{2k_1} \ | \ 00 \ | \ \underbracket{1 \cdots 1}_{2k_2} \ | \ 0 0 \ | \ \underbracket{2 \cdots 2}_{2k_3} \ | \ 0 0 \ | \ \underbracket{2 \cdots 2}_{2k_4} \ | \ \cdots \ | \ \underbracket{2 \cdots 2}_{2k_{\frac{n}{2}}} \ | \ 0 \blo \ | \ \underbracket{22\cdots 2}_{k_0},
\end{align*}
where for $j = 0,1,\ldots \frac{n}{2}$, $k_j \in \Z_{\ge 0}$ and $2k_0 + 2k_1 + 2k_2 +\sum_{3 \le j \le \frac{n}{2}}4k_{j} = \ell$. Therefore,
\begin{align}\label{eq: Dn1 even fixed dwt2}
\mathbf{m} = (2k_1,0,2k_2,0,2k_3,0,2k_4,0,\ldots,2k_{\frac{n}{2}},0,k_0).
\end{align}
Hence, $\bM_\ell(2,1;(1^2, 2^{\eta'_2}),(2))^{C_2}$ is contained in $Y$.

On the contrary, for $\mathbf{m} = (2k_1,0,2k_2,0,\ldots,2k_{\frac{n}{2}},0,k_0) \in Y$, we have
\begin{align*}
\funcMtoW(\mathbf{m}) &
= \underbracket{1 \cdots 1}_{2k_1} \ | \ 00 \ | \ \underbracket{1 \cdots 1}_{2k_2} \ | \ 0 0 \ | \ \underbracket{2 \cdots 2}_{2k_3} \ | \ 00   \ | \ \underbracket{2 \cdots 2}_{2k_4} \ | \ 00  \ | \ \cdots\cdots \ | \ 00 \ | \ \underbracket{2 \cdots 2}_{2k_{\frac{n}{2}}} \ | \ 0\blo \ | \ \underbracket{2 \cdots 2}_{k_0}.
\end{align*}
Thus, $\siggen_2 \sqact_{2,1} \mathbf{m} = \mathbf{m}$ and hence our claim follows.

We have an obvious bijection $\Theta: \left( \pclp{\ell} \right)^{(e,\siggen_2)} \ra \bM_\ell(2,1;(1^2, 2^{\eta'_2}),(2))^{C_2}$ defined by
\[ \Theta \left( \sum_{0 \le i \le n} m_i \Lambda_i \right) = ( 2m_0, 0, 2m_2, 0, \ldots, 2m_{n-4}, 0, 2m_n,0,m_{n-2} ). \]
This completes the proof.
\end{proof}

\begin{remark}\label{rem: Dn1 even fixed Cact2 and Cact1Cact2}
Note that $\upsig_1 \upsig_2 = (0,n-1)(1,n)(2,3)(4,5)\cdots(n-4,n-3) \in \Sym_{[0,n]}$. Therefore,
\begin{align}\label{eq: Dn1 n2 fixed by sig2,sig2}
\left( \pclp{\ell} \right)^{(\siggen_2,\siggen_2)} = \left\{ (m_0,m_1,\ldots,m_n) \in \Z_{\ge 0}^{n+1} \; \middle| \; \begin{array}{l}
m_0 = m_{n-1},~m_{1} = m_n,\\
m_{2j} = m_{2j+1},~\text{for $j = 1,\ldots, \frac{n-4}{2}$},~\text{and}\\
m_0 + m_1 + \sum_{2 \le j \le n-2}2m_j + m_{n-1} + m_n = \ell
\end{array} \right\}.
\end{align}
Combining~\eqref{eq: fixed by Cact2} with~\eqref{eq: Dn1 n2 fixed by sig2,sig2}, we have $\left| \left( \pclp{\ell} \right)^{(e,\siggen_2)} \right| = \left| \left( \pclp{\ell} \right)^{(\siggen_2,\siggen_2)} \right|$.
\end{remark}

\begin{lemma}\label{lem: Dn1 even X_10 X_11 bijection}
For any even integer $n \ge 4 $ and $\ell \in \Z_{>0}$, there is a bijection between $\pclp{\ell}((\ell-1)\Lambda_0+\Lambda_{n-1})$ and $\pclp{\ell}((\ell-1)\Lambda_0+\Lambda_{n})$, that is, $b_{(1,0)} = b_{(1,1)}$.
\end{lemma}

\begin{proof}
Recall that
\begin{align*}
&\pclp{\ell}((\ell-1)\Lambda_0+\Lambda_{n-1}) = \left\{\Lambda \in \pclp{\ell} \; \middle| \; \evfks{1}(\Lambda) \equiv_2 1 \text{ and } \evfks{2}(\Lambda) \equiv_2 1 \right\}\\
&\pclp{\ell}((\ell-1)\Lambda_0+\Lambda_{n}) = \left\{\Lambda \in \pclp{\ell} \; \middle| \; \evfks{1}(\Lambda) \equiv_2 1 \text{ and } \evfks{2}(\Lambda) \equiv_2 0 \right\}.
\end{align*}

Note that, for any $\sum_{0 \le i \le n} m_i \Lambda_i \in \pclp{\ell}((\ell-1)\Lambda_0+\Lambda_{n-1}) \cup \pclp{\ell}((\ell-1)\Lambda_0+\Lambda_{n})$, we have
\[\evfks{1}\left(\sum_{0 \le i \le n} m_i \Lambda_i\right) = m_{n-1} + m_n \equiv_2 1.\]
Note also that, for any $\sum_{0 \le i \le n} m_i \Lambda_i \in \pclp{\ell}((\ell-1)\Lambda_0+\Lambda_{n-1})$, we have
\[\evfks{2}\left(\sum_{0 \le i \le n} m_i \Lambda_i \right) = m_1+m_3+\cdots + m_{n-1} \equiv_2 1.\]
Therefore, for $\sum_{0 \le i \le n} m_i \Lambda_i \in \pclp{\ell}((\ell-1)\Lambda_0+\Lambda_{n-1})$, if $ m_1 + m_3 + \cdots + m_{n-3}$ is even then $m_{n-1}$ should be odd and so $m_{n-1}>0$. If $m_1 + m_3 + \cdots + m_{n-3}$ is odd, $m_{n-1}$ should be even and so $m_n$ should be odd and so $m_{n}>0$.

Let $\uppsi : \pclp{\ell}((\ell-1)\Lambda_0+\Lambda_{n-1}) \ra \pclp{\ell}((\ell-1)\Lambda_0+\Lambda_{n})$ be a function defined by
\[ \uppsi \left(\sum_{0 \le i \le n} m_i \Lambda_i\right) = \begin{cases}
\displaystyle \sum_{0 \le i \le n-2} m_i \Lambda_i + (m_{n-1} - 1)\Lambda_{n-1} + (m_n + 1) \Lambda_n & \text{if $ m_1 + m_3 + \cdots + m_{n-3}$ is even,}\\[1ex]
\displaystyle \sum_{0 \le i \le n-2} m_i \Lambda_i + (m_{n-1} + 1)\Lambda_{n-1} + (m_n - 1) \Lambda_n
& \text{if $ m_1 + m_3 + \cdots + m_{n-3}$ is odd.}
\end{cases}\]
By the above observations, one can easily see that $\uppsi$ is well-defined.

One can easily see that the function $ \uppsi^{-1}: \pclp{\ell}((\ell-1)\Lambda_0+\Lambda_{n}) \ra \pclp{\ell}((\ell-1)\Lambda_0+\Lambda_{n-1})$ defined by
\[\uppsi^{-1} \left(\sum_{0 \le i \le n} m_i \Lambda_i\right) = \begin{cases}
\displaystyle \sum_{0 \le i \le n-2} m_i \Lambda_i + (m_{n-1} + 1)\Lambda_{n-1} + (m_n - 1)\Lambda_n & \text{if $ m_1 + m_3 + \cdots + m_{n-3}$ is even,}\\[1ex]
\displaystyle\sum_{0 \le i \le n-2} m_i \Lambda_i + (m_{n-1} - 1)\Lambda_{n-1} + (m_n + 1)\Lambda_n & \text{if $ m_1 + m_3 + \cdots + m_{n-3}$ is odd.}
\end{cases}\]
is the inverse function of $\uppsi$.
Thus $\uppsi$ is a bijection and hence our assertion follows.
\end{proof}

\begin{proof}[Proof of Theorem~\ref{eq: Dn1 even biCSP}]
By (1) of Proposition \ref{prop: biCSP}, it suffices to show that
\[\begin{array}{l}
\pclp{\ell}(1,1) \hspace{-.3ex} = \hspace{-.3ex}\left|\pclp{\ell}\right|,
~\pclp{\ell}(-1,1) \hspace{-.3ex}=\hspace{-.3ex} \left|\left(\pclp{\ell}\right)^{\hspace{-.5ex}(\siggen_2,e)}\right|,
~\pclp{\ell}(1,\hspace{-.5ex}-1)\hspace{-.3ex} =\hspace{-.3ex} \left|\left(\pclp{\ell}\right)^{\hspace{-.5ex}(e,\siggen_2)}\right|,
~\pclp{\ell}(-1,\hspace{-.5ex}-1) \hspace{-.3ex}=\hspace{-.3ex} \left|\left(\pclp{\ell}\right)^{\hspace{-.5ex}(\siggen_2,\siggen_2)}\right|.
\end{array}\]

We have that
\begin{itemize}
\item $\pclp{\ell}(1,1) = \left|\pclp{\ell}\right|$,
\item by Lemma \ref{lem: Dn1 even fixed by hsig1} and Lemma \ref{lem: Dn1 even hsig1 and Cact1},
$\pclp{\ell}(-1,1) = \left|\left(\pclp{\ell}\right)^{(\siggen_2,e)}\right|$,
\item by Lemma \ref{lem: Dn1 even fixed by hsig2} and Lemma \ref{lem: Dn1 even hsig2 and Cact2}, $\pclp{\ell}(1,-1) = \left|\left(\pclp{\ell}\right)^{(e,\siggen_2)}\right|$ and
\item by Remark \ref{rem: Dn1 even fixed Cact2 and Cact1Cact2} and Lemma \ref{lem: Dn1 even X_10 X_11 bijection},
\begin{align*}
\pclp{\ell}(-1,-1) = \pclp{\ell}(1,-1) = \left|\left(\pclp{\ell}\right)^{(e,\siggen_2)}\right| = \left|\left(\pclp{\ell}\right)^{(\siggen_2,\siggen_2)}\right|.
\end{align*}
\end{itemize}
Hence our assertion follows.
\end{proof}

\section{Formulae on the number of maximal dominant weights} \label{Sec: application}

In this section, exploiting the sieving phenomenon on $\pclp{\ell}$,  we derive a closed formula for $|\mx^+(\Lambda)|$.
Based on this formula, we also derive a recursive formula for $|\mx^+(\Lambda)|$.
Finally, we observe a remarkable symmetry, called \emph{level-rank duality}, on dominant maximal weights.

\subsection{Closed formulae on the number of dominant maximal weights}\label{subsec: The number of dom max wts}
In case of $A_n^{(1)}$ type, we have already given a closed formula for  $|\mx^+(\Lambda)|$
for all $\Lambda = (\ell-1)\Lambda_0 + \Lambda_i \in \DRpclp$ (see Theorem \ref{thm: number of dom wt A}).
We here derive such a formula for an affine Kac-Moody algebra of arbitrary type.
Let us start with an example for reader's understanding.

\begin{example}
Let $\g = E_6^{(1)}$ and $\ell \in \Z_{>0}$.
Then $\DRpclp = \{ \ell \Lambda_0, (\ell-1) \Lambda_0 + \Lambda_1, (\ell-1) \Lambda_0 + \Lambda_6 \}$.
Under the $C_3=\inn{\siggen_3}$ action on $\pclp{\ell}$  as in \eqref{eq: CSP action}, let $N_T$ be the number of all orbits and $N_F$ be the number of free orbits.
By Theorem \ref{thm: CSP} together with \eqref{eq: E61 dwt_ell(q)}, we have
$N_T=|\mx^+(\ell \Lambda_0)|$ and $N_F=|\mx^+((\ell-1)\Lambda_0 + \Lambda_1)|= |\mx^+((\ell-1)\Lambda_0 + \Lambda_6)|$.
Since \[ \left| \pclp{\ell}\right| = 3N_F + \left|\left( \pclp{\ell}\right)^{\siggen_3}\right| = N_T + 2N_F,\]
there follows
\[ N_F = \frac{1}{3}\left( \left| \pclp{\ell}\right| - \left|\left( \pclp{\ell}\right)^{\siggen_3}\right| \right) \quad \text{and} \quad N_T = \left| \pclp{\ell}\right| - 2N_F. \]
\end{example}
\smallskip

Notice that for any type of Kac-Moody algebra $\g$,
\[|\pclp{\ell}| = \left|\bM_\ell(1;\mathbf{a}^\vee) \right|= \left|\bM_\ell(1; \tau \cdot \mathbf{a}^\vee) \right| \qquad \text{ for $\tau \in \Sym_{[0,n]}$},\]
where  $\mathbf{a}^\vee \seteq (a_0^\vee,a_1^\vee,\ldots,a_n^\vee)=[c]_{\Pi^\vee}$. For the definition of $\bM_\ell(1;\bnu)$, see~\eqref{eq: tuset d=1}.

Let $\ell$ be a nonnegative integer, $t_1 \in \Z_{>0}$ and $t_2, t_3, \ldots, t_k \in \Z_{\ge 0}$. Define
\[\upc_\ell(t_1;\emptyset) \seteq \begin{cases}
\ell & \text{if $t_1 > 1$,}\\
0 & \text{if $t_1 = 1$}.
\end{cases}  \]
Choose a nonnegative integer $i_1 \in [0,\upc_\ell(t_1;\emptyset)]$ and define
\[\upc_\ell(t_2; i_1) \seteq \begin{cases}
\floor{\frac{\ell - i_1}{2}} & \text{if $t_2 > 0$,}\\
0 & \text{if $t_2 = 0$}.
\end{cases}\]
For $1 < r < k$, suppose that $i_1,i_2,\ldots,i_{r-1}$ are chosen and $\upc_\ell(t_{r};i_1,i_2,\ldots,i_{r-1})$ is defined. Now, choose a nonnegative integer $i_r \in [0,\upc_\ell(t_{r};i_1,i_2,\ldots,i_{r-1})]$ and define
\[\upc_\ell(t_{r+1}; i_1,i_2,\ldots,i_{r}) \seteq \begin{cases}
\floor{\frac{\ell - \sum_{1 \le s \le r} si_s}{r+1}} & \text{if $t_{r+1} > 0$,}\\
0 & \text{if $t_{r+1} = 0$}.
\end{cases}\]
Since
\[i_r \le \upc_\ell(t_r;i_1,i_2,\ldots,i_{r-1}) \le \frac{\ell - \sum_{1 \le s \le r-1} s i_s}{r} \quad \text{for $1 < r < k$},\]
we have $\sum_{1 \le s \le r} s i_s \le \ell$. This implies that $\upc_\ell(t_{r+1};i_1,i_2,\ldots,i_{r})$ is a nonnegative integer. For $1 \le r \le k$, if $i_1,i_2,\ldots, i_{r-1}$ and $t_r$ are clear in the context, we simply write $\upc_\ell[r]$ for $\upc_\ell(t_r;i_1,\ldots,i_{r-1})$.
With this notation, we have the following lemma.
\begin{lemma}\label{lem: calnum}
	Let $a,P \in \Z_{>0}, \ell \in \Z_{\ge 0}$ and $\bnu = (a^{t_1},(2a)^{t_2},\ldots, (ka)^{t_k}) \in \Z^{P+1}_{\ge 0}$ with $t_1, t_k >0$ and $t_2, t_3, \ldots, t_{k-1} \ge 0$.
	With the above notation, we have
	\begin{align}\label{eq: calnum}
	\left|\bM_\ell(1;\bnu) \right| = 
	\begin{cases}
	\displaystyle{\sum_{i_1 = 0}^{\upc_{\ell/a}(t_1;\emptyset)} \sum_{i_2 = 0}^{\upc_{\ell/a}(t_2; i_1)} \cdots \sum_{i_k = 0}^{\upc_{\ell/a}(t_{k}; i_1,i_2,\ldots,i_{k-1})}   \prod_{r=1}^k \binom{i_r + t_r -1-\delta_{1,r}}{t_r -1-\delta_{1,r}}} & \text{if $a$ divides $\ell$},\\
	\qquad \qquad \qquad 0 & \text{otherwise},
	\end{cases}
	\end{align}
	where $\binom{-1}{-1}$ is set to be $1$.
\end{lemma}
\begin{proof}
	In case where $a$ does not divide $\ell$, in view of~\eqref{eq: tuset d=1}, 
	one has that $\bM_\ell(1;\bnu) = \emptyset$. Therefore we assume that $a$ divides $\ell$. We will prove our assertion by induction on $k$. Let $k=1$. Then $\bnu = (a^{P+1})$ and
	\[\bM_\ell(1;\bnu) = \left\{ (m_0,m_1,\ldots,m_P) \in \Z_{\ge 0}^{P+1} \; \middle| \; \sum_{0 \le j \le P} a \cdot m_j = \ell \right\}. \]
	It follows that $|\bM_\ell(1;\bnu)|$ is equal to $\matr{P+\frac{\ell}{a}}{\frac{\ell}{a}}$. On the other hand, since $P+1 \ge 2$, the right hand side of~\eqref{eq: calnum} is given by
	\[ \sum_{0 \le i_1 \le \frac{\ell}{a}} \matr{i_1 + (P+1) -2}{P+1-2}.\]
	Using Pascal's triangle, one can easily see that it is equal to $\matr{P + \frac{\ell}{a}}{\frac{\ell}{a}}$. Thus we can start the induction.
	
	Let $k >1$ and assume that our assertion holds for all positive integers less than $k$. Set $p_0 \seteq \sum_{1 \le j \le k-1} t_j$ and $\bnu' = (a^{t_1}, (2a)^{t_2} ,\ldots, ((k-1)a)^{t_{k-1}})$. Then, by~\eqref{eq: tuset d=1},
	\[ \bM_\ell(1;\bnu) = \left\{ (m_0,m_1,\ldots,m_P) \; \middle| \;
	\begin{array}{l}
	(m_0 ,m_1, \ldots, m_{{p_0}-1}) \in \bM_{\ell-kai}(1;\bnu'),\\
	(m_{p_0},m_{p_0+1}, \ldots, m_P) \in \bM_{kai}(1;((ka)^{t_k}))
	\end{array}~\text{for some $0 \le i \le \floor{\frac{\ell}{ka}}$}
	\right\}.\]
	It follows that
	\[\left|\bM_\ell(1;\bnu)\right| =  \sum_{i = 0}^{\floor{\ell/ka}}  \left| \bM_{\ell - kai}(1;\bnu') \right| \times \left|\bM_{kai}(1;((ka)^{t_k})) \right|. \]
	Note that
	\[\left|\bM_{kai}(1;((ka)^{t_k})) \right| = \matr{i + t_k - 1}{t_k - 1}.\]
	Thus, by the induction hypothesis, we have
	\begin{equation}\label{eq: card of M(1;v)}
	\begin{aligned}
	\left|\bM_\ell(1;\bnu)\right|& = \sum_{i = 0}^{\floor{\ell/ka}} \left( \sum_{i_1 = 0}^{\upc_{\ell/a - ki}[1]} \sum_{i_2 = 0}^{\upc_{\ell/a - ki}[2]} \cdots \sum_{i_{k-1} = 0}^{\upc_{\ell/a - ki}[k-1]}   \prod_{r=1}^{k-1} \binom{i_r + t_r -1-\delta_{1,r}}{t_r -1-\delta_{1,r}} \right) \cdot \matr{i + t_k - 1}{t_k - 1} \\
	& = \sum_{(i_1,i_2,\ldots,i_{k-1},i) \in R}  \left( \prod_{r=1}^{k-1} \binom{i_r + t_r -1-\delta_{1,r}}{t_r -1-\delta_{1,r}}  \cdot \matr{i + t_k - 1}{t_k - 1}\right),
	\end{aligned}
	\end{equation}
	where
	\begin{align*}
	R := \left\{ (i_1,i_2,\ldots,i_{k-1},i) \in \Z^k \; \middle| \;
	0 \le i \le \floor{\frac{\ell}{ka}}
	\ \text{ and } \
	0 \le i_r \le \upc_{\ell/a - ki}[r] \ \text{ for } 1 \le r \le k-1
	\right\}.
	\end{align*}
	We claim that $R$ is equal to
	\begin{align*}
	R' \seteq \left\{ (i_1,i_2,\ldots,i_{k-1},i_k) \in \Z^k \; \middle| \;
	0 \le i_r \le \upc_{\ell/a}[r] \ \text{ for } 1 \le r \le k   \right\}.
	\end{align*}
	
	To prove $R \subseteq R'$, take $(i_1,i_2,\ldots,i_{k-1},i) \in R$. Note that $\upc_{\ell/a - ki}[r] \le \upc_{\ell/a}[r]$ for $1 \le r \le k-1$. Therefore it suffices to show that
	$i \le \upc_{\ell/a}[k]$.
	Since $t_k >1$, we have
	\[\upc_{\ell/a}[k] = \floor{\frac{\ell/a - \sum_{1 \le s \le k-1} si_s}{k}}.\]
	since
	\[ i_{k-1} \le \upc_{\ell/a - ki}[k-1] \le \frac{\ell/a - ki - \sum_{1 \le s \le k - 2} si_s}{k - 1}, \]
	we have
	\[ i \le \frac{\ell/a - \sum_{1 \le s \le k - 1} si_s}{k} \]
	and hence $i \le \upc_{\ell/a}[k]$.

	For the reverse inclusion, take $(i_1,i_2,\ldots,i_{k}) \in R'$. The condition $i_k \le \upc_{\ell/a}[k]$ implies that
	\begin{align}\label{eq: i_k condition}
	i_k \le \floor{\frac{\ell/a - \sum_{1 \le s \le k-1} si_s}{k}} \le \floor{\frac{\ell}{ka}}.
	\end{align}
	Therefore it suffices to show that $i_r \le \upc_{\ell/a - k i_k}[r]$ for all $1\le r \le k-1$. Let $r \in \{1,2,\ldots, k-1\}$. If $t_r = 0$ then we have $i_r \le \upc_{\ell/a}[r] = 0 = \upc_{\ell/a - ki_k}[r]$. Therefore, our claim holds in this case. Assume that $t_r > 0$. The first inequality in~\eqref{eq: i_k condition} implies that for each $1 \le r \le k-1$,
	\[ \sum_{r \le s \le k-1} s i_s \le \ell/a - k i_k - \sum_{1 \le s \le r-1}s i_s. \]
	Since $i_s$'s are nonnegative for all $1 \le s \le k$, this inequality gives
	\[ i_r \le \floor{\frac{\ell/a - k i_k - \sum_{1 \le s \le r-1}s i_s}{r}} = \upc_{\ell/a - k i_k}[r],\]
	as required.
	
	Applying $R = R'$ to~\eqref{eq: card of M(1;v)}, we finally have
	\[ \left|\bM_{\ell}(1;\bnu) \right| = \sum_{i_1 = 0}^{\upc_{\ell/a}[1] } \sum_{i_2 = 0}^{\upc_{\ell/a}[2] } \cdots \sum_{i_k = 0}^{\upc_{\ell/a}[k] }   \prod_{r=1}^k \binom{i_r + t_r -1-\delta_{1,r}}{t_r -1-\delta_{1,r}}. \qedhere\]
\end{proof}

If there is no danger of confusion on $\ell$, we write
$$ \sfm(\bnu) = \left|\bM_\ell(1;\bnu) \right|.$$
From now on, we will compute the number of $\left(\pclp{\ell}\right)^H$ for all subgroups $H$ of $C_\congN$ for the $H$-action induced by~\eqref{eq: CSP action}.
For instance, in case where $\g =E_6^{(1)}$, we showed in~\eqref{eq: E61 Cact fixed} that
\begin{align}\label{eq: fixed E61}
\left( \pclp{\ell} \right)^{C_3} = \left\{ \sum_{0 \le i \le 6} m_i\Lambda_i \in \pclp{\ell}  \; \middle| \; m_{0} = m_{1} = m_{6},~m_2=m_3=m_5 \right\}.
\end{align}
Since $a^\vee_0 = a^\vee_1 = a^\vee_6 = 1$, $a^\vee_2 = a^\vee_3 = a^\vee_5 = 2$ and $a^\vee_4 = 3$, by \eqref{eq: fixed E61}, we have
\[\left|\left( \pclp{\ell} \right)^{C_3} \right| = \left| \left\{  ( m_0,m_1,m_2) \in \Z_{\ge 0}^3 \; \middle| \; 3m_0 + 6 m_1 + 3m_2 = \ell \right\}  \right|.   \]
Thus $\left|\left( \pclp{\ell} \right)^{C_3}\right|$ is equal to $\sfm(3^2, 6^1)$. Similarly, for $B_n^{(1)}, C_n^{(1)}, A_{2n-1}^{(2)},E_6^{(1)}$, and $E_7^{(1)}$, there exists a unique $\bnu$ such that $\left|\left( \pclp{\ell} \right)^{C_3}\right|$ is equal to $\sfm(\bnu)$. We list all $\bnu$'s in Table~\ref{table: fixed by Cact}:

\begin{table}[h]
\centering
{ \arraycolsep=1.6pt\def\arraystretch{1.5}
\begin{tabular}{c|c|c}
Types & $\left( \pclp{\ell} \right)^{C_\congN}$ & $\bnu$ \\ \hline
$B_{n}^{(1)}$ & $\left\{ \sum_{0 \le i \le n} m_i\Lambda_i \in \pclp{\ell}  \; \middle| \; m_0 = m_n \right\}$ & $(1^1, 2^{n-1})$ \\ \hline
$C_{n}^{(1)}~(n\equiv_2 1)$ & $\left\{ \sum_{0 \le i \le n} m_i\Lambda_i \in \pclp{\ell}  \; \middle| \; m_{2j} = m_{2j+1}~(0\le j \le \frac{n-1}{2}) \right\}$ & $(2^{(n+1)/2})$ \\ \hline
$C_{n}^{(1)}~(n\equiv_2 0)$ & $\left\{ \sum_{0 \le i \le n} m_i\Lambda_i \in \pclp{\ell}  \; \middle| \; m_{2j} = m_{2j+1}~(0\le j \le \frac{n-2}{2}) \right\}$ & $(1^1, 2^{n/2})$ \\ \hline
$A_{2n-1}^{(2)}~(n\equiv_2 1)$ & $\left\{ \sum_{0 \le i \le n} m_i\Lambda_i \in \pclp{\ell}  \; \middle| \; m_{2j} = m_{2j+1}~(0\le j \le \frac{n-1}{2}) \right\}$ &  $(2^1, 4^{(n-1)/2})$ \\ \hline
$A_{2n-1}^{(2)}~(n\equiv_2 0)$ & $\left\{ \sum_{0 \le i \le n} m_i\Lambda_i \in \pclp{\ell}  \; \middle| \; m_{2j} = m_{2j+1}~(0\le j \le \frac{n-2}{2}) \right\}$ & $(2^2, 4^{(n-2)/2})$  \\ \hline
$D_{n+1}^{(2)}$ & $\left\{ \sum_{0 \le i \le n} m_i\Lambda_i \in \pclp{\ell}  \; \middle| \; m_0 = m_n \right\}$ & $(2^n)$ \\ \hline
$E_{6}^{(1)}$ & $\left\{ \sum_{0 \le i \le 6} m_i\Lambda_i \in \pclp{\ell}  \; \middle| \; m_{0} = m_{1} = m_{6},~m_2=m_3=m_5 \right\}$ & $(3^2, 6^1)$  \\ \hline
$E_{7}^{(1)}$ & $\left\{ \sum_{0 \le i \le 7} m_i\Lambda_i \in \pclp{\ell}  \; \middle| \; m_{0} = m_{7},~m_1=m_6,~m_3=m_5 \right\}$ & $(2^2, 4^2, 6^1)$
\end{tabular}
}\\[1.5ex]
\caption{$\left( \dwt_\ell \right)^{C_{\congN}}$ and the corresponding $\bnu$ for other types}
\protect\label{table: fixed by Cact}
\end{table}

For $D_n^{(1)} (n\equiv_2 1)$ type, recall the $C_4$-action on $\pclp{\ell}$ given in~\eqref{eq: C4 action on D_n}. In~\eqref{eq: Dn1 fixed by Cact} and~\eqref{eq: Dn1 fixed by Cact square}, we showed that
\[\left( \pclp{\ell} \right)^{C_4} = \left\{ \sum_{0 \le i \le n} m_i \Lambda_i \in \pclp{\ell} \; \middle| \; m_0 = m_1 = m_{n-1} = m_n,~
m_{2j} = m_{2j+1}\quad \text{for $1\le j \le \frac{n-3}{2}$} \right\}\]
and
\[\left( \pclp{\ell} \right)^{\siggen_4^2} =  \left\{ \sum_{0 \le i \le n} m_i \Lambda_i \in \pclp{\ell} \; \middle| \; m_0 = m_{n-1},~ m_1 = m_n \right\}.\]
Therefore we have
\[\left|\left( \pclp{\ell} \right)^{C_4}\right| = \sfm(4^{(n-1)/2}) \quad \text{and} \quad \left| \left( \pclp{\ell} \right)^{\siggen_4^2}\right| = \sfm( 2^{n-1} ).\]

For $D_n^{(1)} (n\equiv_2 0)$ type, recall the $C_2\times C_2$-action on $\pclp{\ell}$ given in \eqref{eq: C2 times C2 action}. In~\eqref{Dn1 even fixed by Cact_1} and~\eqref{eq: fixed by Cact2}, we showed that
\[\left( \pclp{\ell} \right)^{(\siggen_2,e)} = \left\{ \sum_{0 \le i \le n} m_i \Lambda_i \in \pclp{\ell} \; \middle| \; m_0 = m_n,~m_{1} = m_{n-1}\right\}\]
and
\[\left( \pclp{\ell} \right)^{(e,\siggen_2)} = \left\{ \sum_{0 \le i \le n} m_i \Lambda_i \in \pclp{\ell} \; \middle| \;
m_{2j} = m_{2j+1},~\text{for $j = 0,1,\ldots, \frac{n-4}{2}$},~ m_{n-1} = m_n \right\}.\]
Therefore we have
\[\left| \left( \pclp{\ell} \right)^{(\siggen_2,e)}\right| = \sfm(2^{n-1}) \quad \text{and} \quad \left|\left( \pclp{\ell} \right)^{(e,\siggen_2)}\right| = \sfm(2^3, 4^{(n-4)/2}).\]

To summarize, we have the following closed formula for $|\mx^+(\Lambda)|$ for each $\Lambda \in \DRpclp$.

\begin{theorem}\label{thm: cardinality}
For each $\Lambda \in \DRpclp$, $|\mx^+(\Lambda)|$ is given as in Table \ref{table: num of dom wts}.
\begin{table}[h]
\centering
{ \arraycolsep=1.6pt\def\arraystretch{1.5}
\begin{tabular}{c|c}
Types & $|\mx^+(\Lambda)|$ \quad $\left(\Lambda = (\ell-1)\Lambda_0 + \Lambda_i \in \DRpclp\right)$ \\ \hline
$A_{n}^{(1)}$ & $ \displaystyle{\sum_{d \mid (n+1,\ell,i)} \dfrac{d}{(n+1)+\ell} \sum_{d'|(\frac{n+1}{d},\frac{\ell}{d})} \mu(d') \matr{((n+1)+\ell)/dd'}{\ell/dd'}}$  \\ \hline
$B_{n}^{(1)}$ & $\frac{1}{2}\left(\sfm(1^3, 2^{n-2}) - \sfm(1^1, 2^{n-1})\right) + \delta_{i,0} \sfm(1^1, 2^{n-1})$ \\ \hline
$C_{n}^{(1)}~(n \equiv_2 1)$ & $\frac{1}{2} \left( \sfm(1^{n+1}) - \sfm(2^{(n+1)/2}) \right) + \delta_{i,0} \sfm(2^{(n+1)/2}) $\\ \hline
$C_{n}^{(1)}~(n \equiv_2 0)$ & $\frac{1}{2} \left( \sfm(1^{n+1}) - \sfm(1^1, 2^{n/2}) \right) + \delta_{i,0} \sfm(1^1, 2^{n/2})$ \\ \hline
$D_{n}^{(1)}~(n\equiv_2 1)$ & $\frac{1}{4} \left(\sfm(1^{4}, 2^{n-3}) - \sfm(2^{n-1}) \right) + \frac{\delta(i = 0,1)}{2} \left( \sfm(2^{n-1}) - \sfm(4^{(n-1)/2}) \right) + \delta_{i,0} \sfm(4^{(n-1)/2})$ \\ \hline
$D_{n}^{(1)}~(n\equiv_2 0)$ & $
\frac{1}{4} \left( \sfm(1^{4}, 2^{n-3}) - \sfm(2^{n-1}) \right) + \frac{\delta(i = 0,1)}{2} \left(\sfm(2^{n-1}) - \sfm(2^3, 4^{(n-4)/2}) \right) + \delta_{i,0} \sfm(2^3, 4^{(n-4)/2})$ \\ \hline
$A_{2n-1}^{(2)}~(n \equiv_2 1)$ & $\frac{1}{2} \left( \sfm(1^2, 2^{n-1}) - \sfm(2^1, 4^{(n-1)/2}) \right) + \delta_{i,0} \sfm(2^1, 4^{(n-1)/2})$ \\ \hline
$A_{2n-1}^{(2)}~(n \equiv_2 0)$ & $\frac{1}{2} \left(\sfm(1^2, 2^{n-1}) - \sfm(2^2, 4^{(n-2)/2}) \right) + \delta_{i,0} \sfm(2^2, 4^{(n-2)/2})$ \\ \hline
$A_{2n}^{(2)}$ &  $\sfm(1^1, 2^{n})$ \\ \hline
$D_{n+1}^{(2)}$ & $\frac{1}{2} \left( \sfm(1^2, 2^{n-1}) - \sfm(2^n) \right) + \delta_{i,0} \sfm(2^n)$ \\ \hline
$F_{4}^{(1)}$ & $\sfm(1^2, 2^2, 3)$ \\ \hline
$E_{6}^{(2)}$ & $\sfm(1^1, 2^2, 3^1, 4^1)$ \\ \hline
$G_{2}^{(1)}$ & $\sfm(1^2, 2^1)$ \\ \hline
$D_{4}^{(3)}$ & $\sfm(1^1, 2^1, 3^1)$ \\ \hline
$E_{6}^{(1)}$ & $\frac{1}{3} \left(\sfm(1^3, 2^3, 3^1) - \sfm(3^2, 6^1) \right) + \delta_{i,0} \sfm(3^2, 6^1) $ \\ \hline
$E_{7}^{(1)}$ & $\frac{1}{2} \left(\sfm(1^2, 2^3, 3^2, 4^1) - \sfm(2^2, 4^2, 6) \right) + \delta_{i,0} \sfm(2^2, 4^2, 6)$ \\ \hline
$E_{8}^{(1)}$ & $\sfm(1^1, 2^2, 3^2, 4^2, 5^1, 6^1)$
\end{tabular}
}\\[1.5ex]
\caption{Closed formulae for $|\mx^+(\Lambda)|$}
\protect\label{table: num of dom wts}
\end{table}
\end{theorem}

In Table~\ref{table: num of dom wts}, $\sfm(\bnu)$'s appearing in the closed formulae for $|\mx^+(\Lambda)|$ for the classical types can be expressed in terms of binomial coefficients as follows:
\begin{align*}
&\sfm(1^{n}) = \matr{\ell + n - 1}{\ell},~ \sfm(1^1,2^{n}) = \matr{\floor{\frac{\ell}{2}}+ n}{\floor{\frac{\ell}{2}}},\ \sfm(1^2, 2^{n}) = 
\matr{\floor{\frac{\ell}{2}} + n + 1}{\floor{\frac{\ell}{2}}} + \matr{\floor{\frac{\ell-1}{2}} + n + 1}{\floor{\frac{\ell-1}{2}}},
\\
&\sfm(1^3,2^{n}) = 2\matr{\floor{\frac{\ell}{2}} + n + 1}{\floor{\frac{\ell}{2}} - 1} + 2\matr{\floor{\frac{\ell-1}{2}} + n + 2}{\floor{\frac{\ell - 1}{2}}} + \matr{\floor{\frac{\ell}{2}} + n + 1}{\floor{\frac{\ell}{2}}},\\
&\sfm(1^{4},2^{n-3}) = \begin{cases}
8\matr{\frac{\ell}{2} + n - 1}{\frac{\ell}{2} - 1} + \matr{\frac{\ell}{2} + n - 2}{\frac{\ell}{2}} & \text{if $\ell \equiv_2 0$},\\
4\matr{\frac{\ell-1}{2} + n}{\frac{\ell-1}{2}} + 4\matr{\frac{\ell-1}{2} + n - 1}{\frac{\ell-1}{2} - 1} &\text{if $\ell \equiv_2 1$},
\end{cases}\\
&\sfm(2^{n}) = \delta(\ell \equiv_2 0)\matr{\frac{\ell}{2} + n - 1}{\frac{\ell}{2}},\quad
\sfm(2^{1},4^{n}) = \delta(\ell \equiv_2 0)\matr{\floor{\frac{\ell}{4}} + n}{\floor{\frac{\ell}{4}}}, \\
&\sfm(2^{2},4^{n}) = \delta(\ell \equiv_2 0)\left(\matr{\floor{\frac{\ell-2}{4}} + n + 1}{\floor{\frac{\ell-2}{4}}} + \matr{\floor{\frac{\ell}{4}} + n + 1}{\floor{\frac{\ell}{4}}}\right),\\
&\sfm(2^{3},4^{n}) = \delta(\ell \equiv_2 0)\left(2\matr{\floor{\frac{\ell}{2}} + n + 1}{\floor{\frac{\ell}{2}} - 1} + 2\matr{\floor{\frac{\ell - 1}{2}} + n + 2}{\floor{\frac{\ell - 1}{2}}} + \matr{\floor{\frac{\ell}{2}} + n + 1}{\floor{\frac{\ell}{2}}}\right),\\
&\sfm(4^{n}) = \delta(\ell \equiv_4 0)\matr{\frac{\ell}{4} + n - 1}{\frac{\ell}{4}} .
\end{align*}

For $B_n^{(1)}, C_n^{(1)}, A_{2n}^{(2)}, D_{n+1}^{(2)}$, the formulae in Table~\ref{table: num of dom wts} can be expressed in terms of binomial coefficients. In particular, for $C_n^{(1)}$ type, the formula in even case and that in odd case can be merged
regardless of rank.

\begin{corollary} \label{cor: BCAD binom}
Let $\g = B_n^{(1)}, C_n^{(1)}, A_{2n}^{(2)}, D_{n+1}^{(2)}$. Then, for each $\Lambda \in \DRpclp$, we have a closed formula for $|\mx^+(\Lambda)|$ in terms of binomial coefficients as in Table~\ref{table: num of dom wts_binom}.
\begin{table}[h]
\centering
\fontsize{9}{9}\selectfont
\begin{tabular}{ c | c | c | c}
Types & $|\mx^+(\Lambda)|$  \quad $\left(\Lambda = (\ell-1)\Lambda_0 + \Lambda_i \in \DRpclp\right)$ & Types & $|\mx^+(\Lambda)|$ \quad $\left(\Lambda = (\ell-1)\Lambda_0 + \Lambda_i \in \DRpclp\right)$ \\   \hline
$B_{n}^{(1)}$ &
$\matr{n+\floor{\frac{\ell-\delta_{in}}{2}}}{n}+\matr{n+\floor{\frac{\ell-1-\delta_{in}}{2}}}{n}$ &
$C_{n}^{(1)}$ &
$\dfrac{1}{2} \left( \matr{n + \ell}{n} + (-1)^i  \delta(n \ell\equiv_2 0) \matr{ \lfloor \frac{n+\ell}{2} \rfloor }{\lfloor \frac{n}{2} \rfloor}  \right)$ \\[2ex] \hline
$A_{2n}^{(2)}$ &
$\matr{n+ \floor{\frac{\ell}{2}}}{n}$ &
$D_{n+1}^{(2)}$ &
$\matr{n + \floor{\frac{\ell-\delta_{in}}{2}}}{n}$
\end{tabular}
\fontsize{10}{10}\selectfont
\caption{Closed formulae for $|\mx^+(\Lambda)|$ in terms of binomial coefficients}
\protect\label{table: num of dom wts_binom}
\end{table}
\end{corollary}
\subsection{Consequences of our closed formulae}\label{subsec: inductive and triangular}
In this subsection, we introduce two consequences coming from our closed formulae, recursive formulae and level-rank duality for $|\mx^+(\Lambda)|$.

\subsubsection{Recursive formulae(Triangular arrays)}\label{subsubsec: inductive}
The formulae in Table~\ref{table: num of dom wts} and Table~\ref{table: num of dom wts_binom} enable us to compute $|\mx^+(\Lambda)|$ recursively. We list recursive formulae for $|\mx^+(\Lambda)|$ for all types except for $A_n^{(1)}$ and exceptional types. For clarity, we use $\mx_\g(\Lambda)$ to denote the set of dominant maximal weights to emphasize the rank of affine Kac-Moody algebras in consideration.
Here we deal with only the $C_n^{(1)}$ type. For other types, see Appendix A.

Define $\triT^{C^{(1)}}_0:\Z_{\ge 0} \times \Z_{\ge 0 } \ra \Z_{\ge 0 }$ by
\begin{align}\label{eq: recursiv rel}
&\triT^{C^{(1)}}_0(n,0) = 1 ~ (n \ge 0), \quad \triT^{C^{(1)}}_0(0,\ell) = 1 ~ (\ell \ge 1),~\text{and} \nonumber\\
&\triT^{C^{(1)}}_0(n,\ell) = \triT^{C^{(1)}}_0(n,\ell-1) + \triT^{C^{(1)}}_0(n-1,\ell) - \delta(n\ell \equiv_2 1) \matr{\frac{n+\ell}{2} - 1}{\frac{\ell-1}{2}} ~(n\ge 1,~\ell \ge 1). 
\end{align}
Using the formula in Table~\ref{table: num of dom wts_binom}, one can see that $|\mx^+_{C_n^{(1)}}(0)| = \triT^{C^{(1)}}_0(n,0)$ $(n \ge 2)$,  $|\mx^+_{C_2^{(1)}}(\ell\Lambda_0)| = \triT^{C^{(1)}}_0(2,\ell)$ $(\ell \ge 0)$, and $|\mx^+_{C_n^{(1)}}(\ell \Lambda_0)|$ $(n \ge 3 , \ell \ge 1)$ satisfies the same recursive relation as~\eqref{eq: recursiv rel}.
Therefore we can conclude that $\triT^{C^{(1)}}_0(n,\ell) = |\mx^+_{C_n^{(1)}}(\ell \Lambda_0)|$ for all $n \in \Z_{\ge 2},~\ell \in \Z_{\ge 0 }$. In particular, as a triangular array, $\triT^{C^{(1)}}_0$ can be described as follows: \fontsize{7}{7}\selectfont
\begin{equation*}
\begin{tikzpicture}
\node[] {
$\begin{array}{ccccccccccccccccccccccccc}
&&&&&&&\scriptstyle{n = 0} && \scriptstyle{\ell = 0}\\
&&&&&&\scriptstyle{n = 1}&&1&&\scriptstyle{\ell = 1} \\
&&&&&\scriptstyle{n = 2}&&1&&1 && \scriptstyle{\ell = 2} \\
&&&&\scriptstyle{n = 3}&&1&&1&&1 && \scriptstyle{\ell = 3}\\
&&&\scriptstyle{n = 4}&&1&&2&&2&&1 && \scriptstyle{\ell = 4}\\
&&\scriptstyle{n = 5}&&1&&2&&4&&2&&1 && \scriptstyle{\ell = 5}\\
&\scriptstyle{n = 6}&&1&&3&&6&&6&&3&&1 && \scriptstyle{\ell = 6}\\
\scriptstyle{n = 7}&&1&&3&&9&&10&&9&&3&&1 && \scriptstyle{\ell = 7}\\
&1&&4&&12&&19&&19&&12&&4&&1 \\
\udots&&\udots&&\udots&&\udots&&\vdots&&\ddots&&\ddots&&\ddots&&\ddots &,
\end{array}$
};
\draw (-6,-.6) -- (-0.03,1.2);
\draw (5.6,-.6) -- (-0.03,1.2);
\end{tikzpicture}
\end{equation*} \fontsize{10}{10}\selectfont
which is known to be \emph{Lozani\'{c}'s triangle} \cite[A034851]{OEIS}. 
To compare $\triT^{C^{(1)}}_0(n,\ell)$ with $|\mx^+_{C_n^{(1)}}(\ell \Lambda_0)|$,  we write the triangular array $\triT^{C^{(1)}}_0$ in a different direction from \cite{OEIS}. In the rest of this paper, we use this convention for triangular arrays.

Similarly, define $\triT^{C^{(1)}}_1:\Z_{\ge 0} \times \Z_{\ge 0 } \ra \Z_{\ge 0 }$ by
\begin{align*}
&\triT^{C^{(1)}}_1(n,0) = 0 ~ (n \ge 0), \quad \triT^{C^{(1)}}_1(0,\ell) = 0 ~ (\ell \ge 1),~\text{and}\\
&\triT^{C^{(1)}}_1(n,\ell) = \triT^{C^{(1)}}_1(n,\ell-1) + \triT^{C^{(1)}}_1(n-1,\ell) + \delta(n,\ell \equiv_2 1) \matr{\frac{n+\ell}{2} - 1}{\frac{\ell-1}{2}} ~(n\ge 1,~\ell \ge 1).
\end{align*}
Then for all $n \in \Z_{\ge 2},~\ell \in \Z_{\ge 0 }$, we have $\triT^{C^{(1)}}_1(n,\ell) = |\mx^+_{C^{(1)}}((\ell-1)\Lambda_0 + \Lambda_1)|$. In particular, as a triangular array, $\triT^{C^{(1)}}_1$ can be described as follows: \fontsize{7}{7}\selectfont
\begin{equation*}
\begin{tikzpicture}
\node[] {
$\begin{array}{ccccccccccccccccccccccccc}
&&&&&&&\scriptstyle{n = 0} && \scriptstyle{\ell = 0}\\
&&&&&&\scriptstyle{n = 1}&&0 && \scriptstyle{\ell = 1} \\
&&&&&\scriptstyle{n = 2}&&0&&0&& \scriptstyle{\ell = 2} \\
&&&&\scriptstyle{n = 3}&&0&&1&&0&& \scriptstyle{\ell = 3} \\
&&&\scriptstyle{n = 4}&&0&&1&&1&&0&& \scriptstyle{\ell = 4} \\
&&\scriptstyle{n = 5}&&0&&2&&2&&2&&0&& \scriptstyle{\ell = 5} \\
&\scriptstyle{n = 6}&&0&&2&&4&&4&&2&&0&& \scriptstyle{\ell = 6} \\
\scriptstyle{n = 7}&&0&&3&&6&&10&&6&&3&&0&& \scriptstyle{\ell = 7} \\
&0&&3&&9&&16&&16&&9&&3&&0 \\
\udots&&\udots&&\udots&&\udots&&\vdots&&\ddots&&\ddots&&\ddots&&\ddots &,
\end{array}$
};
\draw (-6,-.6) -- (-0.03,1.2);
\draw (5.6,-.6) -- (-0.03,1.2);
\end{tikzpicture}
\end{equation*} \fontsize{10}{10}\selectfont
which is known in \cite[A034852]{OEIS} and $\triT^{C^{(1)}}_1 = \text{(Pascal triangle)} - \triT^{C^{(1)}}_0$.

\subsubsection{Level-rank duality} \label{sec: LR-duality}
From closed formulae in the subsection~\ref{subsec: The number of dom max wts} and triangular arrays in the subsection~\ref{subsubsec: inductive}, we observe a very noteworthy symmetry, called \emph{level-rank duality}, between certain sets of dominant maximal weights except for $D_n^{(1)}$ and $A_{2n-1}^{(2)}$.
Here we deal with only the $A_n^{(1)}$ type. For other types, see Appendix B.

Let $\Lambda = (\ell-1)\Lambda_0 + \Lambda_i \in \DRpclp$. From Table~\ref{table: num of dom wts} it follows that
\begin{align*}
    \left|\mx_{A_n^{(1)}}^+(\Lambda)\right| = \sum_{d \mid (n+1,\ell,i)} \dfrac{d}{(n+1)+\ell} \sum_{d'|(\frac{n+1}{d},\frac{\ell}{d})} \mu(d') \matr{((n+1)+\ell)/dd'}{\ell/dd'}.
\end{align*}
Hence, for $n \ge 1$ and $\ell >1$, if $(n+1,\ell,i) = (\ell, n+1,j)$ for some $0 \le i,j \le \min(n,\ell)$, then
\begin{align}\label{eq: An level-rank duality}
    \left|\mx_{A_n^{(1)}}^+((\ell-1)\Lambda_0 + \Lambda_i)\right| = \left|\mx_{A_{\ell-1}^{(1)}}^+(n\Lambda_0 + \Lambda_j)\right|,
\end{align}
i.e., when we exchange $n+1$ with $\ell$, the number of dominant maximal weights remains same.

Let us deal with the relation between our duality and Frenkel's duality in \cite{F}. For a residue $i$ modulo $n$, let $\Lambda_i^{(n)}$ denote the $i$th fundamental weight of $A_{n}^{(1)}$. For $\Lambda = \sum_{i=0}^n m_i\Lambda_i^{(n)} \in \pclp{\ell}$, let $\Lambda'$ be the dominant integral weight of $A_{\ell-1}^{(1)}$ defined by 
\[ \Lambda' =  \sum_{i=0}^{n} \Lambda_{m_i + m_{i+1} + \cdots + m_n}^{(\ell-1)} \in \pclp{n + 1}.\]
With this setting, Frenkel found the following duality between the $q$-specialized characters of $V(\Lambda)$ and $V(\Lambda')$:
\[ \dim_q(V(\Lambda)) \prod_{k = 0}^\infty \dfrac{1}{1 - q^{(n+1)k}} = \dim_q(V(\Lambda')) \prod_{k=0}^\infty \dfrac{1}{1-q^{\ell k}},\]
where $\dim_q(V)$ is the $q$-specialized character of $V$ (see \cite[Theorem 2.3]{F} or \cite[Subsection 4.4]{T}).

Now we will show that
\[ \left|\mx_{A_n^{(1)}}^+(\Lambda)\right| = \left|\mx_{A_{\ell-1}^{(1)}}^+(\Lambda')\right|. \]
Recall $\evS$ in~\eqref{eq: S-evaluation}. Letting $\Lambda \in \pclp{\ell}((\ell-1)\Lambda_0 + \Lambda_{i_0})$ and $\Lambda' \in \pclp{n+1}(n\Lambda_0 + \Lambda_{j_0})$, then $\evS(\Lambda) \equiv_{n+1} i_0$ and $\evSp(\Lambda') \equiv_{\ell} j_0$ by Theorem~\ref{thm: eq rel and sieving}. Here $\svs = (1,2,\ldots, n)$ and $\svs' = (1,2,\ldots,\ell-1)$. On the other hand, by the definition of $\Lambda'$, we have
\[\evSp(\Lambda') \equiv_\ell \sum_{i = 0}^n (i+1) m_i, \]
which implies that
\begin{align}\label{eq: ev La - ev La'}
\evSp(\Lambda') - \evS(\Lambda) \equiv_\ell \sum_{i = 0}^n (i+1) m_i - \sum_{i = 0}^n i m_i = \sum_{i = 0}^n m_i = \ell \equiv_\ell 0.
\end{align}
By~\eqref{eq: ev La - ev La'}, we have
\[ j_0 - i_0 \equiv_{(n+1, \ell)} \evSp(\Lambda') - \evS(\Lambda) \equiv_{(n+1, \ell)} 0. \]
Therefore $(n+1,\ell,i_0) = (\ell,n+1,j_0)$ and our assertion follows from~\eqref{eq: An level-rank duality}.

\appendix

\section{Recursive formulae for $|\mx^+(\Lambda)|$} \label{AppendiX: Recursive}

\noindent{\bf $A_{2n}^{(2)}$ type.}
Define $\triT^{A^{(2)}_{\even}}_0:\Z_{\ge 0} \times \Z_{\ge 0 } \ra \Z_{\ge 0 }$ by
\begin{align}\label{eq: recursive A2n2}
&\triT^{A^{(2)}_{\even}}_0(n,0) = \triT^{A^{(2)}_{\even}}_0(n,1) = 1 ~ (n \ge 0), \quad \triT^{A^{(2)}_{\even}}_0(0,\ell) = 1 ~ (\ell \ge 2),~\text{and} \nonumber \\
&\triT^{A^{(2)}_{\even}}_0(n,\ell) = \triT^{A^{(2)}_{\even}}_0(n,\ell-2) + \triT^{A^{(2)}_{\even}}_0(n-1,\ell).
\end{align}
Then for all $n \in \Z_{\ge 2},~\ell \in \Z_{\ge 0 }$, we have $\triT^{A^{(2)}_{\even}}_0(n,\ell) = |\mx^+_{A_{2n}^{(2)}}(\ell \Lambda_0)|$.  In particular, as a triangular array, $\triT_0$ can be described as follows: \fontsize{7}{7}\selectfont
\begin{equation*}
\begin{array}{ccccccccccccccccccccccccc}
&&&&&&&&1 \\
&&&&&&&1&&1 \\
&&&&&&\mycirc{1}&&1&&1 \\
&&&&&1&&1&&\mycirc{2} &&1 \\
&&&&1&&1&&\mydoublecirc{3}&&2&&1 \\
&&&1&&1&&4&&3&&3&&1 \\
&&1&&1&&5&&4&&6&&3&&1 \\
&1&&1&&6&&5&&10&&6&&4&&1 \\
\udots&&\udots&&\udots&&\udots&&\vdots&&\ddots&&\ddots&&\ddots&&\ddots &,
\end{array}
\end{equation*} \fontsize{10}{10}\selectfont
which is known in \cite[A065941]{OEIS}. It is Pascal's triangle with duplicated diagonals, i.e., $\triT^{A^{(2)}_{\even}}_0(n,2\ell) = \triT^{A^{(2)}_{\even}}_0(n,2\ell + 1) = \matr{n + \ell}{\ell}$ for $n, \ell \ge 0$. 
The recursive condition~\eqref{eq: recursive A2n2} says if we add the circled $1$ and the circled $2$ then we get double circled $3$.

\noindent{\bf $B_n^{(1)}$ type.} Define $\triT^{B^{(1)}}_0:\Z_{\ge 0} \times \Z_{\ge 0 } \ra \Z_{\ge 0 }$ by
\begin{align*}
&\triT^{B^{(1)}}_0(n,0) = 1 ~ (n \ge 1), \quad \triT^{B^{(1)}}_0(n,1) = 2 ~ (n \ge 1), \quad \triT^{B^{(1)}}_0(0,\ell) = 2 ~ (\ell \ge 0),~\text{and}\\
&\triT^{B^{(1)}}_0(n,\ell) = \triT^{B^{(1)}}_0(n,\ell-2) + \triT^{B^{(1)}}_0(n-1,\ell) ~(n\ge 1,~\ell \ge 2).
\end{align*}
Then for all $n \in \Z_{\ge 3},~\ell \in \Z_{\ge 0 }$, we have $\triT^{B^{(1)}}_0(n,\ell) = |\mx^+_{B_n^{(1)}}(\ell \Lambda_0)|$.  In particular, as a triangular array, $\triT_0^{B^{(1)}}$ can be described as follows:
\fontsize{7}{7}\selectfont
\begin{equation*}
\begin{array}{ccccccccccccccccccccccccc}
&&&&&&&&2 \\
&&&&&&&1&&2 \\
&&&&&&1&&2&&2 \\
&&&&&1&&2&&3&&2 \\
&&&&1&&2&&4&&4&&2 \\
&&&1&&2&&5&&6&&5&&2 \\
&&1&&2&&6&&8&&9&&6&&2 \\
&1&&2&&7&&10&&14&&12&&7&&2 \\
\udots&&\udots&&\udots&&\udots&&\vdots&&\ddots&&\ddots&&\ddots&&\ddots & .
\end{array}
\end{equation*} \fontsize{10}{10}\selectfont
This array is the triangular array obtained by removing the left boundary diagonal from the triangular array in \cite[A129714]{OEIS} whose row sums are the Fibonacci numbers.

Define $\triT^{B^{(1)}}_n:\Z_{\ge 0} \times \Z_{\ge 0 } \ra \Z_{\ge 0 }$ by
\begin{align*}
&\triT^{B^{(1)}}_n(n,0) = \redo~(n \ge 0) \quad \text{and} \quad \triT^{B^{(1)}}_n(n,\ell) = \triT^{B^{(1)}}_0(n,\ell - 1) \quad \text{for $n \ge 0, \ell>0$}.
\end{align*}
Then for all $n \in \Z_{\ge 3},~\ell \in \Z_{> 0 }$, we have $\triT^{B^{(1)}}_n(n,\ell) = |\mx^+_{B_n^{(1)}}((\ell-1)\Lambda_0 + \Lambda_n)|$.
In particular, as a triangular array, $\triT_n^{B^{(1)}}$ can be described as follows:
\fontsize{7}{7}\selectfont
\begin{equation*}
\begin{array}{ccccccccccccccccccccccccc}
&&&&&&&&\redo \\
&&&&&&&\redo&&2 \\
&&&&&&\redo&&1&&2 \\
&&&&&\redo&&1&&2&&2 \\
&&&&\redo&&1&&2&&3&&2 \\
&&&\redo&&1&&2&&4&&4&&2 \\
&&\redo&&1&&2&&5&&6&&5&&2 \\
&\redo&&1&&2&&6&&8&&9&&6&&2 \\
\udots&&\udots&&\udots&&\udots&&\vdots&&\ddots&&\ddots&&\ddots&&\ddots & .
\end{array}
\end{equation*} \fontsize{10}{10}\selectfont
Note that Table~\ref{table: num of dom wts_binom} says
\[\left|\mx^+_{B_n^{(1)}}((\ell-1)\Lambda_0 + \Lambda_n)\right| =
\matr{n+\floor{\frac{\ell-1}{2}}}{n}+\matr{n+\floor{\frac{\ell-2}{2}}}{n}
= \left|\mx^+_{B_n^{(1)}}((\ell-1)\Lambda_0)\right|  \quad \text{for $\ell > 0$},\]
which explains the reason why we define $\triT^{B^{(1)}}_n(n,\ell)$ as $\triT^{B^{(1)}}_0(n,\ell - 1)$ for $n \ge 0, \ell > 0$. To emphasize this, we denote by $\redo$ the zeros in the left boundary diagonal. Note also that $\triT^{B^{(1)}}_0(n,\ell)$ and hence $\triT^{B^{(1)}}_n(n,\ell)$ can be obtained from $\triT^{A^{(2)}_{\even}}_0$ as follows:
$$  \triT^{B^{(1)}}_0(n,\ell) =  \triT^{A^{(2)}_{\even}}_0(n,\ell) + \triT^{A^{(2)}_{\even}}_0(n,\ell-1) \quad\text{ for $n,\ell \ge 1$}.$$

\noindent{\bf $D_n^{(1)}$ type.}
Define $\triT^{D^{(1)}} _0:\Z_{\ge 0} \times \Z_{\ge 0 } \ra \Z_{\ge 0 }$ by
\begin{align*}
&\triT^{D^{(1)}} _0(n,0) = \triT^{D^{(1)}} _0(n,1) = 1 ~ (n \ge 0), \triT^{D^{(1)}} _0(0,2) = 3,~ \triT^{D^{(1)}} _0(0,2\ell - 1) = 2 ~ (\ell \ge 2),~ \triT^{D^{(1)}} _0(0,2\ell) = 4~(\ell \ge 2),  \\
&\triT^{D^{(1)}} _0(n,\ell) = \triT^{D^{(1)}} _0(n,\ell-2) + \triT^{D^{(1)}} _0(n-1,\ell)\\
& \hspace{11ex} + (-1)^{n}\delta(\ell \equiv_2 0) (1+\delta(n \equiv_2 1))  \triT^{A^{(2)}_{\even}}_0\left( \left\lfloor\frac{n-1}{2} \right\rfloor, \frac{\ell}{2} - 1 \right) \quad(n \ge 1, \ell \ge 2).
\end{align*}
Then for all $n \in \Z_{\ge 4},~\ell \in \Z_{\ge 0 }$, we have $\triT^{D^{(1)}}_0(n,\ell) = |\mx^+_{D_{n}^{(1)}}(\ell \Lambda_0)|$.  In particular, as a triangular array, $\triT_0$ can be described as follows: \fontsize{7}{7}\selectfont
\begin{equation*}
\begin{array}{ccccccccccccccccccccccccc}
&&&&&&&&1 \\
&&&&&&&1&&1 \\
&&&&&&1&&1&&3 \\
&&&&&1&&1&&2&&2 \\
&&&&1&&1&&4&&3&&4 \\
&&&1&&1&&3&&4&&4&&2 \\
&&1&&1&&5&&5&&9&&5&&4 \\
&1&&1&&4&&6&&10&&9&&6&&2 \\
\udots&&\udots&&\udots&&\udots&&\vdots&&\ddots&&\ddots&&\ddots&&\ddots &.
\end{array}
\end{equation*}\fontsize{10}{10}\selectfont

Define $\triT^{D^{(1)}}_1:\Z_{\ge 0} \times \Z_{\ge 0 } \ra \Z_{\ge 0 }$ by
\begin{align*}
&\triT^{D^{(1)}}_1(n,0) = 0 ~ (n \ge 0), \quad \triT^{D^{(1)}}_1(n,1) = 1 ~ (n \ge 0), \triT^{D^{(1)}}_1(0,2\ell - 1) = 2 ~ (\ell \ge 2),~ \triT^{D^{(1)}}_1(0,2\ell) = 0~(\ell \ge 1), \\
&\triT^{D^{(1)}}_1(n,\ell) = \triT^{D^{(1)}}_1(n,\ell-2) + \triT^{D^{(1)}}_1(n-1,\ell)\\
& \hspace{11ex} + (-1)^{n+1}\delta(\ell \equiv_2 0) (1+\delta(n \equiv_2 1))  \triT^{A^{(2)}_{\even}}_0 \left( \left\lfloor\frac{n-1}{2} \right\rfloor, \frac{\ell}{2} - 1 \right) \quad(n \ge 1, \ell \ge 2).
\end{align*}
Then for all $n \in \Z_{\ge 4},~\ell \in \Z_{\ge 0 }$, we have $\triT^{D^{(1)}}_1(n,\ell) = |\mx^+_{D_{n}^{(1)}}((\ell - 1)\Lambda_0 + \Lambda_1)|$.  In particular, as a triangular array, $\triT^{D^{(1)}}_1$ can be described as follows: \fontsize{7}{7}\selectfont
\begin{equation*}
\begin{array}{ccccccccccccccccccccccccc}
&&&&&&&&0 \\
&&&&&&&0&&1 \\
&&&&&&0&&1&&0 \\
&&&&&0&&1&&2&&2 \\
&&&&0&&1&&1&&3&&0 \\
&&&0&&1&&3&&4&&4&&2 \\
&&0&&1&&2&&5&&4&&5&&0 \\
&0&&1&&4&&6&&9&&9&&6&&2 \\
\udots&&\udots&&\udots&&\udots&&\vdots&&\ddots&&\ddots&&\ddots&&\ddots &.
\end{array}
\end{equation*}\fontsize{10}{10}\selectfont

Set $\triT^{D_n^{(1)}}_n = \triT^{B_n^{(1)}}_n$.
Then, for all $n \in \Z_{\ge 4}$ and $\ell \in \Z_{\ge 0 }$, we have
\begin{align*}
|\mx^+_{B_n^{(1)}}((\ell-1)\Lambda_0 + \Lambda_n)| & = \triT^{B^{(1)}}_n(n,\ell) = \triT^{D^{(1)}}_n(n,\ell) = |\mx^+_{D_{n}^{(1)}}((\ell - 1)\Lambda_0 + \Lambda_{n-\epsilon})| \quad (\epsilon \in \{0,1\}).\end{align*}


\noindent{\bf $A_{2n-1}^{(2)}$ type.} Define $\triT_0:\Z_{\ge 0} \times \Z_{\ge 0 } \ra \Z_{\ge 0 }$ by
\begin{align*}
&\triT^{A^{(2)}_{\odd}}_0(n,0) = \triT^{A^{(2)}_{\odd}}_0(n,1) = 1 ~ (n \ge 0), \quad \triT^{A^{(2)}_{\odd}}_0(0,\ell) = 1 ~ (\ell \ge 2),~\text{and}\\
&\triT^{A^{(2)}_{\odd}}_0(n,\ell) = \triT^{A^{(2)}_{\odd}}_0(n,\ell-2) + \triT^{A^{(2)}_{\odd}}_0(n-1,\ell) - \delta(n \equiv_2 1,\ell \equiv_2 0, n > 1) \triT^{A^{(2)}_{\even}}_0 \left(\frac{n-1}{2}, \frac{\ell}{2}-1 \right)   ~(n\ge 1,~\ell \ge 2).
\end{align*}
Then for all $n \in \Z_{\ge 3},~\ell \in \Z_{\ge 0 }$, we have $\triT^{A^{(2)}_{\odd}}_0(n,\ell) = |\mx^+_{A_{2n-1}^{(2)}}(\ell \Lambda_0)|$.  In particular, as a triangular array, $\triT_0$ can be described as follows: \fontsize{7}{7}\selectfont
\begin{equation*}
\begin{array}{ccccccccccccccccccccccccc}
&&&&&&&&1 \\
&&&&&&&1&&1 \\
&&&&&&1&&1&&1 \\
&&&&&1&&1&&2&&1 \\
&&&&1&&1&&3&&2&&1 \\
&&&1&&1&&3&&3&&3&&1 \\
&&1&&1&&4&&4&&6&&3&&1 \\
&1&&1&&4&&5&&8&&6&&4&&1 \\
\udots&&\udots&&\udots&&\udots&&\vdots&&\ddots&&\ddots&&\ddots&&\ddots &.
\end{array}
\end{equation*}\fontsize{10}{10}\selectfont

Define $\triT^{A^{(2)}_{\odd}}_1:\Z_{\ge 0} \times \Z_{\ge 0 } \ra \Z_{\ge 0 }$ by
\begin{align*}
&\triT^{A^{(2)}_{\odd}}_1(n,0) = 0 ~ (n \ge 0), \quad  \triT^{A^{(2)}_{\odd}}_1(n,1) = 1 ~ (n \ge 0), \quad   \triT^{A^{(2)}_{\odd}}_1(0,\ell) = 1 ~ (\ell \ge 1),~\text{and}\\
&\triT^{A^{(2)}_{\odd}}_1(n,\ell) = \triT^{A^{(2)}_{\odd}}_1(n,\ell-2) + \triT^{A^{(2)}_{\odd}}_1(n-1,\ell) + \delta(n \equiv_2 1,\ell \equiv_2 0,n >1) \triT^{A^{(2)}_{\even}}_0\left(\frac{n-1}{2}, \frac{\ell}{2}-1 \right)   ~(n\ge 1,~\ell \ge 2).
\end{align*}
Then for all $n \in \Z_{\ge 3},~\ell \in \Z_{\ge 0 }$, we have $\triT^{A^{(2)}_{\odd}}_1(n,\ell) = |\mx^+_{A_{2n-1}^{(2)}}((\ell-1)\Lambda_0 + \Lambda_1)|$.  In particular, as a triangular array, $\triT^{A^{(2)}_{\odd}}_1$ can be described as follows: \fontsize{7}{7}\selectfont
\begin{equation*}
\begin{array}{ccccccccccccccccccccccccc}
&&&&&&&&0 \\
&&&&&&&0&&1 \\
&&&&&&0&&1&&1 \\
&&&&&0&&1&&1&&1 \\
&&&&0&&1&&1&&2&&1 \\
&&&0&&1&&2&&3&&2&&1 \\
&&0&&1&&2&&4&&3&&3&&1 \\
&0&&1&&3&&5&&6&&6&&3&&1 \\
\udots&&\udots&&\udots&&\udots&&\vdots&&\ddots&&\ddots&&\ddots&&\ddots &.
\end{array}
\end{equation*}\fontsize{10}{10}\selectfont


\noindent{\bf $D_{n+1}^{(2)}$ type.} Set $\triT^{D^{(2)}}_0 \seteq \triT_0^{A^{(2)}_{\even}}$.
Then for all $n \in \Z_{\ge 2},~\ell \in \Z_{\ge 0 }$, we have $\triT^{D^{(2)}}_0(n,\ell) = |\mx^+_{D_{n+1}^{(2)}}(\ell \Lambda_0)|$.
Note that Table~\ref{table: num of dom wts_binom} says
\[\left|\mx^+_{D_{n+1}^{(2)}}(\ell\Lambda_0)\right| = \matr{n + \floor{\frac{\ell}{2}}}{n} = \left|\mx^+_{A_{2n}^{(2)}}(\ell\Lambda_0)\right|,\]
which explains the reason why we define $\triT^{D^{(2)}}_0$ as $\triT_0^{A^{(2)}_{\even}}$.

Define $\triT_n:\Z_{\ge 0} \times \Z_{\ge 0 } \ra \Z_{\ge 0 }$ by
\begin{align*}
\triT^{D^{(2)}}_n(n,0) = \redo ~ (n \ge 0) \quad \text{and} \quad \triT^{D^{(2)}}_n(n,\ell) = \triT_0^{A^{(2)}_{\even}}(n,\ell-1)~(n \ge 0, \ell > 0).
\end{align*}
Then for all $n \in \Z_{\ge 2},~\ell \in \Z_{\ge 0 }$, we have $\triT^{D^{(2)}}_n(n,\ell) = |\mx^+_{D_{n+1}^{(2)}}((\ell-1)\Lambda_0 + \Lambda_n)|$.
Note that Table~\ref{table: num of dom wts_binom} says
\[ \left|\mx^+_{D_{n+1}^{(2)}}((\ell-1)\Lambda_0 + \Lambda_n)\right| = \matr{n + \floor{\frac{\ell-1}{2}}}{n}  = \left|\mx^+_{A_{2n}^{(2)}}((\ell-1)\Lambda_0)\right|  \quad \text{for $\ell > 0$},\]
which explains the reason why we define $\triT^{D^{(2)}}_n(n,\ell)$ as $\triT_0^{A^{(2)}_{\even}}(n,\ell - 1)$ for $n \ge 0, \ell > 0$. To emphasize this, we denote by $\redo$ the zeros in the left boundary diagonal.

\section{Level-rank duality} \label{AppendiX: Level-rank duality}

\noindent{\bf $B_{n}^{(1)}$ type.} From Table~\ref{table: num of dom wts_binom} it follows that
\begin{align*}
\left|\mx_{B_n^{(1)}}^+(\ell\Lambda_0)\right| = \matr{n+\floor{\frac{\ell}{2}}}{n}+\matr{n+\floor{\frac{\ell-1}{2}}}{n}.
\end{align*}
Hence, for $n \ge 3$, $\ell \ge 7$ and $\ell \equiv_2 1$, we have
\begin{align*}
\left|\mx_{B_n^{(1)}}^+(\ell\Lambda_0)\right| = \left|\mx_{B_{(\ell-1)/2}^{(1)}}^+((2n+1)\Lambda_0)\right|,
\end{align*}
i.e., when we exchange $n$ with $(\ell-1)/2$, the number of dominant maximal weights remains same.

From Table~\ref{table: num of dom wts_binom} it follows that
\begin{align*}
\left|\mx_{B_n^{(1)}}^+((\ell-1)\Lambda_0 + \Lambda_n)\right| = \matr{n+\floor{\frac{\ell-1}{2}}}{n}+\matr{n+\floor{\frac{\ell}{2}}-1}{n}.
\end{align*}
Hence, for $n \ge 3$, $\ell \ge 8$ and $\ell \equiv_2 0$, we have
\begin{align*}
\left|\mx_{B_n^{(1)}}^+((\ell-1)\Lambda_0+\Lambda_n)\right| = \left|\mx_{B_{\ell/2-1}^{(1)}}^+((2n+1)\Lambda_0+\Lambda_{\ell/2 -1 })\right|,
\end{align*}
i.e., when we exchange $n$ with $\ell/2-1$, the number of dominant maximal weights remains same.
\medskip

\noindent{\bf $C_{n}^{(1)}$ type.} From Table~\ref{table: num of dom wts_binom} it follows that for any $\Lambda = (\ell-1)\Lambda_0 + \Lambda_i \in \DRpclp$,
\begin{align*}
\left|\mx_{C_n^{(1)}}^+(\Lambda)\right| = \dfrac{1}{2} \left( \matr{\ell+n}{n} + (-1)^i  \delta(n \ell\equiv_2 0) \matr{ \lfloor \frac{\ell+n}{2} \rfloor }{\lfloor \frac{n}{2} \rfloor}  \right).
\end{align*}
Hence, for $n\ge 2$ and $\ell \ge 2$, we have
\begin{align*}
\left|\mx_{C_n^{(1)}}^+((\ell-1)\Lambda_0 + \Lambda_i)\right| = \left|\mx_{C_\ell^{(1)}}^+((n-1)\Lambda_0 + \Lambda_i)\right| \quad \text{for $i = 0,1$},
\end{align*}
i.e., when we exchange $n$ with $\ell$, the number of dominant maximal weights remains same.
\medskip

\noindent{\bf $D_{n}^{(1)}$ type.}
In case where $\ell \equiv_2 0$, from Table~\ref{table: num of dom wts}, it follows that for $i = n-1,n$
\begin{align*}
\left|\mx_{D_{n}^{(1)}}^+((\ell-1)\Lambda_0 + \Lambda_i)\right| = \frac{1}{4}\left(\sfm(1^{4}, 2^{n-3}) - \sfm(2^{n-1}) \right).
\end{align*}
Using Lemma~\ref{lem: calnum}, one can see that
\[\sfm(1^{4}, 2^{n-3}) = \matr{n + \frac{\ell}{2} - 2}{\frac{\ell}{2}} + 8\matr{n + \frac{\ell}{2} - 1}{\frac{\ell}{2}-1} \quad \text{and} \quad \sfm(2^{n-1}) = \matr{n + \frac{\ell}{2} - 2}{\frac{\ell}{2}}\]
and thus
\[\left|\mx_{D_{n}^{(1)}}^+((\ell-1)\Lambda_0 + \Lambda_i)\right| = 2\matr{n + \frac{\ell}{2} - 1}{\frac{\ell}{2}-1}.\]
Hence, for $n \ge 4$, $\ell \ge 9$ and $\ell \equiv_2 0$, we have
\begin{align*}
\left|\mx_{D_{n}^{(1)}}^+((\ell-1)\Lambda_0 + \Lambda_{n})\right| = \left|\mx_{D_{\ell/2 - 1}^{(2)}}^+((2n+1)\Lambda_0 + \Lambda_{\ell/2 - 1})\right|,
\end{align*}
i.e., for an even integer $\ell$, when we exchange $n$ with $\ell/2 - 1$, the number of dominant maximal weights remains same.
\medskip 

\noindent{\bf $A_{2n-1}^{(2)}$ type.} 
In case where $\ell \equiv_2 1$, from Table~\ref{table: num of dom wts}, it follows that for any $\Lambda = (\ell-1)\Lambda_0 + \Lambda_i \in \DRpclp$,
\begin{align*}
\left|\mx_{A_{2n-1}^{(2)}}^+(\Lambda)\right| = \frac{ \sfm(1^2, 2^{n-1})}{2}  = \frac{1}{2}\left(\matr{n + \floor{\frac{\ell-1}{2}}}{\floor{\frac{\ell-1}{2}}} + \matr{n + \floor{\frac{\ell}{2}}}{\floor{\frac{\ell}{2}}}\right).
\end{align*}
Hence, for $n \ge 3$, $\ell \ge 7$ and $\ell \equiv_2 1$, we have
\begin{align*}
\left|\mx_{A_{2n-1}^{(2)}}^+((\ell-1)\Lambda_0 + \Lambda_i)\right| = \left|\mx_{A_{2((\ell-1)/2)}^{(2)}}^+(2n\Lambda_0 + \Lambda_i)\right| \quad \text{for $i=0,1$},
\end{align*}
i.e., for an odd integer $\ell$, when we exchange $n$ with $(\ell-1)/2$, the number of dominant maximal weights remains same.
\medskip 

\noindent{\bf $A_{2n}^{(2)}$ type.} From Table~\ref{table: num of dom wts_binom} it follows that
\begin{align*}
\left|\mx_{A_{2n}^{(2)}}^+(\ell\Lambda_0)\right| = \matr{\floor{\frac{\ell}{2}}+n}{n}.
\end{align*}
Hence, for $n \ge 2$, $\ell \ge 4$ and $\ell \equiv_2 0$ (resp. $\ell \equiv_2 1$), we have
\begin{align*}
\left|\mx_{A_{2n}^{(2)}}^+(\ell\Lambda_0)\right| = \left|\mx_{A_{2(\ell/2)}^{(2)}}^+(2n\Lambda_0)\right| \quad \left(\text{resp. } \left|\mx_{A_{2n}^{(2)}}^+(\ell\Lambda_0)\right| = \left|\mx_{A_{2((\ell-1)/2)}^{(2)}}^+((2n+1)\Lambda_0)\right| ~\right),
\end{align*}
i.e., for an even (resp. odd) integer $\ell$, when we exchange $n$ with $\ell/2$ (resp. $(\ell-1)/2$), the number of dominant maximal weights remains same.
%
\medskip

\noindent{\bf $D_{n+1}^{(2)}$ type.} From Table~\ref{table: num of dom wts_binom} it follows that
\begin{align*}
\left|\mx_{D_{n+1}^{(2)}}^+(\ell\Lambda_0)\right| = \matr{\floor{\frac{\ell}{2}}+n}{n}.
\end{align*}
Hence, for $n \ge 2$, $\ell \ge 4$ and $\ell \equiv_2 0$ (resp. $\ell \equiv_2 1$), we have
\begin{align*}
\left|\mx_{D_{n+1}^{(2)}}^+(\ell\Lambda_0)\right| = \left|\mx_{D_{\ell/2}^{(2)}}^+(2n\Lambda_0)\right| \quad \left(\text{resp. } \left|\mx_{D_{n+1}^{(2)}}^+(\ell\Lambda_0)\right| = \left|\mx_{D_{(\ell-1)/2}^{(2)}}^+((2n+1)\Lambda_0)\right|~ \right),
\end{align*}
i.e., for an even (resp. odd) integer $\ell$, when we exchange $n$ with $\ell/2$ (resp. $(\ell-1)/2$), the number of dominant maximal weights remains same.

From Table~\ref{table: num of dom wts_binom} it follows that
\begin{align*}
\left|\mx_{D_{n+1}^{(2)}}^+((\ell-1)\Lambda_0 + \Lambda_n)\right| = \matr{\floor{\frac{\ell-1}{2}}+n}{n}.
\end{align*}
Hence, for $n \ge 2$, $\ell \ge 5$ and $\ell \equiv_2 0$ (resp. $\ell \equiv_2 1$),
\begin{align*}
&\left|\mx_{D_{n+1}^{(2)}}^+((\ell-1)\Lambda_0+\Lambda_n)\right| = \left|\mx_{D_{\ell/2-1}^{(2)}}^+((2n+1)\Lambda_0+\Lambda_{\ell/2 -1 })\right|\\
&\hspace{20ex}\left(\text{resp. } \left|\mx_{D_{n+1}^{(2)}}^+((\ell-1)\Lambda_0+\Lambda_n)\right| = \left|\mx_{D_{(\ell-1)/2}^{(2)}}^+(2n\Lambda_0+\Lambda_{(\ell-1)/2})\right|~\right),
\end{align*}
i.e., for an even (resp. odd) integer $\ell$, when we exchange $n$ with $\ell/2-1$ (resp. $(\ell-1)/2$), the number of dominant maximal weights remains same.


\begin{thebibliography}{99}

\bibitem{ARR} D. Armstrong,  V. Reiner,  B. Rhoades 
\newblock{\it Parking spaces},  
\newblock{Adv. Math. 269 (2015), 647–706. }

\bibitem{BRS} H. Barcelo, V. Reiner, D.Stanton,
\newblock{\it Bimahonian distribution,}
\newblock{J. Lond. Math. Soc. (2), {\bf 77} (2008), 627-646.}

\bibitem{BFS} O. Barshevsky, M. Fayers, M.Schaps,
\newblock{\it A non-recursive criterion for weights of a highest-weight module for an affine {L}ie algebra},
\newblock{Israel J. Math., {\bf 197} (2013), 237-261.}

\bibitem{C} S. E. Crifo,
\newblock{\it Some maximal dominant weights and their multiplicities for affine Lie algebra representations},
\newblock{Thesis (Ph.D.)-North Carolina State University, 2019.}

\bibitem{EJP} A. Elashvili, M. Jibladze, D. Pataraia,
\newblock{\it Combinatorics of necklaces and ``Hermite reciprocity'',}
\newblock{J. Algebraic Combin., {\bf 10} (1999), 173-188.}

\bibitem{EF}
S.-P. Eu, T.-S. Fu, 
\newblock{\it The cyclic sieving phenomenon for faces of generalized cluster complexes},
\newblock{Adv. in Appl. Math. {\bf 40}, 3 (2008), 350–376.}

\bibitem{F} I. B. Frenkel,
\newblock{\it Representations of affine Lie algebras, Hecke modular forms and Korteweg-de Vries type equations},
\newblock{Lie algebras and related topics (New Brunswick, N.J., 1981), 71-110, Lecture Notes in Math., {\bf 933}, Springer, Berlin-New York, 1982.}

\bibitem{HK} J. Hong, S.-J. Kang,
\newblock{\it Introduction to quantum groups and crystal bases,}
\newblock{Graduate Studies in Mathematics, {\bf 42}, American Mathematical Society, Providence, RI, 2002.}

\bibitem{Humph} J. Humphreys,
\newblock{\it Introduction to Lie algebras and representation theory},
\newblock{Graduate Texts in Mathematics, {\bf 9}, Springer-Verlag, New York-Berlin, 1978.}

\bibitem{JM} R. L. Jayne, K. C. Misra,
\newblock{\it On multiplicities of maximal weights of $\hat{sl}(n)$-modules},
\newblock{Algebr. Represent. Theory, {\bf 17} (2014), 1303-1321.}


\bibitem{Kac0} V. G. Kac,
\newblock{\it Simple irreducible graded Lie algebras of finite growth},
\newblock{Izv. Akad. Nauk SSSR Ser. Mat., {\bf 32} (1968), 1323-1367.}

\bibitem{Kac} \bysame,
\newblock{\it Infinite dimensional Lie algebras},
\newblock{3rd ed., Cambridge University Press, Cambridge, 1990.}


\bibitem{KLO} J. S. Kim, K.-H. Lee, S.-j. Oh,
\newblock{\it Weight multiplicities and Young tableaux through affine crystals},
\newblock{S\'{e}m. Lothar. Combin. {\bf 78B} (2017), Art. 25, 12 pp.}

\bibitem{M} R. V. Moody,
\newblock{\it A new class of Lie algebras},
\newblock{J. Algebra,  {\bf 10} (1968), 211-230.}

\bibitem{RSW} V. Reiner, D. Stanton, D. White,
\newblock{\it The cyclic sieving phenomenon},
\newblock{J. Combin. Theory Ser. A, {\bf 108} (2004), 17-50.}

\bibitem{Reu} C. Reutenauer,
\newblock{\it Free Lie algebras},
\newblock{Clarendon Press, Oxford, 1993.}

\bibitem{R} B. Rhoades, 
\newblock{\it Cyclic sieving, promotion, and representation theory},
\newblock{J. Combin. Theory, Ser. A {\bf 117} (2010), no. 1, 38-76}.

\bibitem{Sagan1}
B. Sagan, {\it Congruence properties of q-analogs}, Adv. Math, {\bf 95} (1992), 127-143.

\bibitem{Sagan2}
\bysame, {\it The cyclic sieving phenomenon: a survey}, Surveys in combinatorics, (2011), 183-233, London Math. Soc. Lecture Note Ser., {\bf 392}, Cambridege Univ. Press, Cambridge, 2011.

\bibitem{OEIS} N. J. A. Sloane,
\newblock{\it The On-line encyclopedia of integer sequences},
\newblock{published electronically at https://oeis.org.}


\bibitem{EC1}
R. P. Stanley, {\it Enumerative combinatorics. Vol. 1. Second edition}, Cambridge Studies in Advanced Mathematics, {\bf 62}, Cambridge University Press, Cambridge, 1999.

\bibitem{T} P. Tingley,
\newblock{\it Three combinatorial models for $\widehat{\rm sl}_n$ crystals, with applications to cylindric plane partitions},
\newblock{Int. Math. Res. Not. IMRN, {\bf 143} (2008), 1-40.}

\bibitem{Tsu} S. Tsuchioka,
\newblock{\it Catalan numbers and level 2 weight structures of $A_{p-1}^{(1)}$},
\newblock{RIMS K\^{o}ky\^{u}roku Bessatsu, {\bf B11} (2009), Res. Inst. Math. Sci. (RIMS), Kyoto, 145-154}.

\bibitem{TW} S. Tsuchioka, M. Watanabe,
\newblock{\it Pattern avoidances seen in multiplicities of maximal weights of affine Lie algebra representations,}
\newblock{Proc. Amer. Math. Soc., {\bf 146} (2018), 15-28.}

\end{thebibliography}
\end{document}